\documentclass[12pt]{article}

\usepackage{epsfig,epsfig}
\usepackage{epstopdf}
\usepackage{amsmath}
\usepackage{mathrsfs}
\usepackage{amsfonts}
\usepackage{bm}
\usepackage{amssymb}
\usepackage{booktabs}
\usepackage{color}
\usepackage{multirow}
\usepackage{graphicx}
\usepackage{subfigure}
\usepackage{float}
\usepackage{appendix}
\usepackage{tikz}
\usetikzlibrary{calc}
\usetikzlibrary{matrix}
\usepackage{mathtools}
\usepackage{algorithm,algpseudocode,float}
\makeatletter
\newenvironment{breakablealgorithm}
  {
   \begin{center}
     \refstepcounter{algorithm}
     \hrule height.8pt depth0pt \kern2pt
     \renewcommand{\caption}[2][\relax]{
       {\raggedright\textbf{\ALG@name~\thealgorithm} ##2\par}%
       \ifx\relax##1\relax 
         \addcontentsline{loa}{algorithm}{\protect\numberline{\thealgorithm}##2}%
       \else 
         \addcontentsline{loa}{algorithm}{\protect\numberline{\thealgorithm}##1}%
       \fi
       \kern2pt\hrule\kern2pt
     }
  }{
     \kern2pt\hrule\relax
   \end{center}
  }
\makeatother

\usepackage{srcltx,graphicx}

\newtheorem{thm}{Theorem}[section]
\newtheorem{prop}[thm]{Proposition}

\newtheorem{proof}[thm]{Proof}

\newtheorem{rem}[thm]{Remark}
\newtheorem{example}[thm]{Example}

\newcommand{\beq}{\begin{equation}}
\newcommand{\eeq}{\end{equation}}

\numberwithin{equation}{section} \topmargin=-2.0cm \oddsidemargin=1cm
\definecolor{ngreen}{rgb}{0.16, 0.5, 0.0}
\begin{document}

\title{A study on CFL conditions for the DG solution of conservation laws on adaptive moving meshes}
\author{Min Zhang\footnote{School of Mathematical Sciences, Peking University, Beijing 100871, China.
E-mail: minzhang@math.pku.edu.cn.
},
~Weizhang Huang\footnote{Department of Mathematics, University of Kansas, Lawrence, Kansas 66045, USA. E-mail: whuang@ku.edu.
},
~and Jianxian Qiu\footnote{School of Mathematical Sciences and Fujian Provincial Key Laboratory of Mathematical Modeling and High-Performance Scientific Computing, Xiamen University, Xiamen, Fujian 361005, China.
E-mail: jxqiu@xmu.edu.cn.
}
}

\date{}
\maketitle
\begin{abstract}
The selection of time step plays a crucial role in improving stability and efficiency in the Discontinuous Galerkin (DG) solution of hyperbolic conservation laws on adaptive moving meshes that typically employs explicit stepping.
A commonly used selection of time step is a direct extension based on Courant-Friedrichs-Levy (CFL) conditions established for fixed and uniform meshes.
In this work, we provide a mathematical justification for those time step selection strategies used in practical adaptive DG computations.
A stability analysis is presented for a moving mesh DG method for linear scalar conservation laws.  Based on the analysis, a new selection strategy of the time step is proposed, which takes into consideration the coupling of the $\alpha$-function (that is related to the eigenvalues of the Jacobian matrix of the flux and the mesh movement velocity) and the heights of the mesh elements. The analysis also suggests several stable combinations of the choices of the $\alpha$-function in the numerical scheme and in the time step selection. Numerical results obtained with a moving mesh DG method for Burgers' and Euler equations are presented.
For comparison purpose, numerical results obtained with an error-based time step-size selection strategy are also given.
\end{abstract}

\vspace{5pt}

\noindent\textbf{The 2020 Mathematics Subject Classification:} 65M50, 65M60

\vspace{5pt}


\newcommand{\h}{\hspace{1.cm}}
\newcommand{\hh}{\hspace{2.cm}}
\newtheorem{yl}{\hspace{1.cm}Lemma}
\newtheorem{dl}{\hspace{1.cm}Theorem}
\renewcommand{\sec}{\section*}
\renewcommand{\l}{\langle}
\renewcommand{\r}{\rangle}
\newcommand{\be}{\begin{eqnarray}}
\newcommand{\ee}{\end{eqnarray}}

\normalsize \vskip 0.2in
\newpage

\section{Introduction}

We are concerned with the stability of the discontinuous Galerkin (DG) solution of conservation laws
on adaptive moving meshes.
The DG method is a powerful numerical tool for use in the simulation of hyperbolic problems.
It was first used by Reed and Hill \cite{Reed-Hill-1973} for the steady radiation transport equation and studied theoretically by Lesaint and Raviart \cite{LeSaint-Raviart-1974}.
The method was extended to conservation laws by Cockburn and Shu (and their coworkers) in a series of papers \cite{DG-series0,DG-series1,DG-series2,DG-series3,DG-series4,DG-series5}.
The DG method has the advantages of high-order accuracy, geometric flexibility, easy use with mesh adaptivity, local data structure,
high parallel efficiency, and a good foundation for theoretical analysis. The DG method
has been used widely in scientific and engineering computation.
Meanwhile, conservation laws typically exhibit discontinuous structures such as shock waves,
rarefaction waves, and contact discontinuities and are amenable to mesh adaptation in their numerical solution to enhance numerical resolution and computational efficiency.
It is natural to combine the DG method with mesh adaptation method for the solution
of conservation laws.

A large amount of work has been done in this area.
For example, Bey and Oden \cite{Bey-Oden-1996} combined the $hp$-method with the DG method
for conservation laws and Li and Tang \cite{Li-Tang-2006} solved two-dimensional conservation laws
using a rezoning moving mesh DG method where the physical variables are interpolated from the old mesh
to the new one using conservative interpolation schemes.
Mackenzie and Nicola \cite{Mackenzie-Nicola-2007} solved the Halmiton-Jacobi equation by the DG method
using a moving mesh method based on the moving mesh partial differential equation (MMPDE)
strategy \cite{Huang-Ren-Russell-1994a,Huang-Russell-2011}.
Vilar et al. \cite{Vilar-Maire-Abgrall-Lagrangian-2014} studied a DG discretization for solving
the two-dimensional gas dynamics equations in Lagrangian formulation.
More recently, Uzunca et al. \cite{Uzunca-2017} employed a moving mesh symmetric interior penalty Galerkin
method (SIPG) to solve PDEs with traveling waves. Luo et al. considered a quasi-Lagrange moving mesh DG method (MMDG) for conservation laws \cite{Luo-Huang-Qiu-2019JCP} and multi-component flows \cite{Luo-LHQC-2021}.
Zhang et al. studied the MMDG solution
for the radiative transfer equation \cite{Zhang-Cheng-Huang-Qiu-2020CiCP,Zhang-Huang-Qiu-2020SISC} and shallow water equations (SWEs) \cite{Zhang-Huang-Qiu-2021JSC,Zhang-Huang-Qiu-2022CICP}.
Zhang et al. \cite{Zhang-Xia-Xu-2021JSC} develop a arbitrary Lagrangian-Eulerian discontinuous Galerkin (ALE-DG) methods for the SWEs.
Wang et al. \cite{Wang-Luo-Shashkov-Lagrangian-2020} developed a reconstructed DG Method
for compressible flows in Lagrangian formulation.

In principle, any marching scheme (e.g., see Hairer and Wanner \cite{HW1991})
can be used for the time integration of DG computations of hyperbolic conservation laws,
including explicit and implicit Runge-Kutta methods \cite{SSP-time,SSPERK-2018} and
multi-step methods \cite{Shu-1998-TVDtime}. Nevertheless, explicit schemes have been widely used
in these computations. There are at least two considerations for this.
First, as we can see later, the stability condition
for explicit schemes when applied to hyperbolic equations typically requires the time step-size
to be proportional to the minimum mesh element size, which is considered acceptable in practical computations
with a uniform mesh. Second, due to the highly nonlinear and hyperbolic nature of conservation laws,
there exists hardly any efficient solver for nonlinear algebraic systems (whose linearization is typically non-symmetric
and non-definite) resulting from the implicit temporal discretization.
As such, it does not seem worth the trouble to increase the time step-size using implicit schemes when a uniform mesh is used.
However, this can be a different situation when an adaptive mesh is employed where
some mesh elements can become very small. While implicit schemes for DG computations on adaptive meshes
deserve further investigations, in this work we focus on explicit schemes and their stability on adaptive moving meshes.

Consider the conservation laws in the form
\begin{equation}
\label{PDE-1}
U_t +  \nabla \cdot \mathbf{F}(U,\bm{x}) = 0, \quad \forall \bm{x} \in \Omega
\end{equation}
where $\Omega$ is a polygonal/polyhedral domain in $\mathbb{R}^d$ ($d \ge 1$), $U = (u_1, ..., u_m)^T$ ($m\ge 1$)
is the unknown function, and the flux $\mathbf{F}(U,\bm{x})$ is an $m$-by-$d$ matrix-valued function of $U$ and $\bm{x}$.
A commonly used selection of time step in adaptive DG computations (e.g., see \cite{Luo-Huang-Qiu-2019JCP,Zhang-Huang-Qiu-2021JSC,Zhang-Huang-Qiu-2022CICP}) is
\begin{equation}
\Delta t =\frac{C_{cfl}\;\sigma_{h,min}}{\alpha_{h}},
\label{cfl-old-2}
\end{equation}
where $\sigma_{h,min}$ is the minimum height or diameter of the mesh elements,
$C_{cfl}$ is a positive parameter, $\alpha_{h}$ is the maximum absolute
value of the eigenvalues of the Jacobian matrix of $(\mathbf{F} - U \dot{\bm{X}})\cdot \bm{n}$ (with respect to $U$) taken over all of the edges, $\bm{n}$ is the unit outward normal vector of the edges, and $\dot{\bm{X}}$ is the piecewise linear mesh velocity function.
The choice (\ref{cfl-old-2}) is a direct extension of the CFL conditions studied and
used in DG computation for fixed meshes (e.g., see \cite{DG-series2,DG-review}).
Some researchers have used a different yet mathematical equivalent form,
\begin{equation}
\Delta t = \frac{C_{cfl}}{\alpha_{h} \max\limits_{K \in \mathcal{T}_h}
\Big (\frac{1}{|K|} \sum\limits_{e \in \partial K}  |e| \Big )},
\label{cfl-old-1}
\end{equation}
where $K$ is an element of the mesh $\mathcal{T}_h$, $\partial K$ is the boundary of $K$,
$e$ is an edge of $K$, and $|K|$ and $|e|$ denote the area of $K$
and the length of $e$, respectively.
This condition has also been established by Zhang et al. \cite{Zhang-Xia-Shu-2012JSC}
for positivity preservation for fixed unstructured triangular meshes.
To show the equivalence between (\ref{cfl-old-2}) and (\ref{cfl-old-1}), we notice that, for a simplex $K$,
\[
\max\limits_{e \in \partial K} |e|  \le \sum\limits_{e \in \partial K} |e|  \le (d+1) \max\limits_{e \in \partial K} |e| ,
\]
which implies
\[
\frac{d}{\sigma_{K,min}} \le \frac{1}{|K|} \sum\limits_{e \in \partial K} |e|  \le \frac{d (d+1)}{\sigma_{K,min}} ,
\]
where $\sigma_{K,min}$ denotes the minimum height of $K$ and we have used a geometric property of simplexes,
$|K| = (1/d) \sigma_{K,min} \max\limits_{e \in \partial K} |e|$. Taking maximum over all elements, we get
\begin{equation}
\label{cfl-old-3}
 \frac{\sigma_{h,min}}{d (d+1)} \le
\frac{1}{\max\limits_{K}\Big ( \frac{1}{|K|} \sum\limits_{e \in \partial K} |e| \Big ) }
\le \frac{\sigma_{h,min}}{d} ,
\end{equation}
which gives the equivalence of (\ref{cfl-old-2}) and (\ref{cfl-old-1}).

While the condition (\ref{cfl-old-2}) or (\ref{cfl-old-1}) appears to work well in existing adaptive DG computation for conservation laws, there lacks a theoretical justification of them for non-uniform and moving meshes.
One may also wonder if the coupling between physical quantities and mesh elements
can be taken into consideration for time step selection. The objective of this work is to study these issues.
To be specific, we consider a quasi-Lagrange MMDG method
\cite{Luo-Huang-Qiu-2019JCP, Zhang-Cheng-Huang-Qiu-2020CiCP}
with the Lax-Friedrichs (LF) flux for (\ref{PDE-1}). A CFL condition for the $L^1$ stability of the MMDG method
with $P^0$ elements is then established and analyzed for scalar linear equations (cf. Proposition~\ref{prop-p0-cfl}),
which provides a theoretical justification for the stability of the method. Moreover,
based on this analysis, for the MMDG method with $P^k$ elements ($k \geq1$) and general conservation laws
we propose to choose $\Delta t$ as
\begin{equation}
\Delta t = \frac{C_{cfl}}{ \max\limits_{K \in \mathcal{T}_h}
\Big (\frac{1}{|K|} \sum\limits_{e \in \partial K}  \alpha_e |e| \Big )}
= \frac{C_{cfl}}{ \max\limits_{K \in \mathcal{T}_h}
\Big ( \sum\limits_{e \in \partial K}  \alpha_e \frac{|e|}{|K|} \Big )} ,
\label{cfl-new-1}
\end{equation}
where $\alpha_e=\alpha_e(U,\bm{\dot{X}})$ is the maximum absolute
value of the eigenvalues of the Jacobian matrix of $(\mathbf{F} - U \dot{\bm{X}})\cdot \bm{n}$ (with respect to $U$) taken over edge $e$. This choice is very similar to (\ref{cfl-old-1}).
Indeed, it reduces to (\ref{cfl-old-1}) when $\alpha_e$ is replaced
with the global $\alpha_h$. However, unlike (\ref{cfl-old-1}), the condition (\ref{cfl-new-1}) takes into consideration the spatial variation of $\alpha$ and its coupling with the element height (i.e., $|K|/|e|$).
Moreover, it shows that $\Delta t$ can be increased if the mesh velocity $\bm{\dot{X}}$ can be chosen to minimize $\alpha_e$ in regions where the element height is relatively small.
An example of this is Lagrangian-type methods (e.g.~\cite{Hirt-ALE1-1971,Hirt-ALE2-1974,Kucharik-Shashkov-ALE-2014}) where the mesh velocity is taken as the fluid velocity.
On the other hand, the mesh movement can be determined by other considerations. For example, in the current work
we consider the use of the MMPDE moving mesh method to determine the mesh movement based on solution-Hessian.
In this case, the mesh velocity does not necessarily minimize $\alpha_e$.
If $\alpha_e$ does not change significantly over the domain, then (\ref{cfl-new-1}) is mathematically equivalent to (\ref{cfl-old-2}) and $\Delta t$ is determined essentially by the minimum height of the mesh elements.

The CFL condition such as (\ref{cfl-new-1}) provides a selection strategy for time step-size based on stability. This
CFL-condition-based strategy is widely used in the DG computation of hyperbolic conservation laws.
On the other hand, it is common practice to use an error-based time step-size selection strategy in the computation of ordinary
differential equations (e.g., see Hairer and Wanner \cite{HW1991}).
It is worth studying how error-based time step-size selection strategies fare in the DG computation of hyperbolic conservation laws.
To this end, we use the third-order explicit Strong Stability Preserving (SSP) embedded RK (SSP-ERK(4,3)) pair
of \cite{SSPERK-2018} and the standard
PI controller step-size selection strategy (e.g., see Hairer and Wanner \cite{HW1991}) and present numerical examples
to demonstrate the feasibility of this approach.

The MMPDE moving mesh method \cite{Huang-Ren-Russell-1994a,Huang-Russell-2011} is used to generate
adaptive moving meshes for the numerical examples presented in this work.
A key idea of the MMPDE method is to view any nonuniform mesh as a uniform one
in some Riemannian metric specified by a tensor $\mathbb{M} = \mathbb{M}(\bm{x})$,
a symmetric and uniformly positive definite matrix-valued function that provides the information needed for determining the size, shape, and orientation of the mesh elements throughout the domain.
It has been shown analytically and numerically in \cite{Huang-Kamenski-2018MC}
that the moving mesh generated by the MMPDE method stays nonsingular (free of tangling)
if the metric tensor is bounded and the initial mesh is nonsingular. It is worth pointing out that other adaptive moving mesh methods
(such as Lagrangian-type methods) can also be used; e.g., see
\cite{Bai94a,Baines-2011,Bey-Oden-1996,BHR09,Dumbser-Boscheri-Lagrangian-2013,
Hirt-ALE1-1971,Hirt-ALE2-1974,Huang-Russell-2011,Kucharik-Shashkov-ALE-2014,Morgan-etal-Lagrangian-2014,Tan05,
Vilar-Maire-Abgrall-Lagrangian-2014,Wang-Luo-Shashkov-Lagrangian-2020} and references therein.

An outline of the paper is as follows. The MMDG method is described in Section~\ref{SEC:MMDG} and
the $L^1$ stability analysis of the method with $P^0$ elements is carried out for scalar linear equations
in Section~\ref{SEC:CFL}.
Numerical examples are presented in Section~\ref{SEC:numerics}. In these examples,
the MMPDE moving mesh method is used to generate
adaptive moving meshes. The section also contains the descriptions of the procedure of the MMDG method and
the definition of the metric tensor that is used to control mesh concentration.
The conclusions are given in the final section \ref{SEC:conclusions}.

\section{The moving mesh DG method}
\label{SEC:MMDG}

In this section we describe a quasi-Lagrange MMDG method \cite{Luo-Huang-Qiu-2019JCP, Zhang-Cheng-Huang-Qiu-2020CiCP,Zhang-Huang-Qiu-2022CICP} for solving hyperbolic conservation laws in the form
(\ref{PDE-1}).

To start with, we assume that a sequence of simplicial meshes, $\mathcal{T}^0_h,\,\mathcal{T}^1_h,\, ...$, have been given for $\Omega$ at time instants $t_0,\, t_1,\, ...$ and these meshes have the same number of elements and vertices and the same connectivity.
For numerical results presented in this work, we use the MMPDE moving mesh method
\cite{Huang-Ren-Russell-1994a,Huang-Russell-2011, Huang-Kamenski-2015JCP}
to generate these meshes (cf. Section~\ref{SEC:numerics}).
For any $n \ge 0$ and $t\in[t_n, t_{n+1}]$, we define $\mathcal{T}_h(t)$ as the mesh with the same number
of elements ($N$) and vertices ($N_v$) and the same connectivity as $\mathcal{T}_h^n$, and having
the vertices given by
\begin{equation}\label{location+speed}
\begin{split}
\bm{x}_i(t)= \frac{t-t_n}{\Delta t_n}\bm{x}_i^{n}+\frac{t_{n+1}-t}{\Delta t_n}\bm{x}_i^{n+1},
\quad i = 1,...,N_v ,\quad \Delta t_n = t_{n+1}-t_n.
\end{split}
\end{equation}
Define the piecewise linear mesh velocity function as
\begin{equation}
\label{Xdot-1}
\dot{\bm{X}}(\bm{x},t)= \sum_{i=1}^{N_v} \dot{\bm{x}}_i \phi_i(\bm{x},t)
= \sum_{i=1}^{N_v} \frac{\bm{x}_i^{n+1}-\bm{x}_i^{n}}{\Delta t_n} \phi_i(\bm{x},t),\quad t\in[t_n,t_{n+1}]
\end{equation}
where $ \phi_i(\bm{x},t)$ is the linear basis function at ${\bm{x}}_i$ and $\dot{\bm{x}}_i$ is the nodal velocity.
For any element $K \in \mathcal{T}_h(t)$,
let $P^{k}(K)$ be the set of polynomials of degree at most $k\ge 1$ on $K$.
The DG finite element space is defined as
\begin{equation}\label{Vh}
\mathcal{V}^{k}_h(t)= \{u\in L^2(\Omega):\; u|_{K}\in P^{k}(K),
\; \forall K\in \mathcal{T}_h(t) \}.
\end{equation}

We now are ready to describe the DG discretization of \eqref{PDE-1}.
Multiplying it with an arbitrary test function $\phi\in \mathcal{V}^{k}_h(t)$,
integrating the resulting equation over $K$, and using the Reynolds transport theorem, we get
\begin{equation}\label{DG-0-1}
\frac{d}{d t}\int_{K}U_h \phi d\bm{x} -  \int_{K} \mathbf{H}(U_h,\bm{x}) \cdot \nabla \phi d\bm{x}
+ \sum\limits_{e \in \partial K}\int_{e}\phi \hat{\mathbf{H}}(U_h,\bm{x}) d s= 0,
\end{equation}
where $U_h = U_h(\bm{x}, t)$, $\mathbf{H}(U_h,\bm{x}) = \mathbf{F}(U_h,\bm{x}) - U_h \dot{\bm{X}}$ is the new flux
accounting for the effect of mesh movement,
$\hat{\mathbf{H}}(U_h,\bm{x}) \approx \mathbf{H}(U_h,\bm{x})\cdot \bm{n}$ is a numerical flux, and
$\bm{n}$ is the unit outward normal to edge $e$.
Let $\lambda_{m}$'s be the eigenvalues of the Jacobian matrix of $\mathbf{H}(U_h,\bm{x})\cdot \bm{n}$ with respect to $U_h$ and
$U_{h,K}^{int}$ and $U_{h,K}^{ext}$ be the values of $U_h$ in $K$ and $K'$, respectively, where $K'$ is
the element sharing $e$ with $K$.
Define the $\alpha$-function as
\begin{equation}\label{alpha-f}
\alpha(U_h,\bm{x}) = \max\limits_m \Big (\max\big(
\big|\lambda_{m}(U_{h,K}^{int},\bm{x})\big|, \;
\big|\lambda_{m}(U_{h,K}^{ext},\bm{x})\big|\big)\Big{)}.
\end{equation}
In this work we consider the global/local Lax-Friedrichs (LF) numerical flux,
\begin{align*}
& \hat{ \mathbf{H}}(U_h,\bm{x}) =
\frac{1}{2}\left (
\big{(}\mathbf{H}(U_{h,K}^{int},\bm{x})+\mathbf{H}(U_{h,K}^{ext},\bm{x}) \big{)} \cdot \bm{n}
-\alpha_{LF} (U_{h,K}^{ext}-U_{h,K}^{int})\right ),
\quad \bm{x} \in e \subset \partial K
\end{align*}
where $\alpha_{LF} $ denotes a choice of the $\alpha$-function in this numerical flux.
In practical computation, the second and third terms in \eqref{DG-0-1} are calculated
typically by Gaussian quadrature rules. We denote those by
\begin{equation}
\label{Guass-quad}
\begin{split}
&\int_{K} \mathbf{H}(U_h,\bm{x}) \cdot \nabla \phi d\bm{x}
\approx |K|\sum\limits_{\bm{x}_{G}^K} w_G^K \big( \mathbf{H}(U_h,\bm{x}) \cdot \nabla \phi \big) \big|_{\bm{x}_{G}^K},\quad \hbox{with}\quad \sum w_G^K =1
\\&
\int_{e}\phi \hat{\mathbf{H}}(U_h,\bm{x}) d s \approx
 |e|\Big(\sum\limits_{\bm{x}_{G}^e}w_G^e \big( \phi \hat{\mathbf{H}}(U_h,\bm{x})\big) \big|_{\bm{x}_{G}^e}\Big), \quad \hbox{with}\quad \sum w_G^e =1
\end{split}
\end{equation}
where $\bm{x}_{G}^K$'s and $\bm{x}_{G}^e$'s are the Gauss points on $K$ and $e$, respectively.
For the analytical analysis in the next section, we assume that the weights $w_G^K$'s and $w_G^e$'s are nonnegative.
Combining the above with \eqref{DG-0-1} we obtain the semi-discrete MMDG scheme as
\begin{align}
\label{DG-0}
\frac{d}{d t}\int_{K}U_h \phi d\bm{x}
& -|K|\sum\limits_{\bm{x}_{G}^K} w_G^K \big( \mathbf{H}(U_h,\bm{x}) \cdot \nabla \phi \big) \big|_{\bm{x}_{G}^K}
\\
& + \sum\limits_{e \in \partial K}|e|\Big(\sum\limits_{\bm{x}_{G}^e}w_G^e
\big( \phi \hat{\mathbf{H}}(U_h,\bm{x})\big) \big|_{\bm{x}_{G}^e}\Big) = 0.
\notag
\end{align}
One choice of $\alpha_{LF}$ is the point-wise value of $\alpha (U_h,\bm{x})$, i.e.,
\begin{equation}
\alpha_{p} = \alpha(U_h,\bm{x})|_{\bm{x}_{G}^e} = \max\limits_{m} \Big( \max\limits \big (
\big|\lambda_{m}(U_{h,K}^{int},\bm{x})|_{\bm{x}_{G}^e} \big|, \; \big|\lambda_{m}(U_{h,K}^{ext},\bm{x})|_{\bm{x}_{G}^e} \big|\big) \Big ), ~\bm{x}_G^e \in e\subset\partial K.
\label{alpha-p}
\end{equation}
It is worth pointing out that \eqref{alpha-p} is calculated pointwise.
We can choose it differently, for example, by taking the maximum value over each edge (denoted as $\alpha_{e}$),
or all edges in the mesh (denoted as $\alpha_{h}$), i.e.,
\begin{align}
&\alpha_{e} = \max\limits_{\bm{x}_G^e}  \alpha(U_h,\bm{x})|_{\bm{x}_{G}^e} =
\max\limits_{\bm{x}_G^e,m} \Big ( \max\big (
\big|\lambda^{m}(U_{h,K}^{int},\bm{x})|_{\bm{x}_{G}^e} \big|,
\; \big|\lambda^{m}(U_{h,K}^{ext},\bm{x})|_{\bm{x}_{G}^e} \big|\big) \Big ) ,\label{alpha-e}
\\
&\alpha_{h}=\max\limits_{K,e,\bm{x}_G^e}  \alpha(U_h,\bm{x})|_{\bm{x}_{G}^e} =
\max\limits_{K,e,\bm{x}_G^e,m} \Big(\max\big (
\big|\lambda_{m}(U_{h,K}^{int},\bm{x})|_{\bm{x}_{G}^e} \big|, \; \big|\lambda_{m}(U_{h,K}^{ext},\bm{x})|_{\bm{x}_{G}^e} \big|\big{)} \Big ).\label{alpha-h}
\end{align}
It is the remark that one choice of the $\alpha_{LF}$ in local LF in the computation (e.g., see \cite{DG-review,Luo-Huang-Qiu-2019JCP}) is taken based on cell average, i.e.,
\begin{align}
&\tilde{\alpha}_{e} = \max\limits_{\bm{x}_G^e,m} \Big(\max\big (
\big|\lambda_{m}(\bar{U}_{h,K},\bm{x}_G^e) \big|, \; \big|\lambda_{m}(\bar{U}_{h,K'},\bm{x}_G^e) \big|\big{)} \Big ),\label{alpha-local}
\end{align}
where $\bar{U}_{h,K}$ is the average of $U_h$ on $K$, $\bar{U}_{h,K'}$ is the average of $U_h$ on $K'$, and $K$ and $K'$ are sharing $e\in \partial K$.

A fully discrete MMDG scheme can be obtained by discretizing (\ref{DG-0}) in time. A stability analysis is presented in the next section for a simple case with the explicit Euler scheme and numerical results are presented in Section~\ref{SEC:numerics} with a third-order strong-stability-preserving (SSP) Runge-Kutta scheme.

\section{CFL conditions on adaptive moving meshes}
\label{SEC:CFL}

Generally speaking, it is difficult, if not impossible, to obtain a CFL condition rigorously for a fully discrete version
of the MMDG scheme (\ref{DG-0}) for general conservation law (\ref{PDE-1}).
To gain insight into the stability, we consider a simple
situation with $P^0$-DG for spatial discretization, the explicit Euler scheme for time integration, and
for linear scalar conservation laws with the flux $\mathbf{F} = \bm{a}(\bm{x},t) U$. We also assume that $\Omega$ is
cubical, periodic boundary conditions are used, and $\bm{a}(\bm{x},t)$ is periodic in each coordinate direction.

Under these assumptions, we have
\[
\mathbf{H}(U_h,\bm{x}) = (\bm{a}(\bm{x},t)-\dot{\bm{X}}) U_h,
\quad \lambda(\bm{x},t) = (\bm{a}(\bm{x},t)-\dot{\bm{X}}) \cdot \bm{n},
\quad
\alpha(\bm{x},t) = \big|\lambda(\bm{x},t)\big|.
\]
Moreover,
\begin{align*}
\hat{ \mathbf{H}}(U_h,\bm{x}) & = \frac{1}{2}\Big(
\big(\mathbf{H}(U_{h,K}^{int},\bm{x})+\mathbf{H}(U_{h,K}^{ext},\bm{x}) \big) \cdot \bm{n}
-\alpha_{LF}(\bm{x},t) (U_{h,K}^{ext}-U_{h,K}^{int})\Big)
\\
& = \frac{1}{2}\left ( \lambda(\bm{x},t) (U_{h,K}^{int} + U_{h,K}^{ext}) -\alpha_{LF}(\bm{x},t) (U_{h,K}^{ext}-U_{h,K}^{int})\right )
\\
& = \frac{\alpha_{LF}(\bm{x},t) +\lambda(\bm{x},t) }{2} U_{h,K}^{int}
- \frac{\alpha_{LF}(\bm{x},t) - \lambda(\bm{x},t)}{2} U_{h,K}^{ext} .
\end{align*}
Applying the explicit Euler scheme to (\ref{DG-0}) and taking $\phi = 1$, we get
\begin{equation}
\label{DG-1}
|K^{n+1}| \bar{U}_{h,K}^{n+1}
= |K^{n}|  \bar{U}_{h,K}^{n}
-\Delta t_n \sum\limits_{e \in \partial K^n}|e|
\sum\limits_{\bm{x}_{G}^e}w_G^e \, \hat{\mathbf{H}}(U_h,\bm{x}) \big|_{\bm{x}_{G}^e} = 0 ,
\end{equation}
where $\bar{U}_{h,K}^{n}$ is the average of $U_h^{n}$ on $K^n$, $\bar{U}_{h,K}^{n+1}$ is the average of $U_h^{n+1}$ on $K^{n+1}$, and $K^n$ and $K^{n+1}$ are the corresponding elements in $\mathcal{T}_h^n$
and $\mathcal{T}_h^{n+1}$, respectively.

\begin{prop}
\label{prop-p0-cfl}
The MM $P^0$-DG scheme \eqref{DG-1} with $\mathbf{F} = \bm{a}(\bm{x},t) U$ is $L^1$-stable
under the CFL condition
\begin{equation}
\label{cfl-2}
\Delta t_n \le \frac{1}{\max\limits_{K^n}
\Big (\frac{1}{|K^{n}|} \sum\limits_{e \in \partial K^n}|e|\sum\limits_{\bm{x}_{G}^e} w_G^e~ \alpha_{CFL}^{n}(\bm{x}_{G}^e)\Big ) },
\end{equation}
where $\alpha_{CFL}^{n}(\bm{x}) \geq \alpha_{LF}^{n}(\bm{x}) $ and the subscripts $LF$ and $CFL$ stand for the LF numerical flux and CFL condition, respectively.
\end{prop}

\begin{proof}
From the assumption of $P^0$-DG, for any $\bm{x} \in e$ of $K^n$,
we have
\[
U_{h,K}^{int}|_{\bm{x}_{G}^e} = \bar{U}_{h,K}^{int} = \bar{U}_{h,K}, \quad
U_{h,K}^{ext}|_{\bm{x}_{G}^e} = \bar{U}_{h,K}^{ext} = \bar{U}_{h,K'},
\]
where $K'$ is the element sharing $e$ with $K$. With this, \eqref{DG-1} can be rewritten as
\begin{align}
|K^{n+1}| \bar{U}_{h,K}^{n+1}
&= |K^{n}|  \bar{U}_{h,K}^{n}
- \Delta t_n  \sum\limits_{e \in \partial K^n} |e| \sum\limits_{\bm{x}_{G}^e}w_G^e \Big(\frac{\alpha_{LF}^n+\lambda^n}{2}  U_{h,K}^{n,int} - \frac{ \alpha_{LF}^n-\lambda^n}{2}  U_{h,K}^{n,ext} \Big )
\Big|_{\bm{x}_{G}^e}
\notag \\&=
|K^{n}|  \bar{U}_{h,K}^{n}
- \Delta t_n  \sum\limits_{e \in \partial K^n} |e|\bar{U}_{h,K}^{n}  \sum\limits_{\bm{x}_{G}^e}w_G^e \Big(\frac{\alpha_{LF}^n(\bm{x}_{G}^e)+\lambda^n(\bm{x}_{G}^e)}{2}  \Big)
\notag \\&~~~~~~~~~~~~~~~~+\Delta t_n \sum\limits_{e \in \partial K^n} |e|\bar{U}_{h,K}^{n,ext} \sum\limits_{\bm{x}_{G}^e}w_G^e \Big(\frac{ \alpha_{LF}^n(\bm{x}_{G}^e)-\lambda^n(\bm{x}_{G}^e)}{2}  \Big)
\notag \\&=
|K^{n}|  \bar{U}_{h,K}^{n} \Big [
1 - \frac{\Delta t_n}{|K^{n}|}  \sum\limits_{e \in \partial K^n} |e|  \sum\limits_{\bm{x}_{G}^e}w_G^e \Big(\frac{\alpha_{LF}^n(\bm{x}_{G}^e)+\lambda^n(\bm{x}_{G}^e)}{2}  \Big)\Big ]
\notag \\&~~~~~~~~~~~~~~~~+\Delta t_n \sum\limits_{e \in \partial K^n} |e|\bar{U}_{h,K}^{n,ext} \sum\limits_{\bm{x}_{G}^e}w_G^e
\Big(\frac{ \alpha_{LF}^n(\bm{x}_{G}^e)-\lambda^n(\bm{x}_{G}^e)}{2}  \Big) .
\label{DG-2}
\end{align}
Notice that, for any $\bm{x}_{G}^e \in e\subset \partial K^n$, we have
\[
0\leq\alpha_{LF}^n(\bm{x}_{G}^e) + \lambda^n(\bm{x}_{G}^e)\leq 2\alpha_{CFL}^n(\bm{x}_{G}^e),
\quad 0 \le \alpha_{LF}^n(\bm{x}_{G}^e) - \lambda^n(\bm{x}_{G}^e).
\]
From the CFL condition \eqref{cfl-2}, we have
\[
1 - \frac{\Delta t_n}{|K^{n}|} \sum\limits_{e \in \partial K^n}|e| \sum\limits_{\bm{x}_{G}^e}w_G^e
 \Big( \frac{\alpha_{LF}^n(\bm{x}_{G}^e)+\lambda^n(\bm{x}_{G}^e)}{2} \Big ) \geq 0 .
\]
From this, taking the absolute value on both side of \eqref{DG-2} gives
\begin{align*}
|K^{n+1}|  |\bar{U}_{h,K}^{n+1}| \leq & |K^{n}|  |\bar{U}_{h,K}^{n}|  \Big [ 1 -
\frac{\Delta t_n}{|K^{n}|} \sum\limits_{e \in \partial K^n}|e|
\sum\limits_{\bm{x}_{G}^e}w_G^e \Big( \frac{\alpha_{LF}^n(\bm{x}_{G}^e)+\lambda^n(\bm{x}_{G}^e)}{2} \Big ) \Big ]
\notag \\
& \quad + \Delta t_n  \sum\limits_{e \in \partial K^n} |e||\bar{U}_{h,K}^{n,ext} | \sum\limits_{\bm{x}_{G}^e}w_G^e \Big ( \frac{ \alpha_{LF}^n(\bm{x}_{G}^e) - \lambda^n(\bm{x}_{G}^e)}{2}\Big).
\end{align*}
Summing this over all elements, we get
\begin{align*}
\sum_{K^{n+1}} |K^{n+1}|  |\bar{U}_{h,K}^{n+1}| \leq & \sum_{K^n} |K^{n}|  |\bar{U}_{h,K}^{n}|  \Big [ 1 -
\frac{\Delta t_n}{|K^{n}|} \sum\limits_{e \in \partial K^n}|e|
\sum\limits_{\bm{x}_{G}^e}w_G^e \Big( \frac{\alpha_{LF}^n(\bm{x}_{G}^e)+\lambda^n(\bm{x}_{G}^e)}{2} \Big ) \Big ]
\notag \\
& \quad + \Delta t_n \sum_{K^n} \sum\limits_{e \in \partial K^n} |e||\bar{U}_{h,K}^{n,ext} | \sum\limits_{\bm{x}_{G}^e}w_G^e \Big ( \frac{ \alpha_{LF}^n(\bm{x}_{G}^e) - \lambda^n(\bm{x}_{G}^e)}{2}\Big).
\end{align*}
We notice that $\sum\limits_{e \in \partial K^n}  |e| \bar{U}_{h,K}^{n,ext}(\cdots) $
goes over all neighboring elements ($K'$) of $K^n$ and
each term  can be considered to be associated with $K'$ but with $\bm{n}$ being changed to $-\bm{n}$ (because the unit outward normal of $e$ in view of $K^n$ is opposite to the unit outward normal of $e$ in view of $K'$).
From the periodicity assumption
on the boundary conditions and $a(\bm{x},t)$, we can rewrite the second term on the right-hand side of the above equation as
\begin{align*}
& \Delta t_n \sum_{K^n} \sum\limits_{e \in \partial K^n} |e||\bar{U}_{h,K}^{n,ext}| \sum\limits_{\bm{x}_{G}^e}w_G^e \Big ( \frac{ \alpha_{LF}^n(\bm{x}_{G}^e) - \lambda^n(\bm{x}_{G}^e)}{2}\Big)
\\
& = \Delta t_n \sum_{K^n} \sum\limits_{e \in \partial K^n} |e| |\bar{U}_{h,K}^{n}| \sum\limits_{\bm{x}_{G}^e} w_G^e\Big ( \frac{ \alpha_{LF}^n(\bm{x}_{G}^e) + \lambda^n(\bm{x}_{G}^e)}{2}\Big) .
\end{align*}
Combining these, we have
\begin{align*}
\sum_{K^{n+1}} |K^{n+1}|  |\bar{U}_{h,K}^{n+1}| &\leq\sum_{K^n} |K^{n}|  |\bar{U}_{h,K}^{n} |\Big [ 1 -
\frac{\Delta t_n}{|K^{n}|} \sum\limits_{e \in \partial K^n}|e|
\sum\limits_{\bm{x}_{G}^e}w_G^e \Big( \frac{\alpha_{LF}^n(\bm{x}_{G}^e)+\lambda^n(\bm{x}_{G}^e)}{2} \Big ) \Big ]
\notag \\
& \quad + \Delta t_n \sum_{K^n} \sum\limits_{e \in \partial K^n} |e||\bar{U}_{h,K}^{n} | \sum\limits_{\bm{x}_{G}^e}w_G^e \Big ( \frac{ \alpha_{LF}^n(\bm{x}_{G}^e) + \lambda^n(\bm{x}_{G}^e)}{2}\Big)
\\
& = \sum_K |K^{n}| | \bar{U}_{h,K}^{n}| .
\end{align*}
Hence, the scheme is $L^1$-stable.
\end{proof}

\vspace{20pt}

It should be pointed out that the CFL condition (\ref{cfl-2}) is only a sufficient condition. Nevertheless, it offers
several insights on the maximum time step allowed by stability. We elaborate these in the following remarks.

\begin{rem}
\label{rem-cfl-0}
Condition (\ref{cfl-2}) involves an important factor $|e|/|K|$. It is known that
\begin{equation}
\label{height-1}
\frac{|e|}{|K|} = \frac{d}{\sigma_e}, \quad \forall e \in \partial K
\end{equation}
where $\sigma_e$ is an element height defined as the distance between $e$ and the vertex of $K$ opposite to $e$.
Using this, we can rewrite (\ref{cfl-2}) into
\begin{equation}
\label{cfl-2-1}
\Delta t \le \frac{1}{d \max\limits_{K}
\Big ( \sum\limits_{e \in \partial K} \frac{1}{\sigma_e}\sum\limits_{\bm{x}_{G}^e} w_G^e~\alpha_{CFL}(\bm{x}_{G}^e) \Big ) }.
\end{equation}
This indicates that the allowed maximum time step depends on the coupling between
$\sigma_e$ and $\alpha_{CFL}(\bm{x}_{G}^e)$.
Recall that $\alpha$ (cf. (\ref{alpha-f})) is defined as the maximum of the eigenvalue of the Jacobian matrix of
$(\mathbf{F}(U_h,\bm{x}) - U_h \dot{\bm{X}})\cdot \bm{n}$. Thus, if the mesh velocity can be chosen to minimize
$\alpha_{CFL}$ (such as in Lagrangian-type methods; e.g., see
\cite{Hirt-ALE1-1971,Hirt-ALE2-1974,Kucharik-Shashkov-ALE-2014}), a larger time step can be used.
On the other hand, the mesh movement can be determined by other considerations. For example, the meshes
in the examples of Section \ref{SEC:numerics} are moved via the MMPDE method using solution-Hessian based mesh
adaptation.  In this case, the mesh velocity does not necessarily minimize $\alpha_{CFL}$.
The condition (\ref{cfl-2-1}) shows that, if $\alpha_{CFL}(\bm{x}_{G}^e)$ does not change significantly over the domain,
$\Delta t$ is determined by the minimum height $\sigma_{h,min} = \min_{K, e} \sigma_e$ of the mesh elements.
\end{rem}

\begin{rem}
\label{rem-cfl-1}
From the above proof we can see that the choice of $\alpha$ in the CFL condition (\ref{cfl-2}) can be different from that in the DG scheme \eqref{alpha-p} as long as
\begin{equation}
\label{alpha-0}
\alpha_{CFL}(\bm{x}) \geq\alpha_{LF}(\bm{x}).
\end{equation}
For example, we can use $\alpha_{LF} = \alpha_{p} $ (pointwise, cf. \eqref{alpha-p}) for the scheme (denoted as $\alpha_{LF,p}$) and $\alpha_{CFL} = \alpha_{h}$ (global, cf. \eqref{alpha-h}) for the CFL condition (denoted as $\alpha_{CFL,h} $),
which has been expressed for the general form (\ref{PDE-1}) of conservation laws.
This works since (\ref{alpha-0}) is satisfied.
On the other hand, the proof of Proposition~\ref{prop-p0-cfl} will not hold in general for the choice with
$\alpha_{LF} = \alpha_{h}$ and $ \alpha_{CFL} = \alpha_{p}$ since it violates (\ref{alpha-0}).
As a consequence, it is unclear if there is a theoretical guarantee that the scheme is $L^1$-stable for this choice.
\end{rem}

\begin{rem}
\label{rem-cfl-2}
For the choice $\alpha_{CFL} = \alpha_h$ (denoted as $\alpha_{CFL,h}$), the CFL condition \eqref{cfl-2} becomes
\begin{equation}
\Delta t \le \frac{1}{\alpha_{CFL,h} \max\limits_{K}\Big ( \frac{1}{|K|} \sum\limits_{e \in \partial K} |e| \Big ) } .
\label{cfl-3}
\end{equation}
This corresponds to the CFL condition (\ref{cfl-old-1}) that has been commonly used in existing adaptive DG computation
and is a direct extension of CFL conditions used for fixed, uniform meshes.
\end{rem}

\begin{rem}
\label{rem-cfl-3}
We can use something in between the very local $\alpha_p$ \eqref{alpha-p} and the global one in \eqref{alpha-h}.
For example, we take the maximum value of $\alpha$ over the Gauss points on edge $e$ (cf. \eqref{alpha-e}).
For this choice ($\alpha_{CFL} = \alpha_e$, denoted as $\alpha_{CFL,e}$), the CFL condition \eqref{cfl-2} becomes
\begin{equation}
\Delta t \le \frac{1}{\max\limits_{K} \Big ( \frac{1}{|K|} \sum\limits_{e \in \partial K}  |e|\alpha_{CFL,e} \Big )} .
\label{cfl-4}
\end{equation}
\end{rem}

\begin{rem}
\label{rem-cfl-4}
Since an explicit Runge-Kutta scheme can be expressed as a combination of
the explicit Euler scheme with different time stepsize, we expect that the above analysis applies
to explicit Runge-Kutta schemes as well. Moreover,
for general $P^k$-DG ($k\ge 0$) and general systems of conservation laws, based on
\cite{DG-series0,DG-series2}, we suggest to use (\ref{cfl-new-1}),
where $\alpha_e$ is defined in \eqref{alpha-e}.
A choice of $C_{cfl}$
is $C_{cfl} \leq 1/(2k+1)$ \cite{DG-series2}.
\end{rem}

\section{Numerical results}
\label{SEC:numerics}

In this section we present numerical results obtained with the MMDG method described in the previous sections
with the third-order explicit SSP Runge-Kutta scheme (SSP RK3)\cite{SSP-time} for one- and two-dimensional Burgers' equations and Euler equations.
The moving mesh is generated by the MMPDE moving mesh method; e.g., see
\cite[Section 4]{Zhang-Cheng-Huang-Qiu-2020CiCP} or \cite[Section 3]{Zhang-Huang-Qiu-2020SISC} for a brief yet complete description of the method and
\cite{Huang-Ren-Russell-1994a,Huang-Sun-2003JCP,Huang-Russell-2011,Huang-Kamenski-2015JCP,Huang-Kamenski-2018MC}
for a more detailed description and a development history.
A key idea of the MMPDE method is to view any nonuniform mesh as a uniform one
in some Riemannian metric specified by a tensor $\mathbb{M} = \mathbb{M}(\bm{x})$,
a symmetric and uniformly positive definite matrix-valued function that  provides the information needed for
determining the size, shape, and orientation of the mesh elements throughout the domain.

In this work we use an optimal metric tensor based on the $L^2$-norm of piece linear interpolation error
\cite{Huang-Sun-2003JCP, Huang-Russell-2011}.
To be specific, we consider a physical variable $u$ and its finite element approximation $u_h$.
Let $H_K$ be a recovered Hessian of $u_h$ on $K\in \mathcal{T}_{h}$ such as one obtained
using least squares fitting. Assuming that the eigen-decomposition of $H_K$ is given by
\[
H_K = Q\hbox{diag}(\lambda_1,\cdots,\lambda_d)Q^T,
\]
where $Q$ is an orthogonal matrix, we define
\[
|H_K| = Q\hbox{diag}(|\lambda_1|,...,|\lambda_d|)Q^T.
\]
The metric tensor is defined as
\begin{equation}
\label{M-u}
\mathbb{M}_{K} =\det \big{(}\beta_{h}\mathbb{I}+|H_K|\big{)}^{-\frac{1}{d+4}}
\big{(}\beta_{h}\mathbb{I}+|H_K|\big{)},
\quad \forall K \in \mathcal{T}_h
\end{equation}
where $\mathbb{I}$ is the identity matrix, $\det(\cdot)$ is the determinant of a matrix,
and $\beta_{h}$ is a regularization parameter defined through the algebraic equation
\[
\sum_{K\in\mathcal{T}_h}|K|\, \hbox{det}(\beta_{h}\mathbb{I}+|H_K|)^{\frac{2}{d+4}}
=2\sum_{K\in\mathcal{T}_h}|K|\,
\hbox{det}(|H_K|)^{\frac{2}{d+4}}.
\]
Roughly speaking, the choice of (\ref{M-u}) is to concentrate mesh points in regions where
the determinant of the Hessian is large.

In our numerical results,  we use the physical solution $u$ to compute the metric tensor for Burgers' equation
and the density $\rho$ and the entropy $\mathcal{S} = \ln(P\rho^{-\gamma})$ for the Euler equations, unless otherwise stated.
To explain the latter, we first compute $\mathbb{M}^{\mathcal{\rho}}_{K} $ and $\mathbb{M}^{\mathcal{S}}_{K} $ using
(\ref{M-u}) with $u = \rho$ and $\mathcal{S}$, respectively.
Then, a new metric tensor is obtained through matrix intersection as
\begin{equation}\label{M-Srho}
\tilde{\mathbb{M}}_{K}
=\frac{{\mathbb{M}}^{\mathcal{S}}_{K} }{|||{\mathbb{M}}^{\mathcal{S}}_{K}|||}
\cap\frac{\mathbb{M}^{\rho}_{K}}{|||{\mathbb{M}}^{\rho}_{K}|||} ,
\end{equation}
where $|||\cdot |||$ denotes the maximum absolute value
of the entries of a matrix and ``$\cap$" stands for matrix intersection.
The reader is referred to
\cite{Zhang-Cheng-Huang-Qiu-2020CiCP} for the definition and geometric interpretation of matrix intersection.

The procedure of the MMDG method is presented in Algorithm~\ref{QLMM-DG}.
\begin{breakablealgorithm}
\caption{The MMDG method for hyperbolic conservation laws.}
\label{QLMM-DG}
\begin{itemize}
\item[0.] {\bf Initialization.}
For a given initial mesh $\mathcal{T}_h^0$, project the initial physical variables into the DG space
$\mathcal{V}_h^{k,0}$ to obtain $U^0_h$.

For $n = 0, 1, \cdots$, do

\item[1.] {\bf Mesh adaptation.}
\begin{enumerate}
\item [(1.1)] Compute the time step $\tilde{\Delta} t_n$ according to \eqref{cfl-new-1}
	based on $\mathcal{T}_h^n=\{\bm{x}^{n}_i\}$ and $U^n_h$, i.e.,
	\begin{equation}
	\tilde{\Delta} t_n = \frac{C_{cfl}}{\max\limits_{K^n} \Big ( \frac{1}{|K^{n}|} \sum\limits_{e \in \partial K^n}
	|e|  \tilde{\alpha}_{CFL}^n\Big )} .
	\label{cfl-5-1}
	\end{equation}
where $\tilde{\alpha}^n_{CFL}$ is chosen based on the eigenvalues of the Jacobian matrix of $\mathbf{F}\cdot \bm{n}$ (with respect to $U$) evaluated on edge $e$.
  \item [(1.2)] Compute the metric tensor $\mathbb{M}$ based on $\mathcal{T}_h^n$ and $U^n_h$.
  \item [(1.3)] Generate the new mesh $\tilde{\mathcal{T}}_h^{n+1}=\{\tilde{\bm{x}}^{n+1}_i\}$ using the MMPDE moving mesh method.
  \item [(1.4)] Compute the nodal mesh velocity as
\[
\dot{\bm{x}}^{n}_i =\frac{\tilde{\bm{x}}^{n+1}_i - \bm{x}^{n}_i}{\tilde{\Delta} t_n},\quad i=1,...,N_v.
\]
\item [(1.5)] Compute the time step $\Delta t_n$ (using \eqref{cfl-new-1}) based on $\mathcal{T}_h^n$, $\tilde{\mathcal{T}}_h^{n+1}$ and $U^n_h$ as
\begin{equation}
	\Delta t_n = \frac{C_{cfl}}{\max\limits_{K}
\left[\max \left(\frac{ \sum\limits_{e \in \partial K^n}  |e|\alpha^n_{CFL} }{|K^{n}|},\frac{ \sum\limits_{e \in \partial \tilde{K}^{n+1}}  |e|\alpha^n_{CFL} }{|\tilde{K}^{n+1}|}\right)\right]},
	\label{cfl-5}
	\end{equation}
where $\alpha^n_{CFL}$ is chosen based on the eigenvalues of the Jacobian matrix of $(\mathbf{F} - U \dot{\bm{X}})\cdot \bm{n}$ (with respect to $U$) evaluated on edge $e$.
\item [(1.6)] Finally, the physical mesh $\mathcal{T}_h^{n+1}$ is defined as
\[
\bm{x}^{n+1}_i =\bm{x}^{n}_i + \Delta t_n ~\dot{\bm{x}}^{n}_i ,\quad i=1,...,N_v .
\]
\end{enumerate}
\item[2.] {\bf Solution of the physical equations on the moving mesh.}
Integrate the physical equations from $t_n$ to $t_{n+1}$ using the MMDG scheme to obtain $U^{n+1}_h$.
\end{itemize}
\end{breakablealgorithm}

It is remarked that the CFL condition (\ref{cfl-5}) has been used in the above algorithm to take the old and new meshes
$\mathcal{T}_h^n$ and $\mathcal{T}_h^{n+1}$ into consideration.
The rationale behind this is that the SSP RK3 scheme we use for the time integration has three stages that can be viewed roughly as the explicit Euler scheme from
$t_n$ to $t_{n+1}$, $t_{n+1}$ to $t_{n+\frac{1}{2}}$, and $t_{n+\frac{1}{2}}$ to $t_{n+1}$, respectively.
The computation of the right-hand side of (\ref{DG-1}) involves both $\mathcal{T}_h^n$ and $\mathcal{T}_h^{n+1}$
and thus it would be better to take the effects of these meshes into consideration more directly
in time step even the mesh velocity is involved in the computation of $\alpha_e$.
In principle, we should also update $\alpha_e$ during the Runge-Kutta stages. However, this can cause
changes in the time step-size during the Runge-Kutta stepping, which requires to re-start the stepping with a new time step-size.
To avoid this complicity, we choose to freeze $\alpha_e$ at $t=t_n$.

We also note that $\alpha_e$ has been used in the algorithm and it can be replaced by $\alpha_h$.
Unless otherwise stated, $\alpha_{e}$ is chosen in the LF numerical flux.
In fact, in the following numerical examples, we consider two options $\alpha_e$ and $\alpha_h$
for each of the LF numerical flux and the CFL condition. For example, ($\alpha_{CFL,h}$, $\alpha_{LF,e}$) indicates that $\alpha_h$ is used in the CFL condition and $\alpha_{e}$ used in the LF numerical flux.
Furthermore, we consider moving mesh $P^k$-DG with $k = 1$, $2$, and $3$.
We take the CFL number $C_{cfl}$ as $0.3$ for $P^1$-DG, $0.15$ for $P^2$-DG, and $0.1$ for $P^3$-DG.

The CFL condition such as (\ref{cfl-5}) provides a selection strategy for time step-size based on stability. This strategy is
widely used in the DG computation of hyperbolic conservation laws.
On the other hand, it is common practice to use an error-based time step-size selection strategy in the computation of ordinary
differential equations (e.g., see Hairer and Wanner \cite{HW1991}).
It is worth studying how error-based time step-size selection strategies fare in the DG computation of hyperbolic conservation laws.
Interestingly, Conde et al. \cite{SSPERK-2018} developed a family of embedded pairs of SSP Runge-Kutta schemes and
combined them with WENO5 (for spatial discretization on fixed meshes) for the numerical solution of
the one-dimensional Euler equations.
Particularly, they compared the performance of the third-order explicit SSP embedded RK (SSP-ERK(4,3)) pair with RK(3,2) of Bogacki and Shampine \cite{ode23-1989} and showed that both pairs lead to stable computation while
SSP-ERK(4,3) offers stronger stability. In this section, we also present numerical results to demonstrate
this idea of using error-based time step-size selection strategies for the computation of hyperbolic conservation laws.
Our demonstration is more general than that of \cite{SSPERK-2018} in the sense that we consider one- and two-dimensional
conservation laws and DG discretization on adaptive moving meshes. We use SSP-ERK(4,3) of \cite{SSPERK-2018}
(see its Butcher tableau in Table \ref{tab:Butcher}) and the standard PI controller for time step-size
selection \cite{HW1991}, with the relative and absolute tolerances chosen as $1\times 10^{-6}$ and $1\times 10^{-8}$, respectively.
This selection strategy is compared with that based on the CFL condition.

\begin{table}[H]
\caption{Butcher tableau for SSP ERK(4,3).}
\vspace{3pt}
\centering
\label{tab:Butcher}
\begin{tabular}{cccc|c}
                &                &              &                & $0$ \\
  $\frac{1}{2}$ &                &              &                & $\frac{1}{2}$ \\
  $\frac{1}{2}$ &$\frac{1}{2}$   &              &                & $1 $   \\
  $\frac{1}{6}$ &$\frac{1}{6}$   &$\frac{1}{6}$ &               & $\frac{1}{2}$ \\
  \hline
  $\frac{1}{6}$ &$\frac{1}{6}$   &$\frac{1}{6}$ & $\frac{1}{2}$  & $1 $\\
  $\frac{1}{4}$ &$\frac{1}{4}$   &$\frac{1}{4}$ & $\frac{1}{4}$  & $1 $ \\
\end{tabular}
\end{table}

\begin{example}\label{Burgers-1d}
(1D Burgers' equation)
\end{example}
We first consider Burgers' equation in one dimension,
\begin{equation}
\label{burgers-2d}
u_t +  \big(\frac{u^2}{2}\big)_x  = 0,\quad x\in(0,2)
\end{equation}
subject to the initial condition $u(x,0) = \frac{1}{2} + \sin(\pi x)$ and
periodic boundary conditions. The final time is $T = 1$.

The mesh trajectories, solution, and time step-size obtained with the moving mesh $P^k$-DG method ($k=1,2,3$)
and $N=100$ are shown in Fig.~\ref{Fig:Burgers-1d-Pk}.
Three combinations of $\alpha_e$ and $\alpha_h$ are used: ($\alpha_{CFL,e}$, $\alpha_{LF,e}$), ($\alpha_{CFL,h}$, $\alpha_{LF,e}$), and ($\alpha_{CFL,h}$, $\alpha_{LF,h}$).
They all lead to stable computation and almost identical solutions and mesh trajectories.
A close examination on $\Delta t$ indicates that ($\alpha_{CFL,e}$, $\alpha_{LF,e}$) gives a slightly larger
and less oscillatory time step than those with ($\alpha_{CFL,h}$, $\alpha_{LF,e}$)
and ($\alpha_{CFL,h}$, $\alpha_{LF,h}$). We may attribute the former to the fact $\alpha_e \le \alpha_h$ (cf. (\ref{cfl-5})).
However, it is unclear to the authors why ($\alpha_{CFL,h}$, $\alpha_{LF,e}$)
and ($\alpha_{CFL,h}$, $\alpha_{LF,h}$), which appear to be more stable, lead to more oscillatory $\Delta t$.
This feature is also observed in other examples presented in this section except Example \ref{RP-2d}.
The figures also show that the time step behaves similar qualitatively
and quantitatively for ($\alpha_{CFL,h}$, $\alpha_{LF,e}$) and ($\alpha_{CFL,h}$, $\alpha_{LF,h}$).

The mesh trajectories ($N=100$) obtained with the $P^1$-DG method of $\Delta t$ and a PI-controller time step-size selection
strategy (denoted as ERK PI) are plotted in Fig.~\ref{Fig:Burgers-1d-Pk-ERK}(a).
The solution and time step-size are obtained with the $P^1$-DG method and ERK PI and ($\alpha_{CFL,e}$, $\alpha_{LF,e}$)
step-size selection strategies are shown in Fig.~\ref{Fig:Burgers-1d-Pk-ERK}(d)
and Fig.~\ref{Fig:Burgers-1d-Pk-ERK}(g), respectively.
The results of $P^2$-DG and $P^3$-DG are also shown in Fig.~\ref{Fig:Burgers-1d-Pk-ERK}.
From Fig.~\ref{Fig:Burgers-1d-Pk-ERK}, we see that both time-size selection strategies lead to stable computation
and almost identical solutions and mesh trajectories. ERK PI produces slightly larger time step-size than
($\alpha_{CFL,e}$, $\alpha_{LF,e}$), and this is especially true for higher-order DG. To guarantee the stability, the CFL number
and thus $\Delta t$ in the CFL condition decrease when $k$ increases. This is reflected in Fig.~\ref{Fig:Burgers-1d-Pk-ERK}.
On the other hand, $\Delta t$ selected by ERK PI does not seem to have big changes when $k$ increases.
As a result, the difference in $\Delta t$ selected by ERK PI and CFL condition becomes bigger as $k$ increases.
Note that ERK PI leads to more oscillatory $\Delta t$, which is consistent with observations made in the computation
of general stiff equations; e.g., see \cite{HW1991}.

The time step-size obtained with $N=50,~100,~200$ is shown in Fig~\ref{Fig:Burgers-1d-Pk-ERK-dt}.
The number of time steps obtained with $N=50,~100,~200$ is shown in Fig~\ref{Fig:Burgers-1d-Pk-Nsteps}.
The number of time steps appears to be a linear function of $N$ for all cases.
Interestingly, the slope of the straight lines for CFL-based selection is significantly larger than that associated with ERK PI,
and becomes larger as $k$ increases. On the other hand, the lines associated with ERK PI for $k=1,2,3$ differ only slightly.

\begin{figure}[H]
\centering
\subfigure[$P^1$-DG: mesh trajectories]{
\includegraphics[width=0.31\textwidth,trim=25 5 45 20,clip]{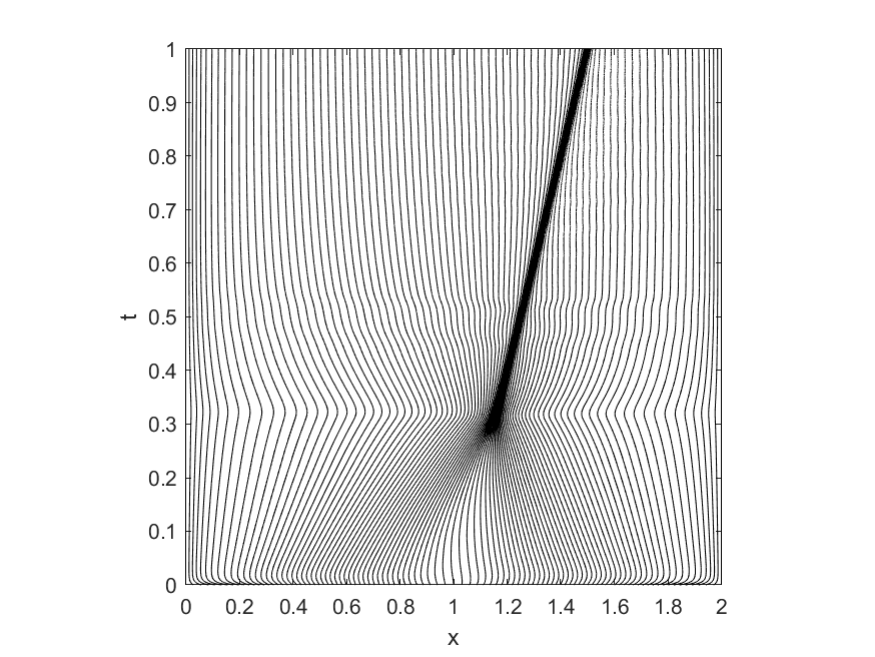}}
\subfigure[$P^2$-DG: mesh trajectories]{
\includegraphics[width=0.31\textwidth,trim=25 5 45 20,clip]{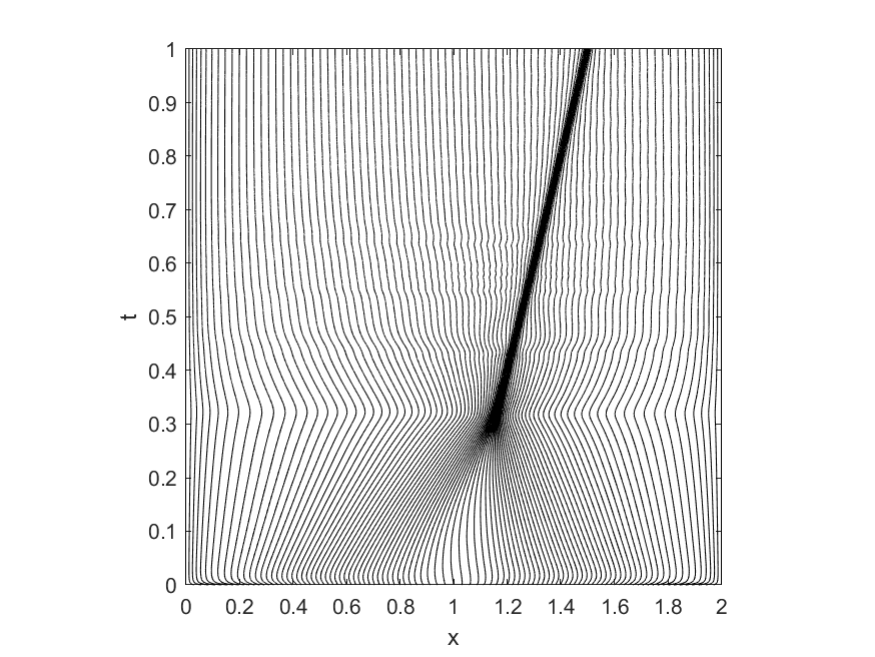}}
\subfigure[$P^3$-DG: mesh trajectories]{
\includegraphics[width=0.31\textwidth,trim=25 5 45 20,clip]{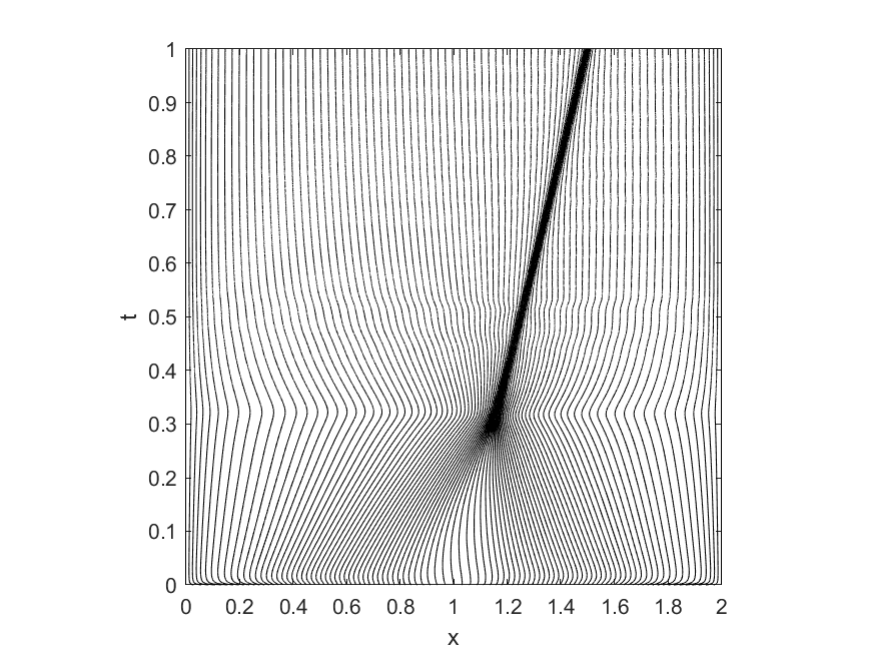}}
\subfigure[$P^1$-DG: solution $u$]{
\includegraphics[width=0.31\textwidth,trim=25 5 45 20,clip]{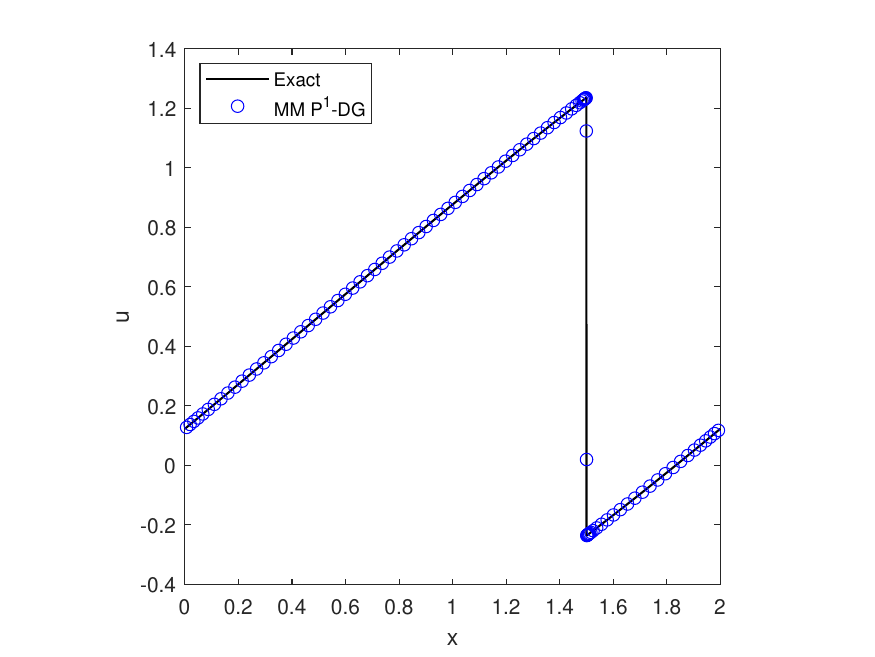}}
\subfigure[$P^2$-DG: solution $u$]{
\includegraphics[width=0.31\textwidth,trim=25 5 45 20,clip]{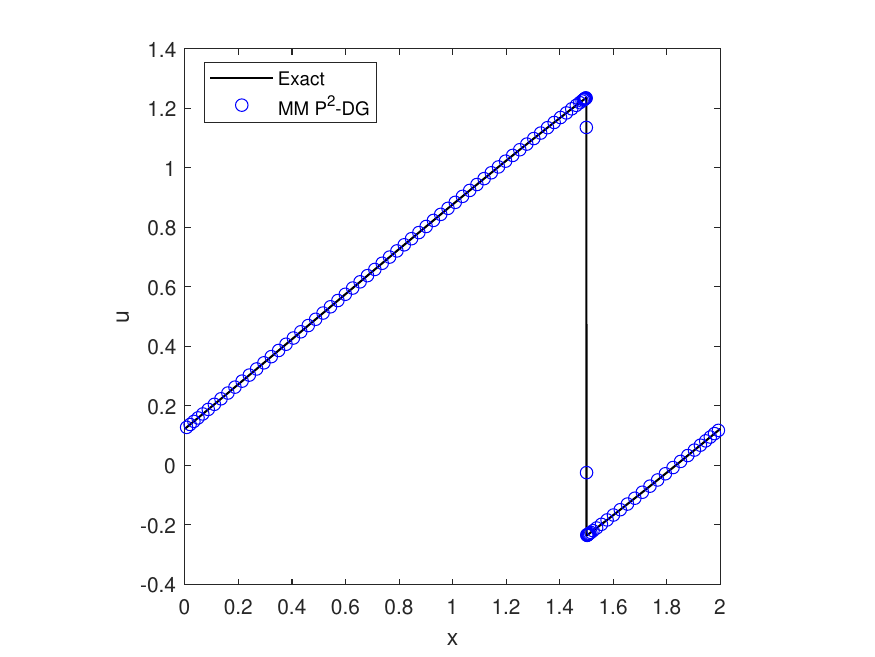}}
\subfigure[$P^3$-DG: solution $u$]{
\includegraphics[width=0.31\textwidth,trim=25 5 45 20,clip]{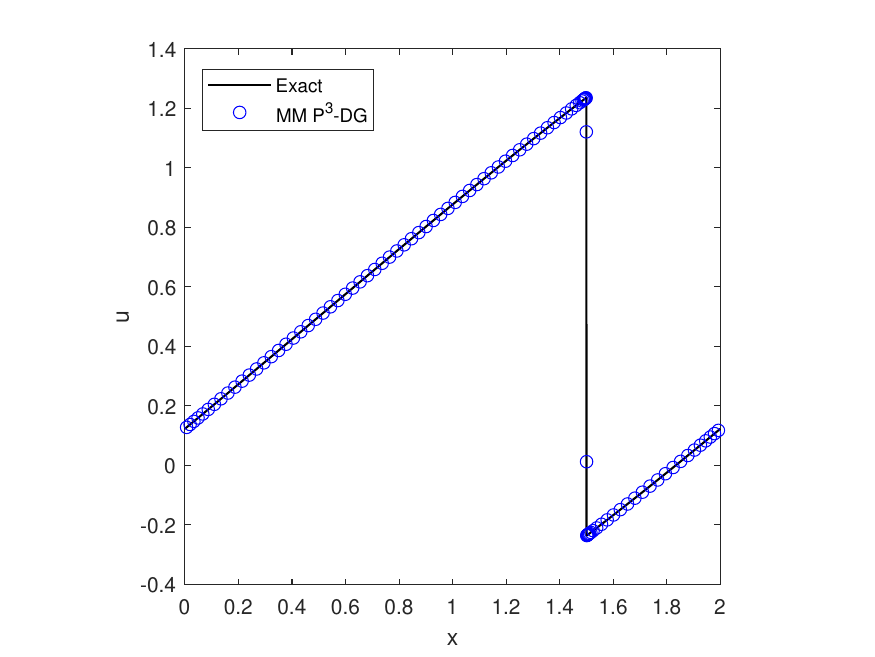}}
\subfigure[$P^1$-DG: $\Delta t$]{
\includegraphics[width=0.31\textwidth,trim=25 5 45 20,clip]{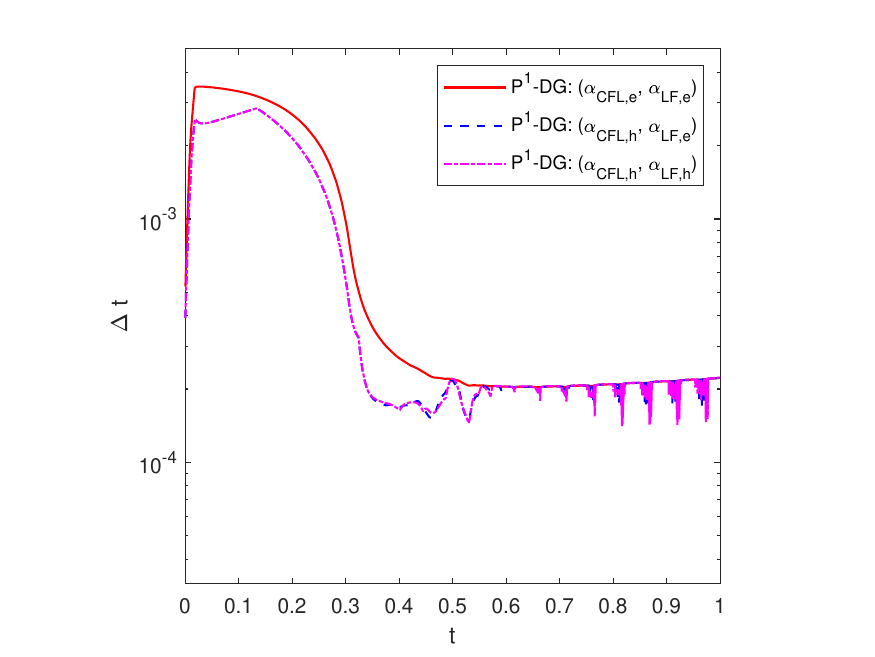}}
\subfigure[$P^2$-DG: $\Delta t$]{
\includegraphics[width=0.31\textwidth,trim=25 5 45 20,clip]{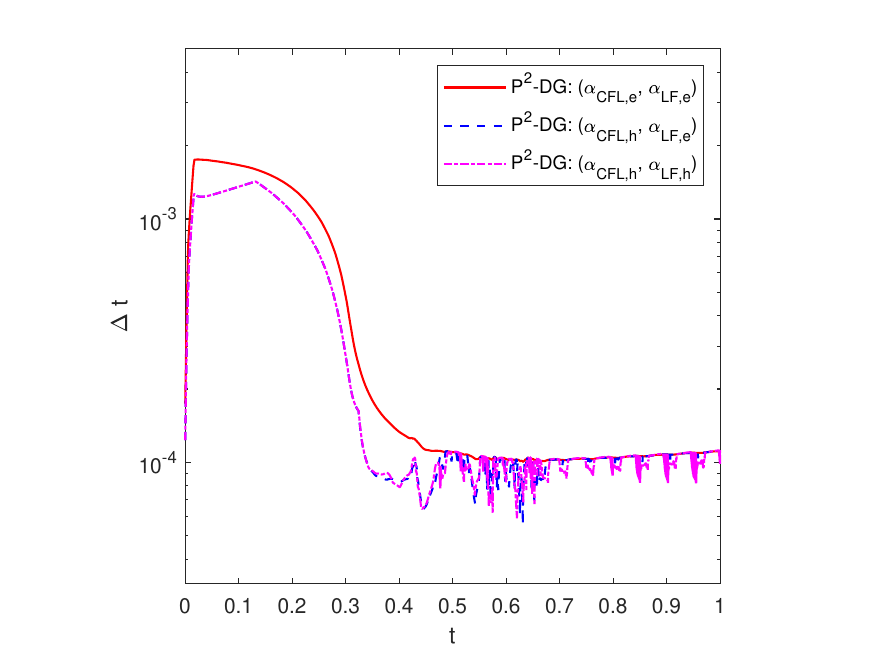}}
\subfigure[$P^2$-DG: $\Delta t$]{
\includegraphics[width=0.31\textwidth,trim=25 5 45 20,clip]{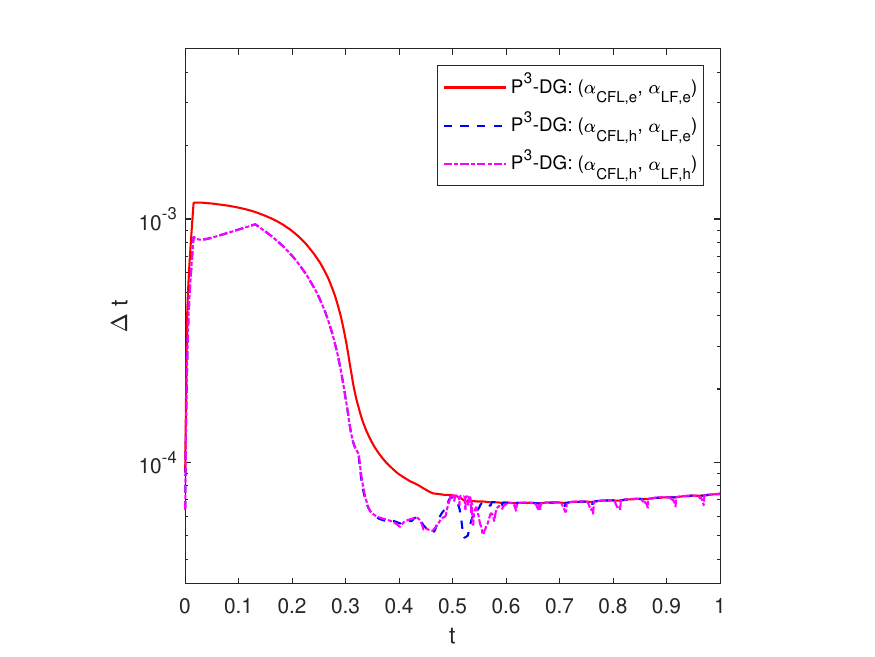}}
\caption{
Example \ref{Burgers-1d}. The mesh trajectories, solution, and time step-size are obtained with the $P^k$-DG method ($k=1,2,3$) and a moving mesh of $N=100$.}
\label{Fig:Burgers-1d-Pk}
\end{figure}

\begin{figure}[H]
\centering
\subfigure[$P^1$-DG: mesh trajectories]{
\includegraphics[width=0.31\textwidth,trim=25 5 45 20,clip]{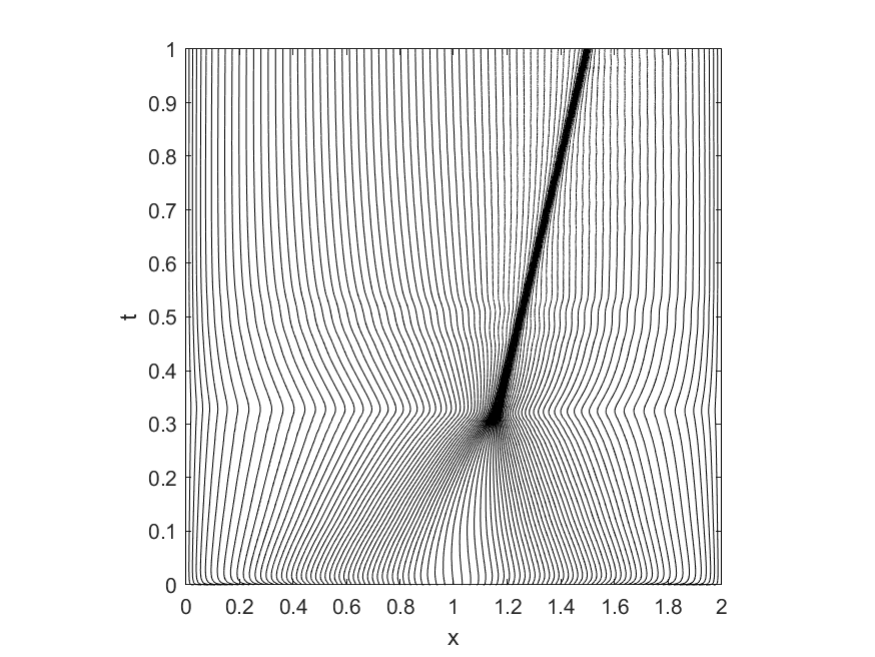}}
\subfigure[$P^2$-DG: mesh trajectories]{
\includegraphics[width=0.31\textwidth,trim=25 5 45 20,clip]{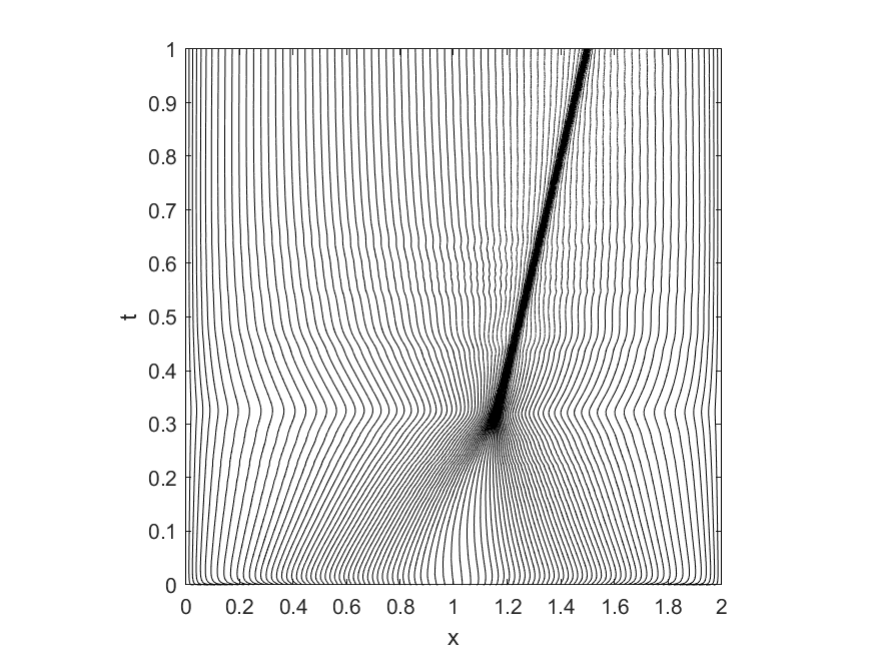}}
\subfigure[$P^3$-DG: mesh trajectories]{
\includegraphics[width=0.31\textwidth,trim=25 5 45 20,clip]{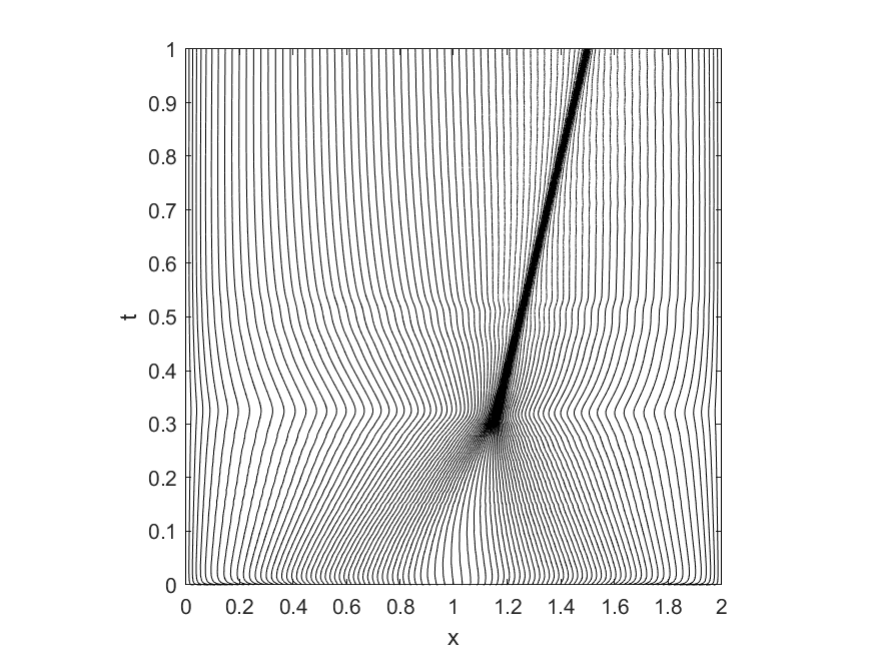}}
\subfigure[$P^1$-DG: solution $u$]{
\includegraphics[width=0.31\textwidth,trim=25 5 45 20,clip]{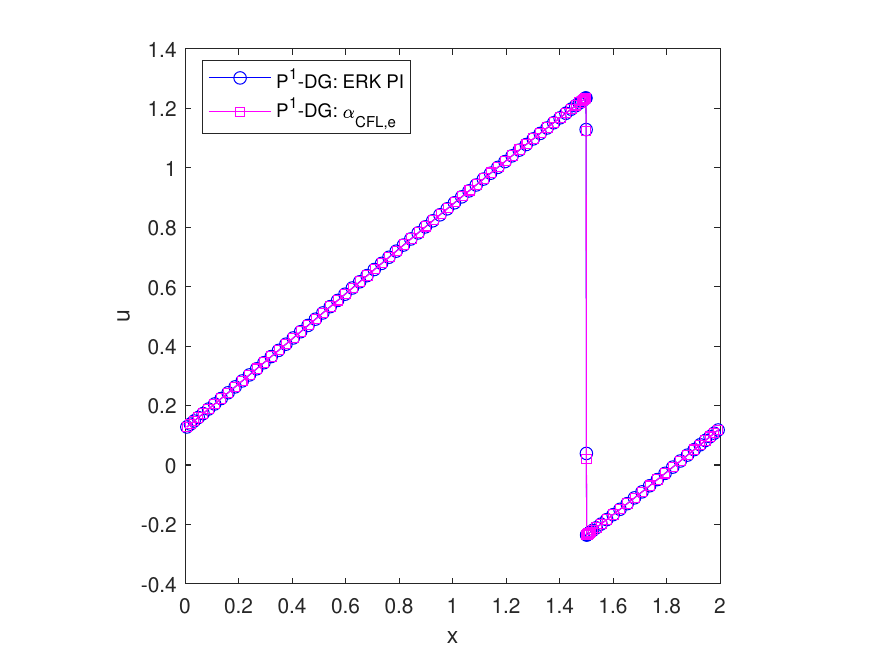}}
\subfigure[$P^2$-DG: solution $u$]{
\includegraphics[width=0.31\textwidth,trim=25 5 45 20,clip]{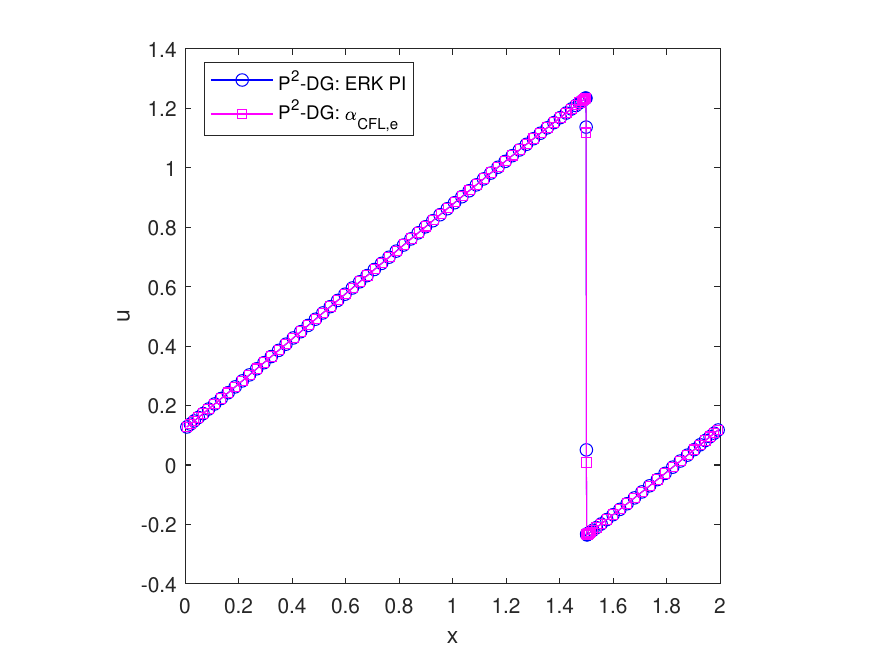}}
\subfigure[$P^3$-DG: solution $u$]{
\includegraphics[width=0.31\textwidth,trim=25 5 45 20,clip]{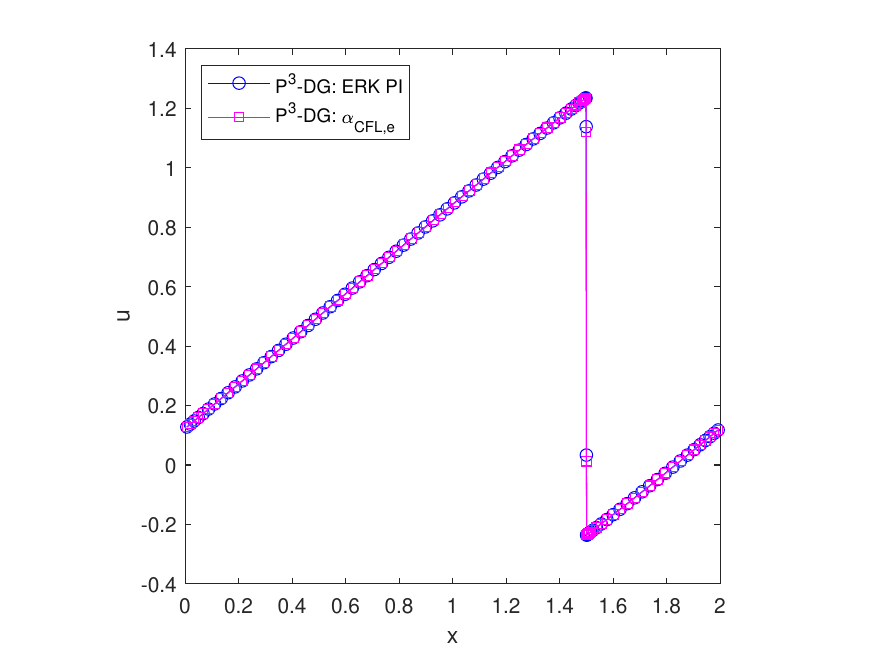}}
\subfigure[$P^1$-DG: $\Delta t$]{
\includegraphics[width=0.31\textwidth,trim=25 5 45 20,clip]{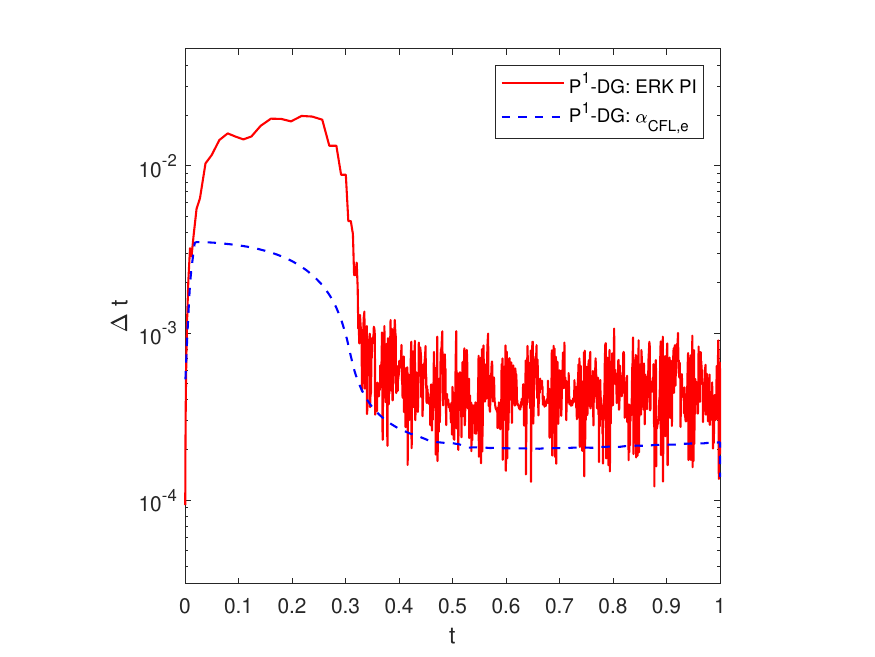}}
\subfigure[$P^2$-DG: $\Delta t$]{
\includegraphics[width=0.31\textwidth,trim=25 5 45 20,clip]{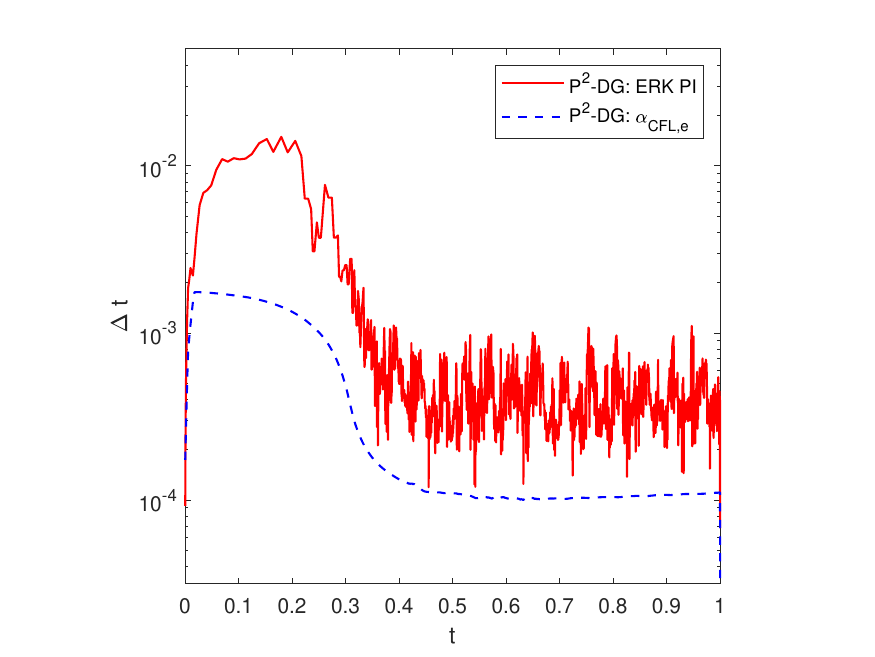}}
\subfigure[$P^3$-DG: $\Delta t$]{
\includegraphics[width=0.31\textwidth,trim=25 5 45 20,clip]{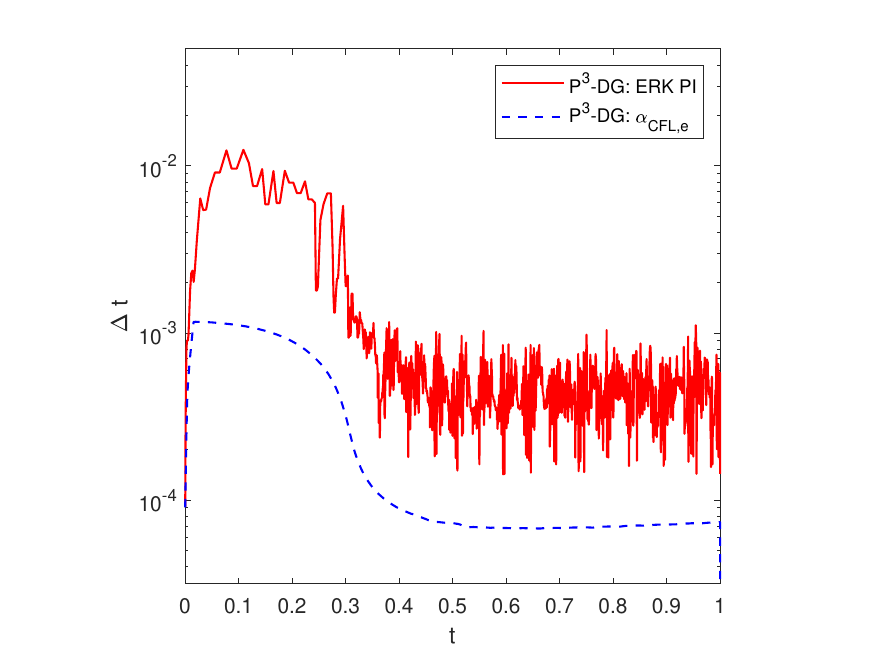}}
\caption{
Example \ref{Burgers-1d}. The mesh trajectories, solution, and time step-size are obtained with the $P^k$-DG method ($k=1,2,3$) and with ERK PI and CFL ($\alpha_{CFL,e}$, $\alpha_{LF,e}$) step-size selection strategies ($N=100$). The mesh trajectories
and solutions are almost indistinguishable for both strategies.}
\label{Fig:Burgers-1d-Pk-ERK}
\end{figure}

\begin{figure}[H]
\centering
\subfigure[$P^1$-DG: $(\alpha_{CFL,e},\alpha_{LF,e})$]{
\includegraphics[width=0.31\textwidth,trim=25 5 45 20,clip]{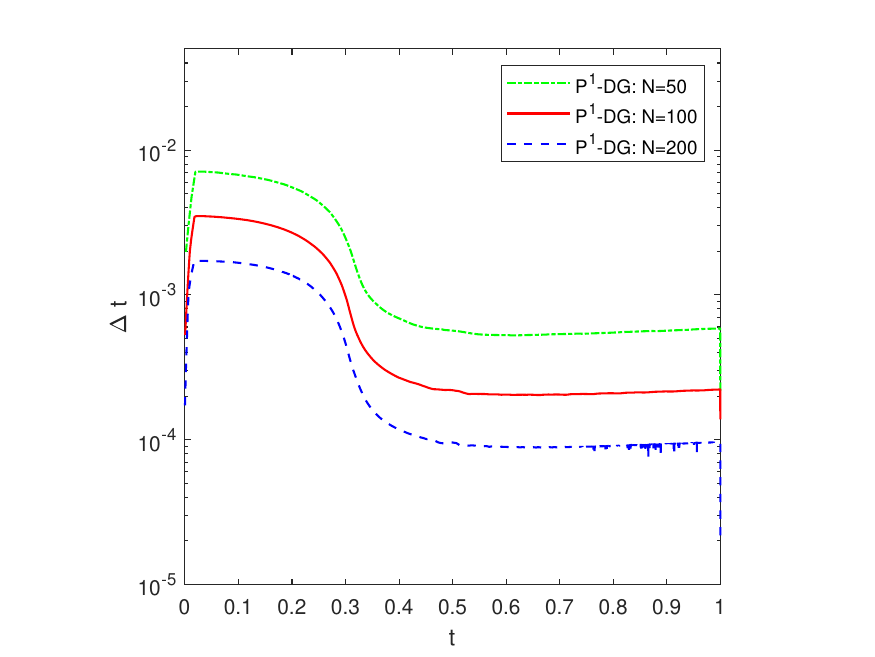}}
\subfigure[$P^2$-DG: $(\alpha_{CFL,e},\alpha_{LF,e})$]{
\includegraphics[width=0.31\textwidth,trim=25 5 45 20,clip]{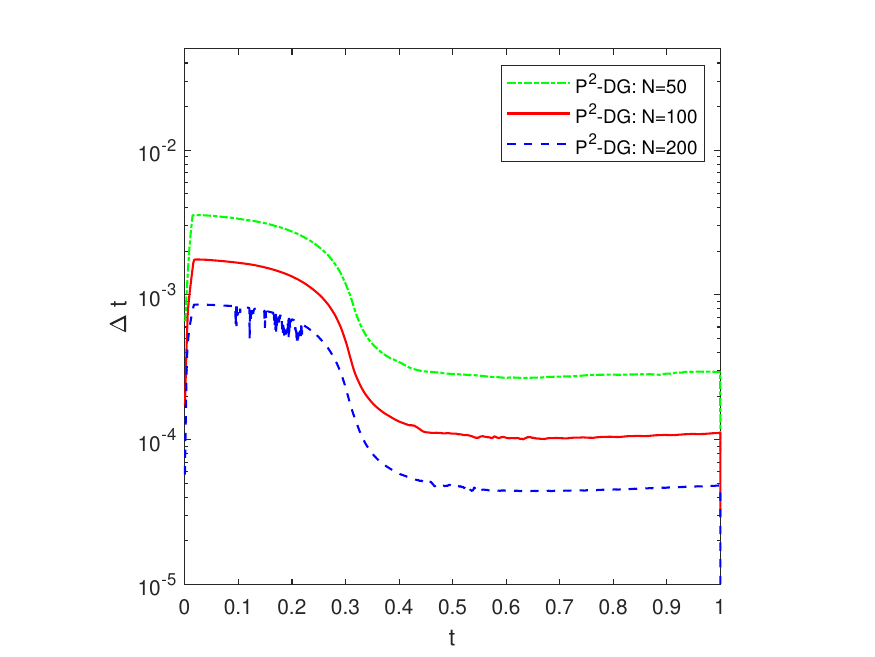}}
\subfigure[$P^3$-DG: $(\alpha_{CFL,e},\alpha_{LF,e})$]{
\includegraphics[width=0.31\textwidth,trim=25 5 45 20,clip]{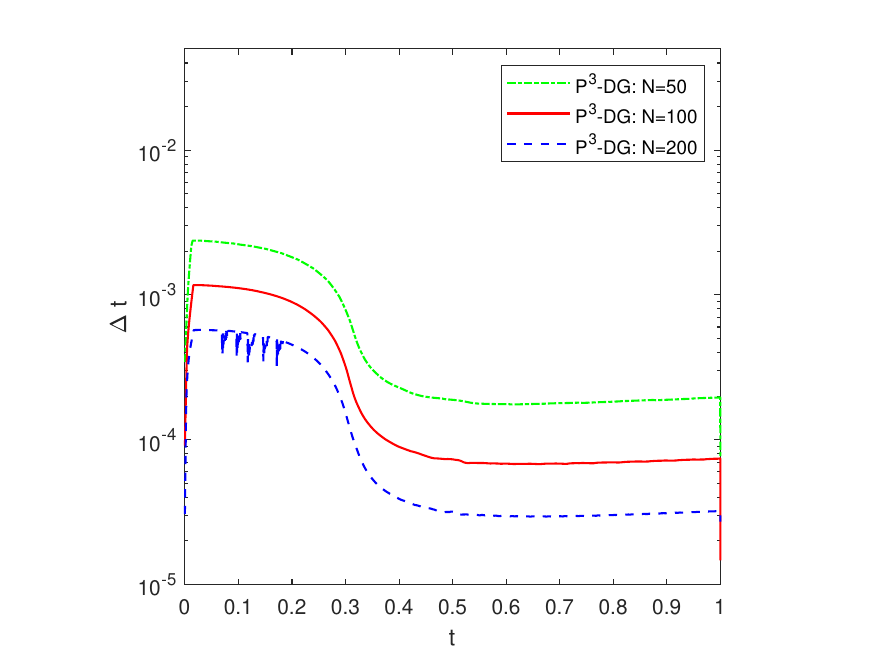}}
\subfigure[$P^1$-DG: ERK PI]{
\includegraphics[width=0.31\textwidth,trim=25 5 45 20,clip]{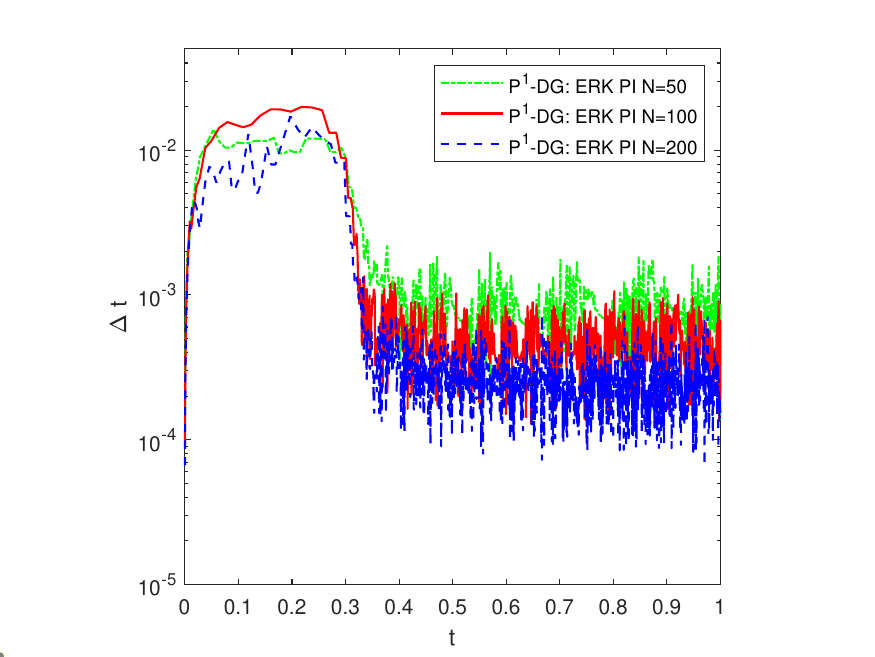}}
\subfigure[$P^2$-DG: ERK PI]{
\includegraphics[width=0.31\textwidth,trim=25 5 45 20,clip]{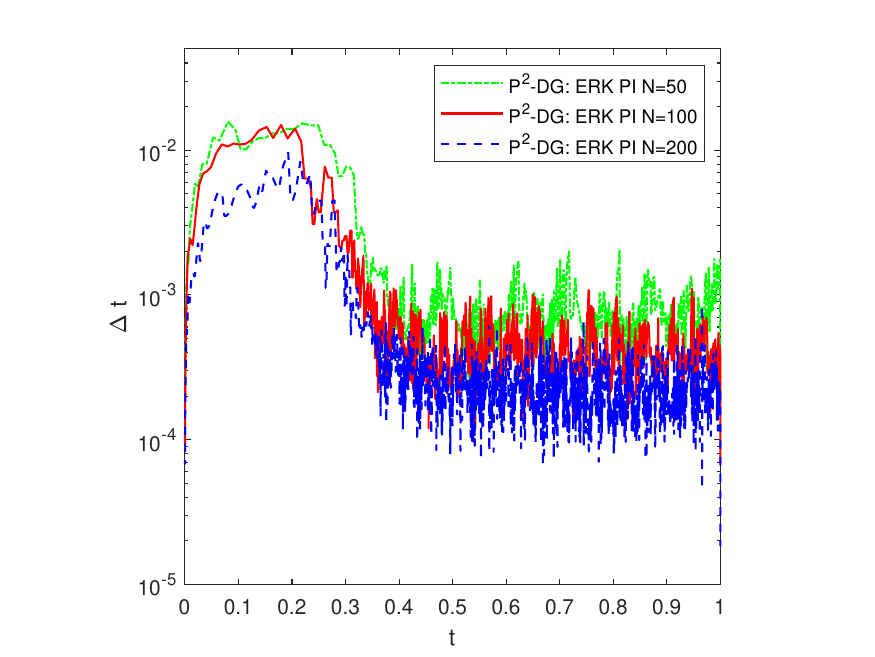}}
\subfigure[$P^3$-DG: ERK PI]{
\includegraphics[width=0.31\textwidth,trim=25 5 45 20,clip]{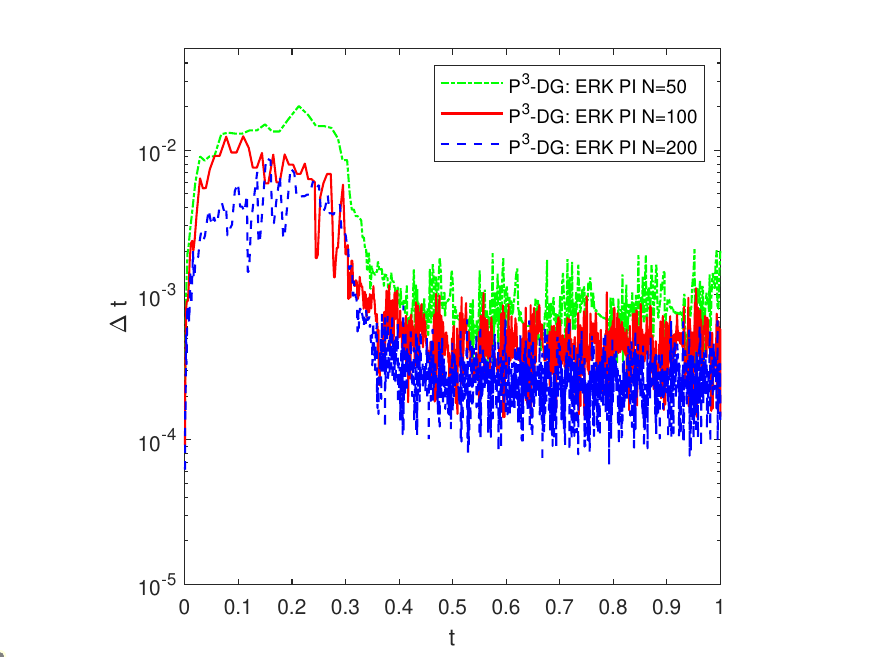}}
\caption{
Example \ref{Burgers-1d}. The time step-size is obtained with the $P^k$-DG method ($k=1,2,3$) and the moving meshes of $N=50,\,100,\, 200$, and with ERK PI and CFL ($\alpha_{CFL,e}$, $\alpha_{LF,e}$) time step-size selection strategies.}
\label{Fig:Burgers-1d-Pk-ERK-dt}
\end{figure}

\begin{figure}[H]
\centering
\includegraphics[width=0.4\textwidth,trim=25 5 45 0,clip]{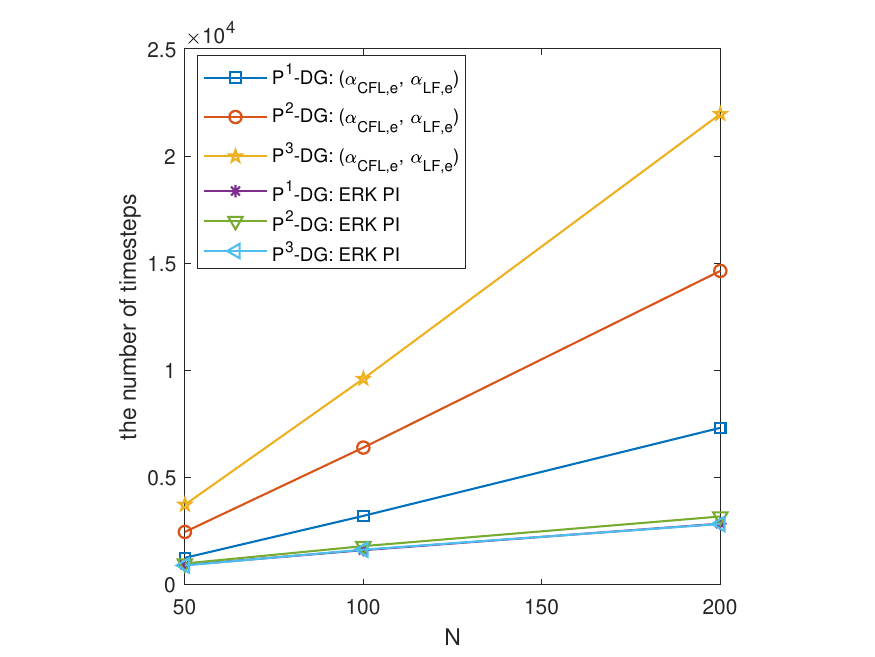}
\caption{
Example \ref{Burgers-1d}. The number of time steps is obtained with the $P^k$-DG method ($k=1,2,3$) and the moving meshes of $N=50,\,100,\, 200$, and with ERK PI and CFL ($\alpha_{CFL,e}$, $\alpha_{LF,e}$) time step-size selection strategies.}
\label{Fig:Burgers-1d-Pk-Nsteps}
\end{figure}

\begin{example}\label{Sod}
(Sod shock tube problem for 1D Euler equations)
\end{example}
We consider the Sod problem of the Euler equations,
\begin{equation}
\label{Euler-Eq}
\frac{\partial}{\partial t} \begin{pmatrix} \rho \\ \rho u \\ E \end{pmatrix}
 + \frac{\partial}{\partial x} \begin{pmatrix} \rho u \\ \rho u^2 + P \\ u (E + P) \end{pmatrix} = 0,
\quad x \in (-5, 5)
\end{equation}
where $\rho$ is the density, $u$ is the velocity, $E$ is the energy density, and $P$ is the pressure.
The equation of the state is $E = P/(\gamma - 1) + \rho u^2/2$ with $\gamma = 1.4$.
The initial conditions are given by
\begin{align}
\label{sod-initial}
(\rho,u,P) =
\begin{cases}
(1,~0,~1), \quad & x\leq0 \\
(0.125,~0,~0.1), \quad & x>0.\\
\end{cases}
\end{align}
The computation is stopped at $T = 2$.
The solution of the problem contains a shock wave, a rarefaction, and a contact discontinuity.

The mesh trajectories, density, and time step-size obtained with the $P^k$-DG method ($k=1,2,3$)
and a moving mesh of $N=200$ are shown in Figs.~\ref{Fig:Sod-Pk}.
We can see that the mesh points are concentrated correctly around the shock wave, rarefaction, and contact discontinuity.
The computation is stable for all three choices of $\alpha$. Similar observations can be made for $\Delta t$
as in the previous example except for the case with $P^2$-DG where ($\alpha_{CFL,h}$, $\alpha_{LF,e}$)
leads to larger $\Delta t$ for $0.2 < t < 0.8$ than ($\alpha_{CFL,e}$, $\alpha_{LF,e}$).
To explain this, we notice that different CFL conditions can lead to different $\Delta t$, which can affect
the mesh adaptation and the time integration of the physical equations. These effects do not seem significant
in the previous example and in other cases in this example. However, they are slightly more significant in this case.
From Fig.~\ref{Fig:Sod-dx}, we can see that ($\alpha_{CFL,h}$, $\alpha_{LF,e}$) results in a slightly larger
minimum mesh spacing, which in turns gives slightly larger $\Delta t$.

Once again, the results show that $\Delta t$ associated with ($\alpha_{CFL,h}$, $\alpha_{LF,e}$)
and ($\alpha_{CFL,h}$, $\alpha_{LF,h}$) has large oscillations than that associated with
($\alpha_{CFL,e}$, $\alpha_{LF,e}$).

It is worth mentioning that we have tried ($\alpha_{CFL,e}$, $\alpha_{LF,h}$), which violates (\ref{alpha-0}), in the computation with moving mesh $P^1$-DG. The computation stops at around $t =0.0011$ when $\Delta t$ becomes machine-precision.
Similar unstable computations have also been observed for this choice for other $P^k$ elements.

The results obtained with ERK PI (SSP-ERK(4,3)) and CFL ($\alpha_{CFL,e}$, $\alpha_{LF,e}$) time step-size selection
strategies are shown in Fig.~\ref{Fig:Sod-Pk-ERK} for comparison purpose. As Example \ref{Burgers-1d}, we can see that
both selection strategies lead to stable computation and almost identical solutions and mesh trajectories.
$\Delta t$ associated with ERK PI is slightly larger than that associated with CFL ($\alpha_{CFL,e}$, $\alpha_{LF,e}$),
and this is especially true and the difference in $\Delta t$ for both is larger for higher-order DG methods.

\begin{figure}[H]
\centering
\subfigure[$P^1$-DG: mesh trajectories]{
\includegraphics[width=0.31\textwidth,trim=25 5 45 18,clip]{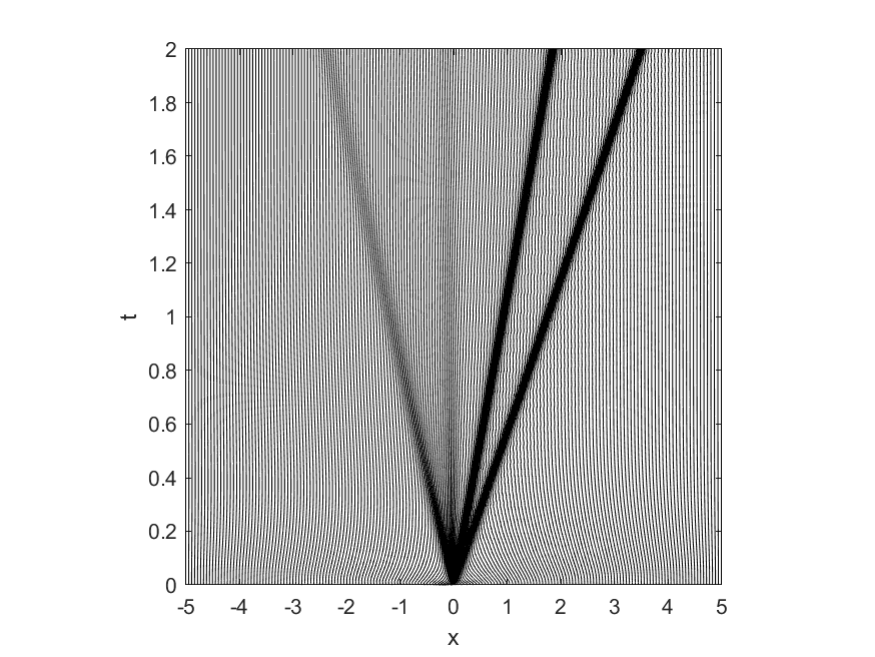}}
\subfigure[$P^2$-DG: mesh trajectories]{
\includegraphics[width=0.31\textwidth,trim=25 5 45 18,clip]{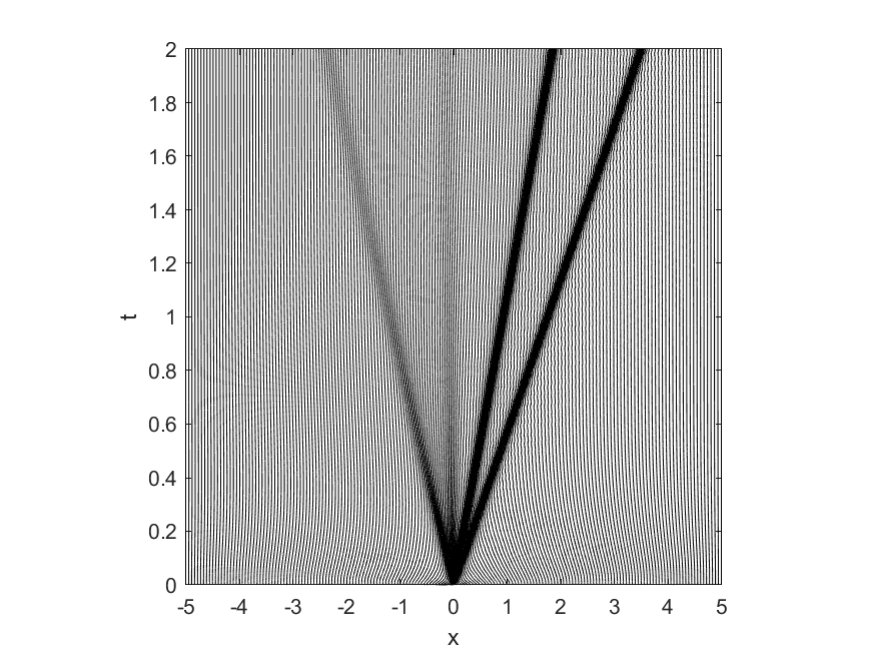}}
\subfigure[$P^3$-DG: mesh trajectories]{
\includegraphics[width=0.31\textwidth,trim=25 5 45 18,clip]{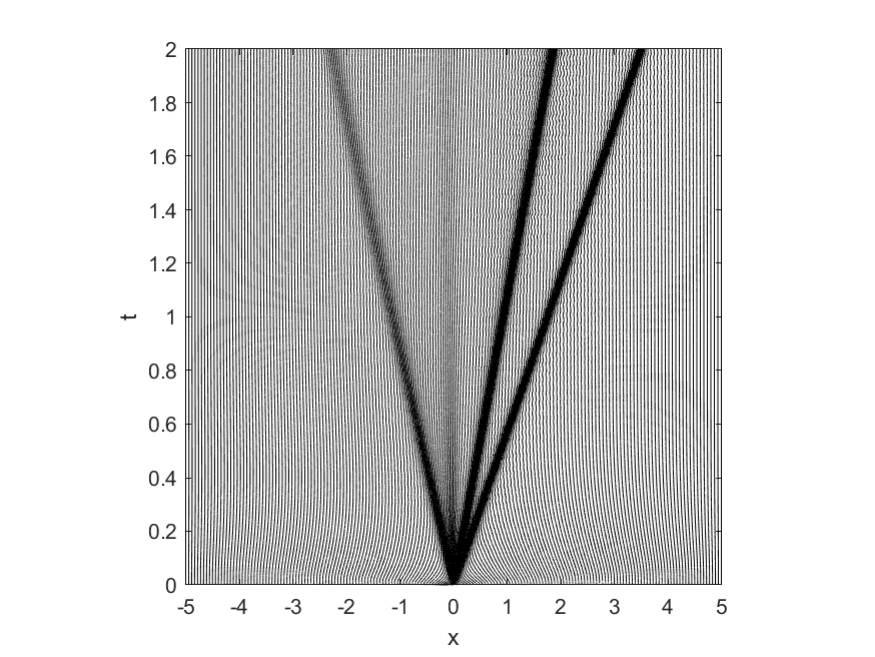}}
\subfigure[$P^1$-DG: density $\rho$]{
\includegraphics[width=0.31\textwidth,trim=25 5 45 18,clip]{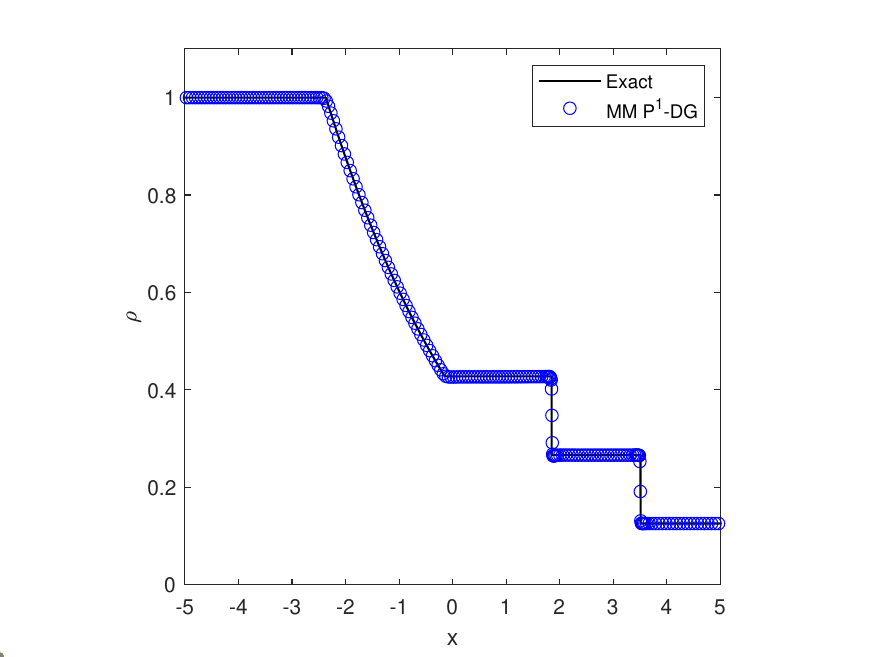}}
\subfigure[$P^2$-DG: density $\rho$]{
\includegraphics[width=0.31\textwidth,trim=25 5 45 18,clip]{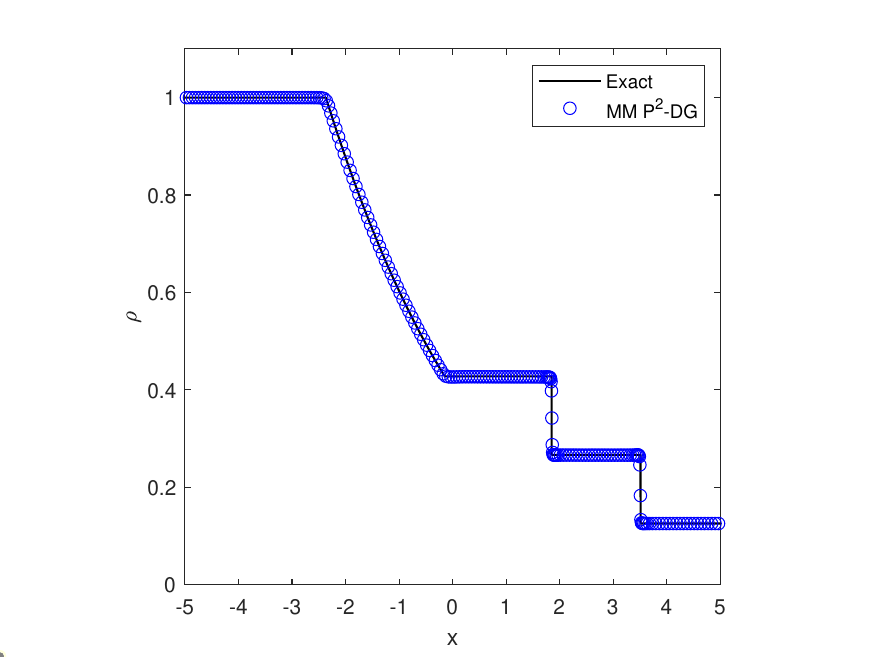}}
\subfigure[$P^3$-DG: density $\rho$]{
\includegraphics[width=0.31\textwidth,trim=25 5 45 18,clip]{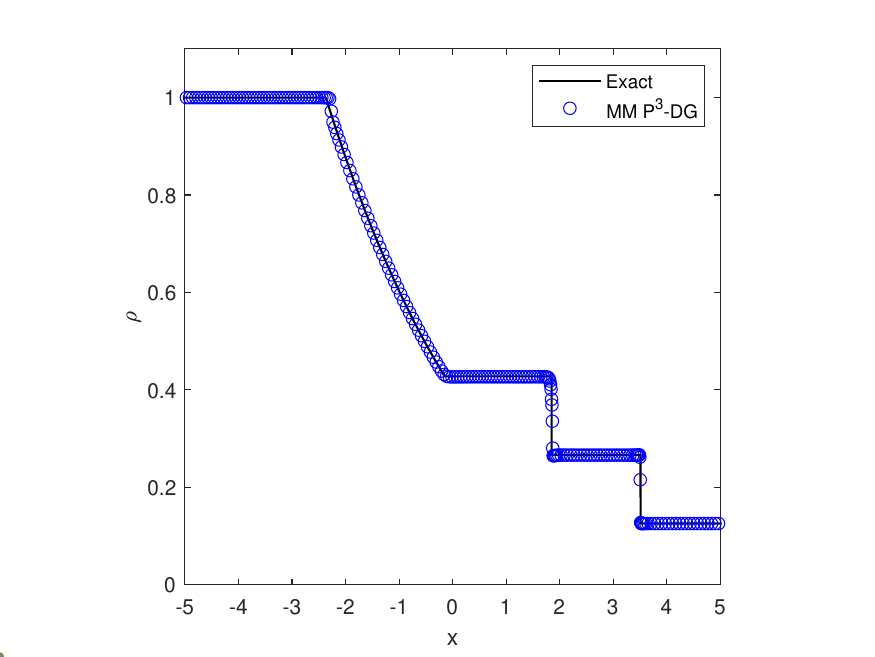}}
\subfigure[$P^1$-DG: $\Delta t$]{
\includegraphics[width=0.31\textwidth,trim=25 5 45 18,clip]{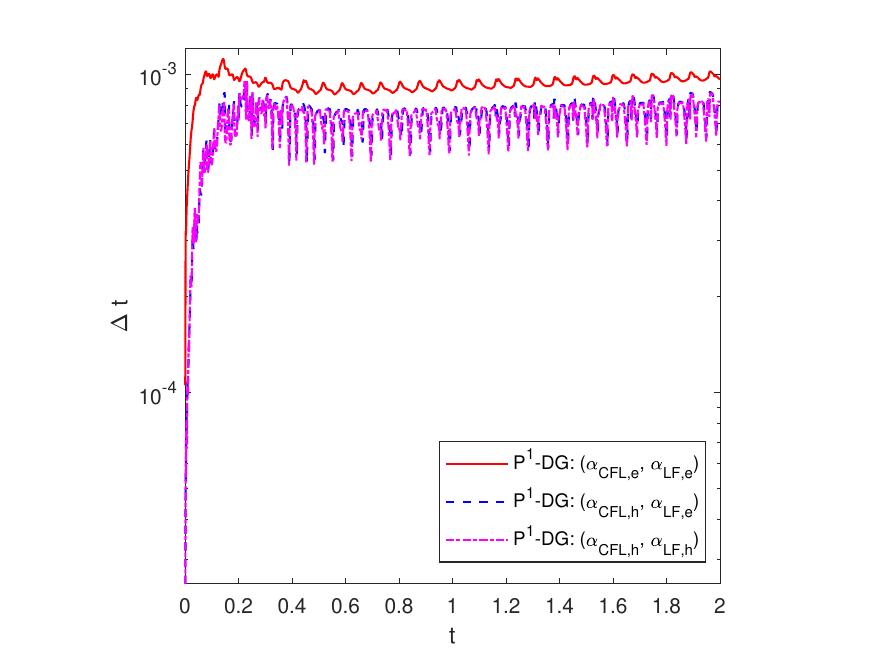}}
\subfigure[$P^2$-DG: $\Delta t$]{
\includegraphics[width=0.31\textwidth,trim=25 5 45 18,clip]{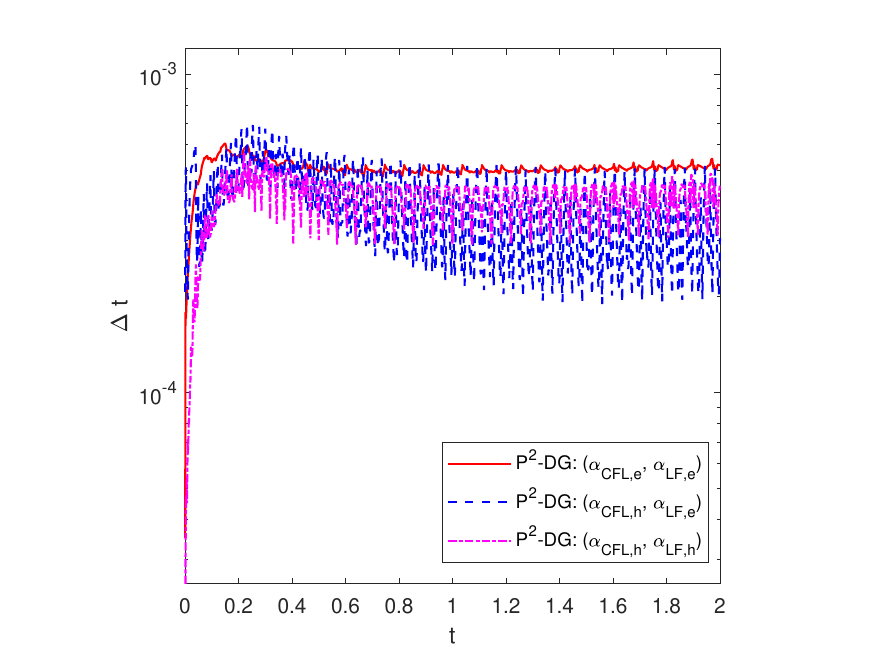}}
\subfigure[$P^3$-DG: $\Delta t$]{
\includegraphics[width=0.31\textwidth,trim=25 5 45 18,clip]{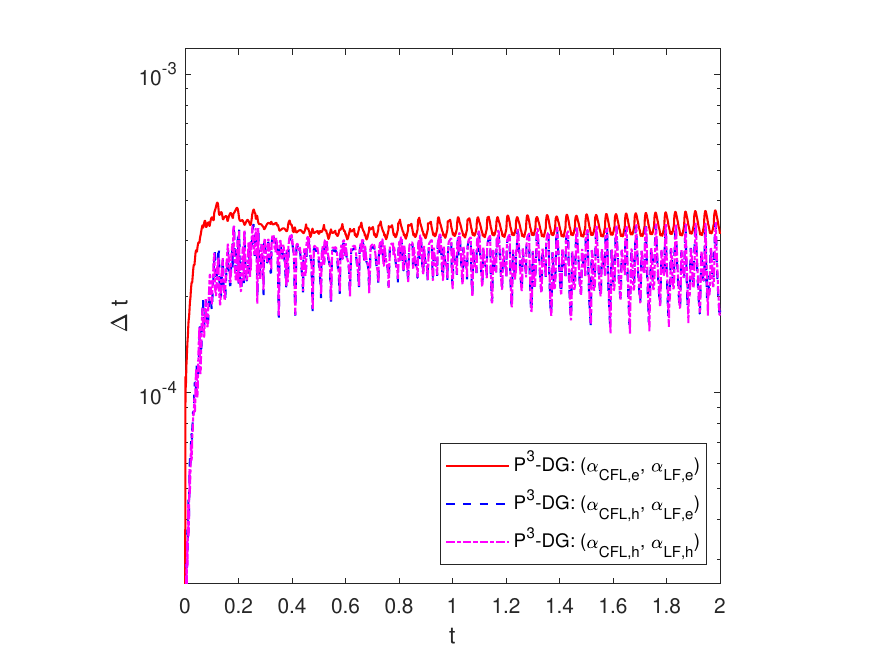}}
\caption{Example \ref{Sod}.
The mesh trajectories, solution, and time step-size are obtained with the $P^k$-DG method ($k=1,2,3$) and a moving mesh of $N=200$.}
\label{Fig:Sod-Pk}
\end{figure}

\begin{figure}[H]
\centering
\subfigure[$P^1$-DG]{
\includegraphics[width=0.31\textwidth,trim=25 5 45 10,clip]{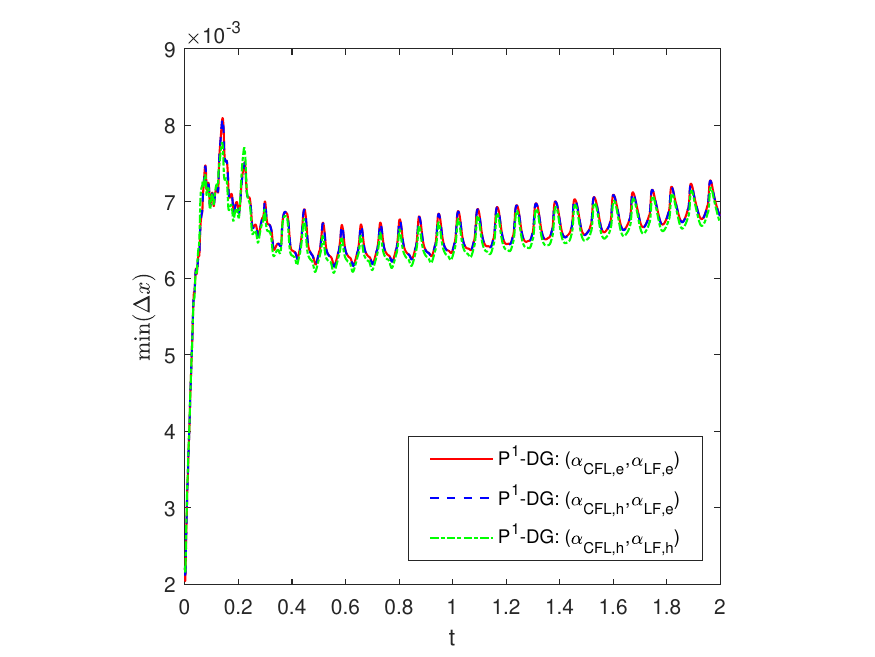}}
\subfigure[$P^2$-DG]{
\includegraphics[width=0.31\textwidth,trim=25 5 45 10,clip]{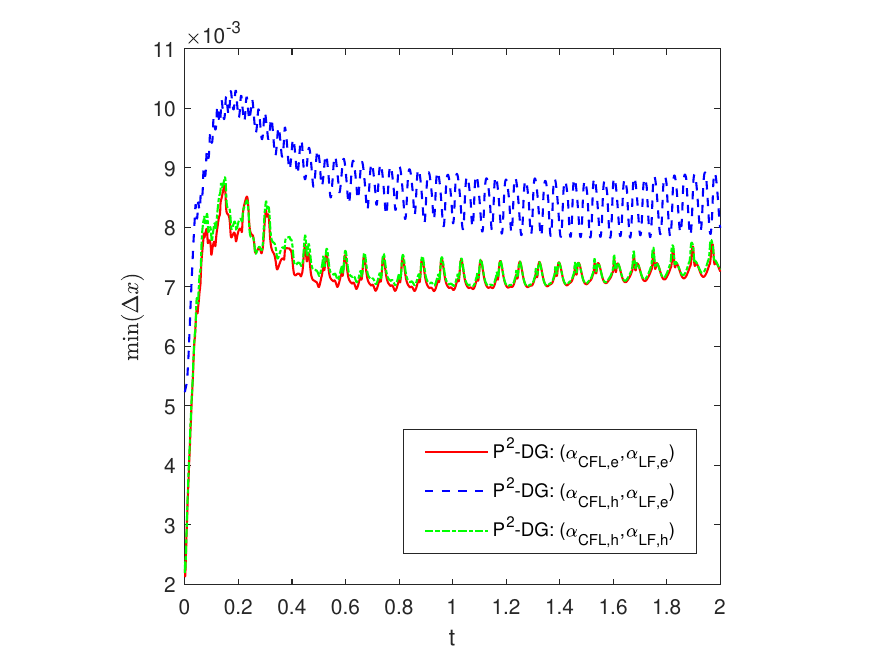}}
\subfigure[$P^3$-DG]{
\includegraphics[width=0.31\textwidth,trim=25 5 45 10,clip]{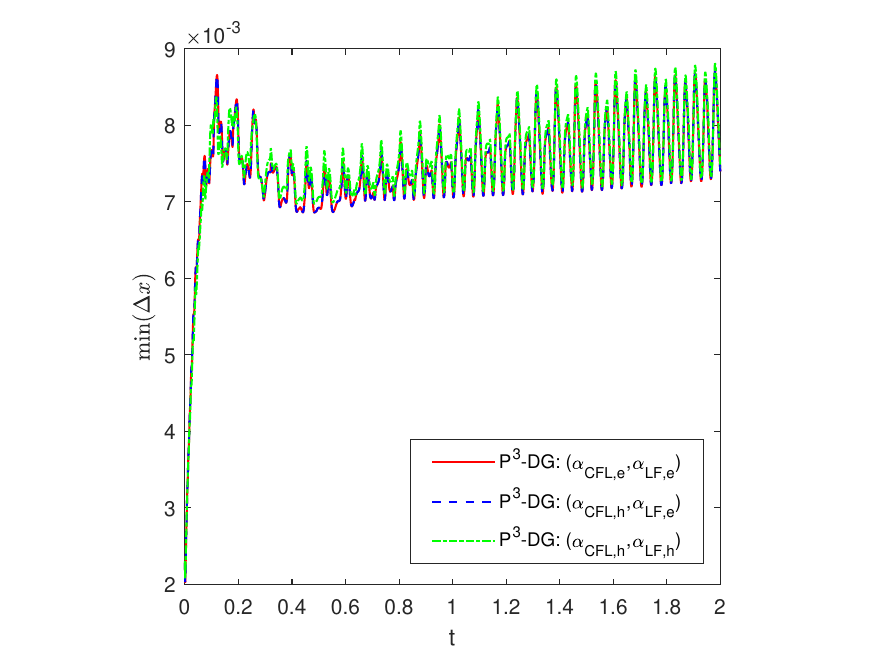}}
\caption{Example \ref{Sod}. The evolution of the minimal mesh spacing $\min(\Delta x)$
with the time obtained with the $P^k$-DG method ($k=1,2,3$) and a moving mesh of $N=200$.}
\label{Fig:Sod-dx}
\end{figure}

\begin{figure}[H]
\centering
\subfigure[$P^1$-DG: mesh trajectories]{
\includegraphics[width=0.31\textwidth,trim=25 5 45 18,clip]{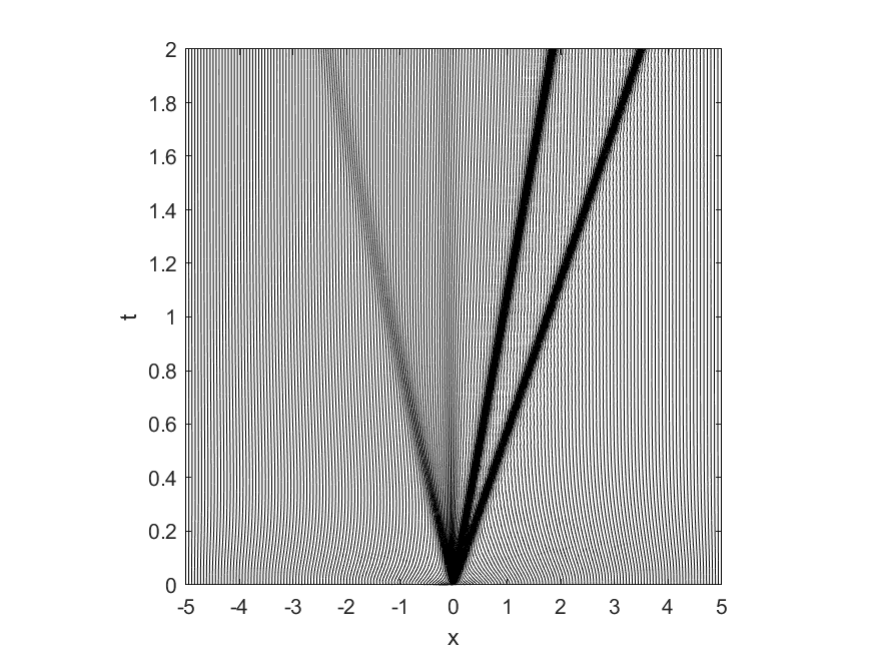}}
\subfigure[$P^2$-DG: mesh trajectories]{
\includegraphics[width=0.31\textwidth,trim=25 5 45 18,clip]{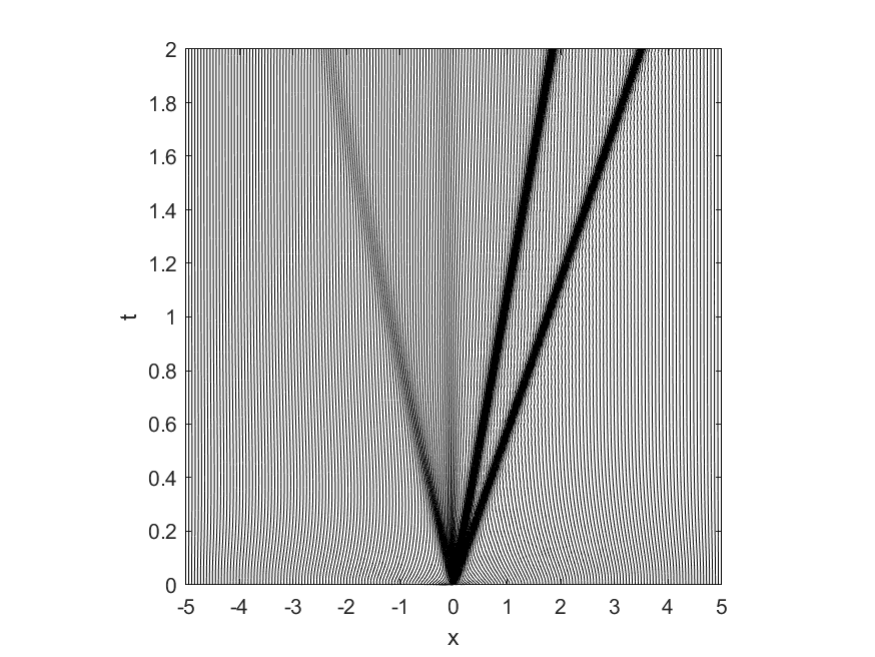}}
\subfigure[$P^3$-DG: mesh trajectories]{
\includegraphics[width=0.31\textwidth,trim=25 5 45 18,clip]{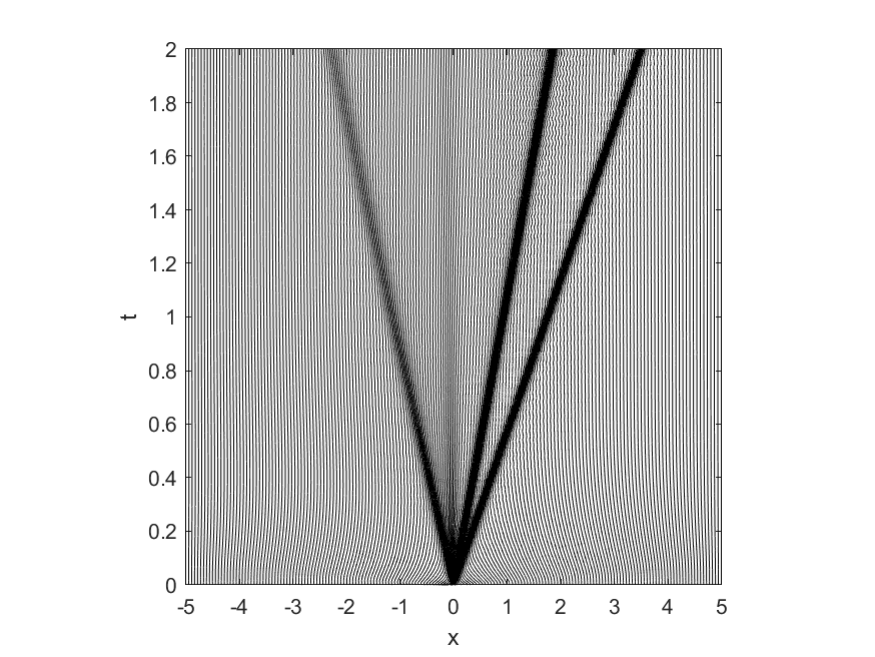}}
\subfigure[$P^1$-DG: density $\rho$]{
\includegraphics[width=0.31\textwidth,trim=25 5 45 18,clip]{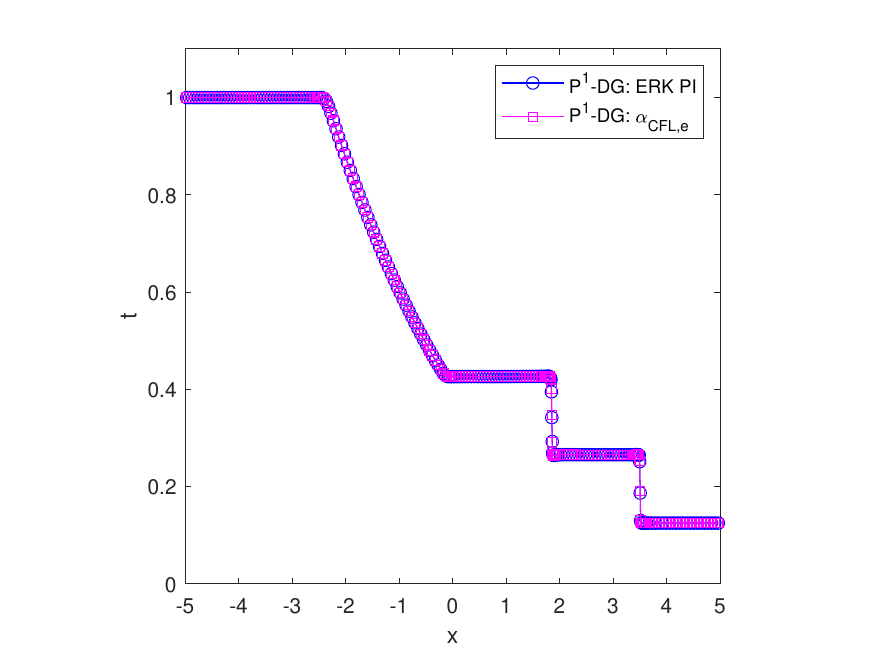}}
\subfigure[$P^2$-DG: density $\rho$]{
\includegraphics[width=0.31\textwidth,trim=25 5 45 18,clip]{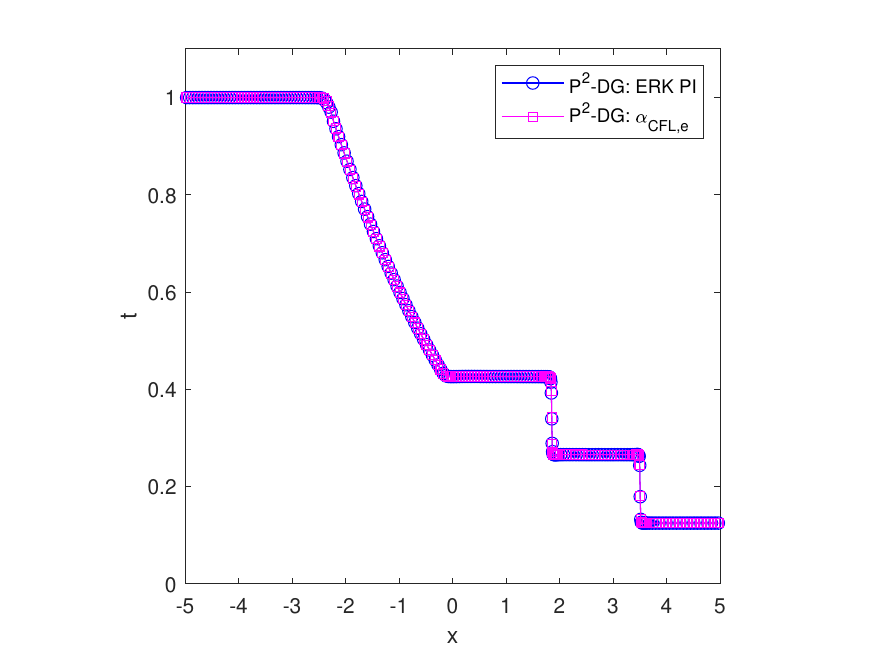}}
\subfigure[$P^3$-DG: density $\rho$]{
\includegraphics[width=0.31\textwidth,trim=25 5 45 18,clip]{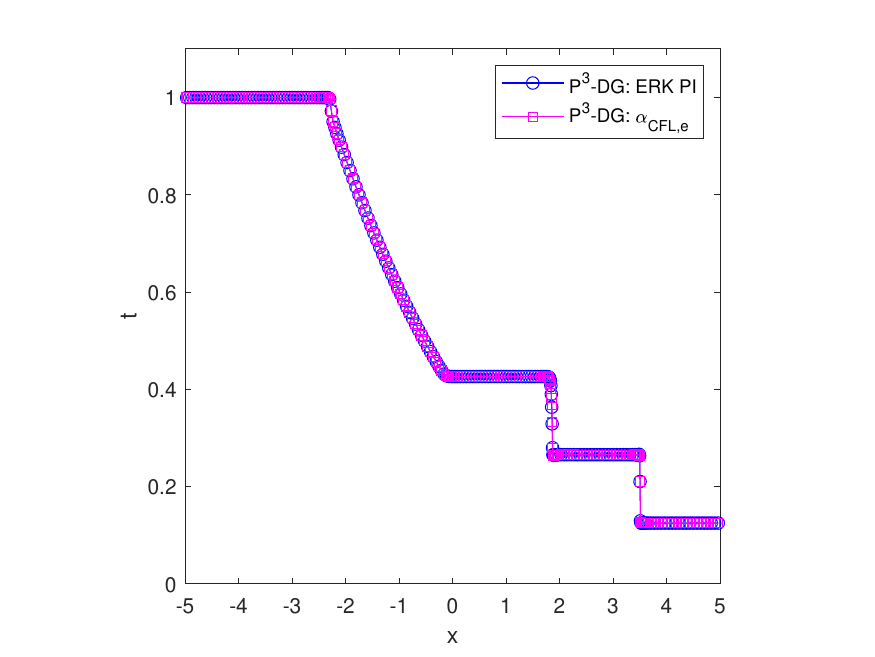}}
\subfigure[$P^1$-DG: $\Delta t$]{
\includegraphics[width=0.31\textwidth,trim=25 5 45 18,clip]{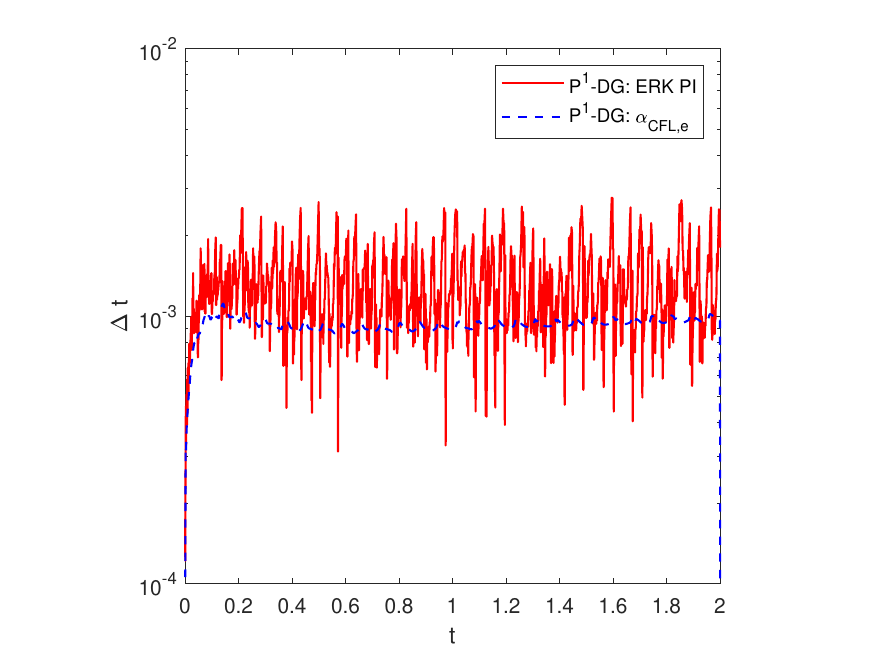}}
\subfigure[$P^2$-DG: $\Delta t$]{
\includegraphics[width=0.31\textwidth,trim=25 5 45 18,clip]{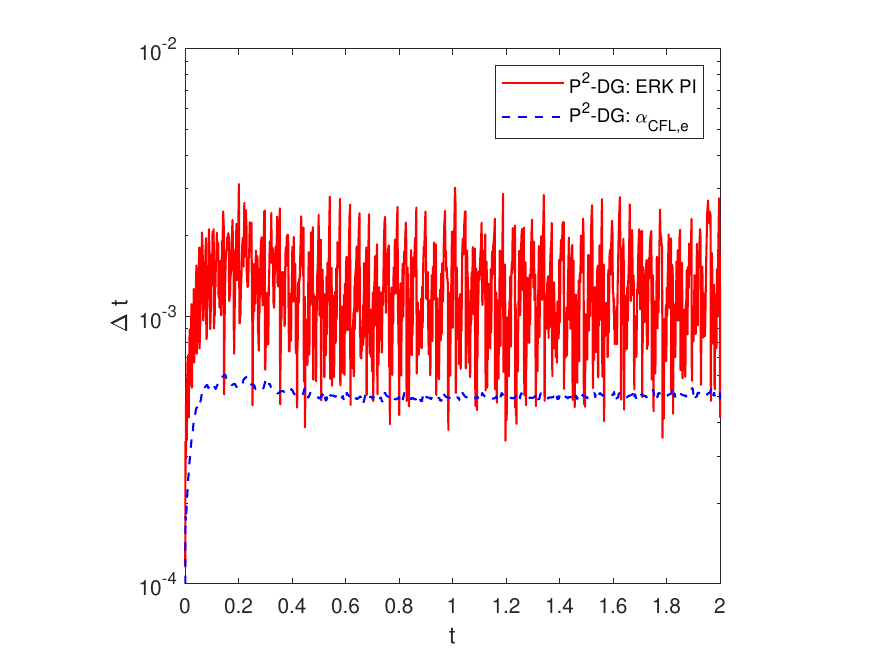}}
\subfigure[$P^3$-DG: $\Delta t$]{
\includegraphics[width=0.31\textwidth,trim=25 5 45 18,clip]{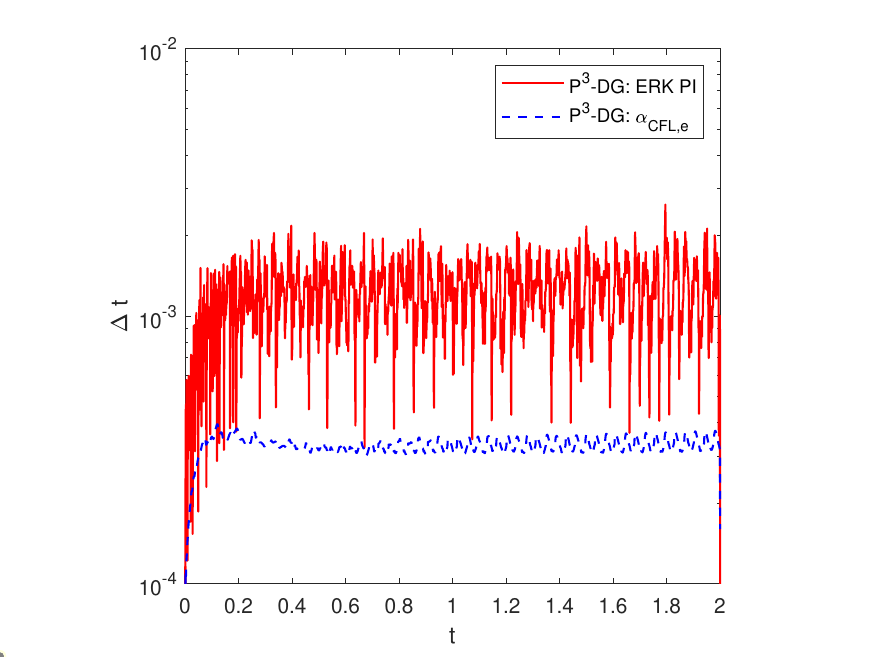}}
\caption{
Example \ref{Sod}. The mesh trajectories, solution, and time step-size are obtained with the $P^k$-DG method ($k=1,2,3$)
and with ERK PI and CFL  ($\alpha_{CFL,e}$, $\alpha_{LF,e}$) step-size selection strategies ($N=200$). The mesh trajectories
and solutions are almost indistinguishable for both strategies.}
\label{Fig:Sod-Pk-ERK}
\end{figure}

\begin{example}\label{Lax}
(Lax problem for 1D Euler equations)
\end{example}
In this example, we consider the Lax problem of the Euler equations (\ref{Euler-Eq})
with the following initial conditions
\begin{align}\label{Lax-initial}
(\rho,u,P) =
\begin{cases}
(0.445,~0.698,~3.528), \quad & x\leq0\\
(0.5, ~0, ~0.571), \quad & x>0 .
\end{cases}
\end{align}
The final time is $T = 1.3$.
The mesh trajectories, density, and time step-size $\Delta t$ obtained with the $P^k$-DG method ($k=1,2,3$) and a moving mesh of $N=200$ are shown in Fig.~\ref{Fig:Lax-Pk}.
In this example (and following examples), we show results only for two choices of $\alpha$,
($\alpha_{CFL,e}$, $\alpha_{LF,e}$) and ($\alpha_{CFL,h}$, $\alpha_{LF,e}$),
since ($\alpha_{CFL,h}$, $\alpha_{LF,e}$) and ($\alpha_{CFL,h}$, $\alpha_{LF,h}$)
produce almost identical results. The results show that the computation is stable and $\Delta t$ associated with ($\alpha_{CFL,e}$, $\alpha_{LF,e}$) is slightly larger and has small oscillations than that associated
with ($\alpha_{CFL,h}$, $\alpha_{LF,e}$).
\begin{figure}[H]
\centering
\subfigure[$P^1$-DG: mesh trajectories]{
\includegraphics[width=0.31\textwidth,trim=25 5 40 18,clip]{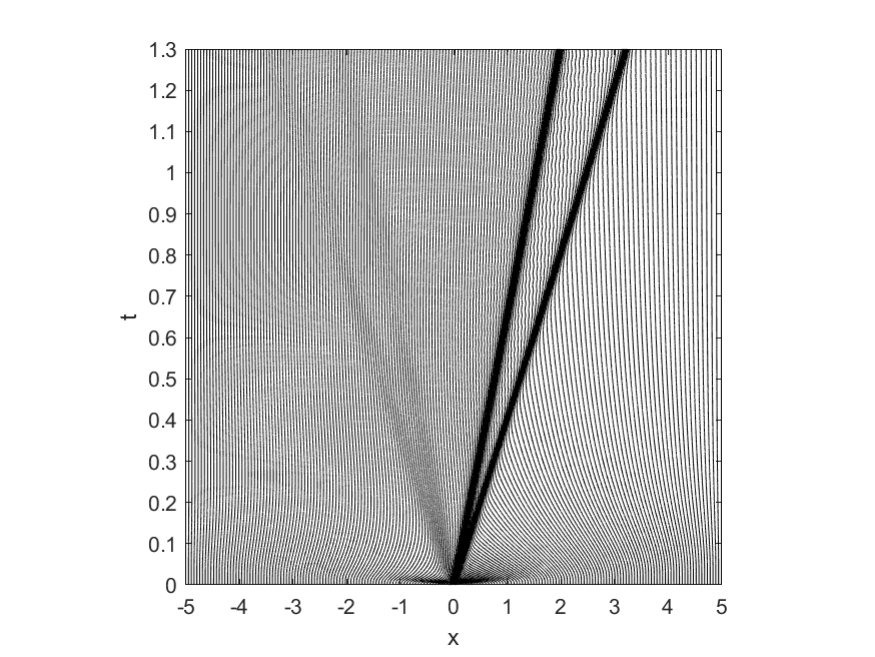}}
\subfigure[$P^2$-DG: mesh trajectories]{
\includegraphics[width=0.31\textwidth,trim=25 5 40 18,clip]{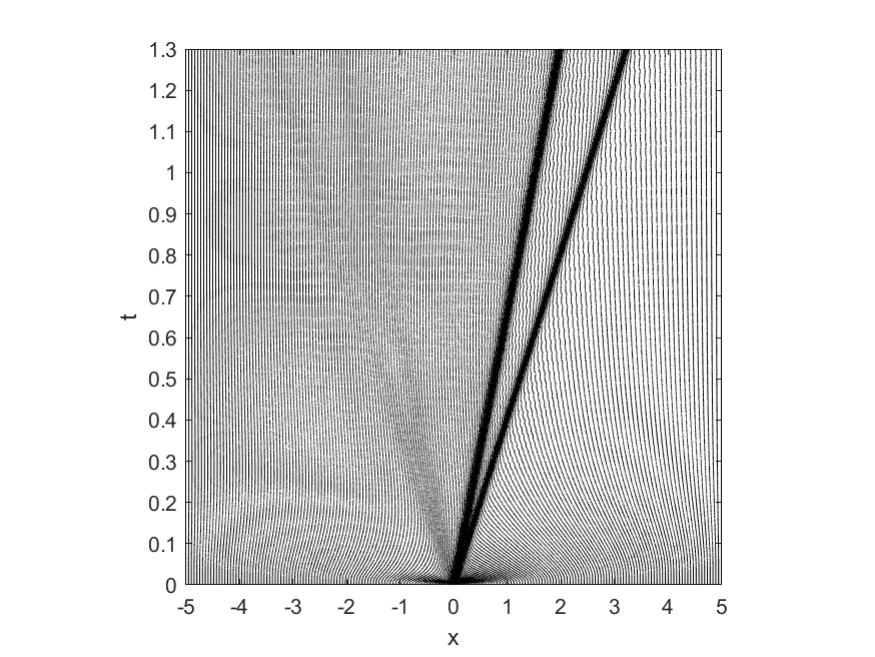}}
\subfigure[$P^3$-DG: mesh trajectories]{
\includegraphics[width=0.31\textwidth,trim=25 5 40 18,clip]{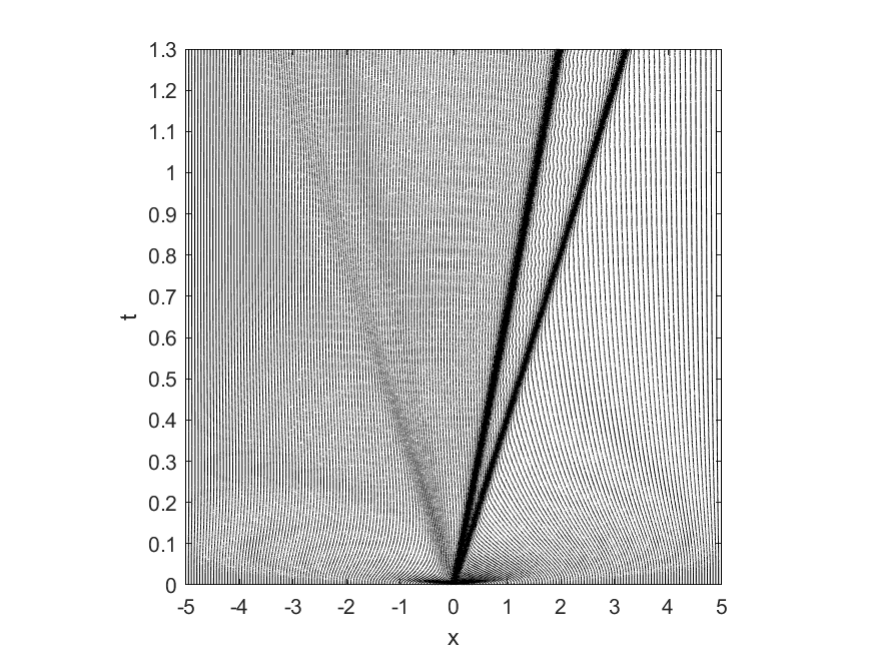}}
\subfigure[$P^1$-DG: density $\rho$]{
\includegraphics[width=0.31\textwidth,trim=25 5 40 18,clip]{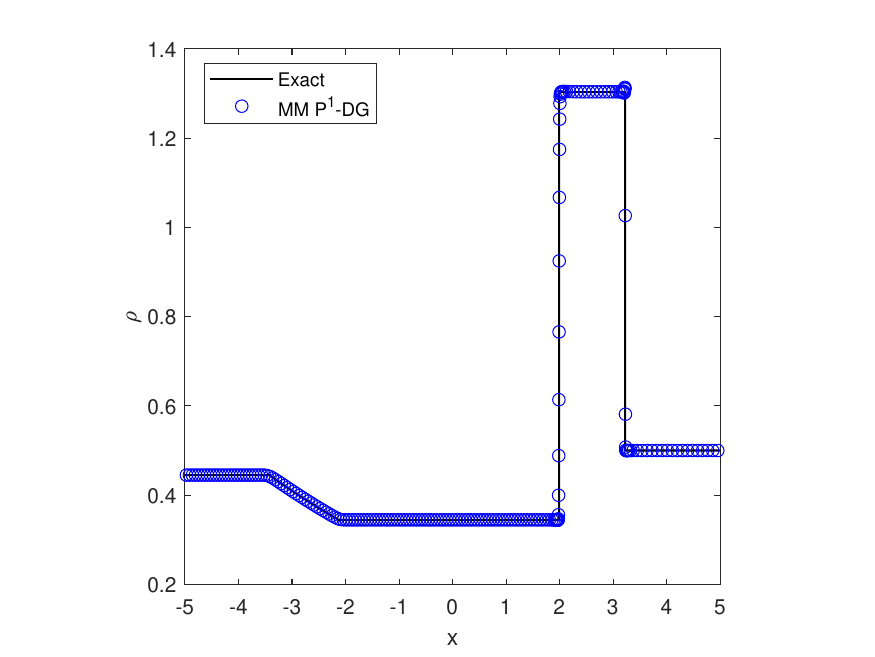}}
\subfigure[$P^2$-DG: density $\rho$]{
\includegraphics[width=0.31\textwidth,trim=25 5 40 18,clip]{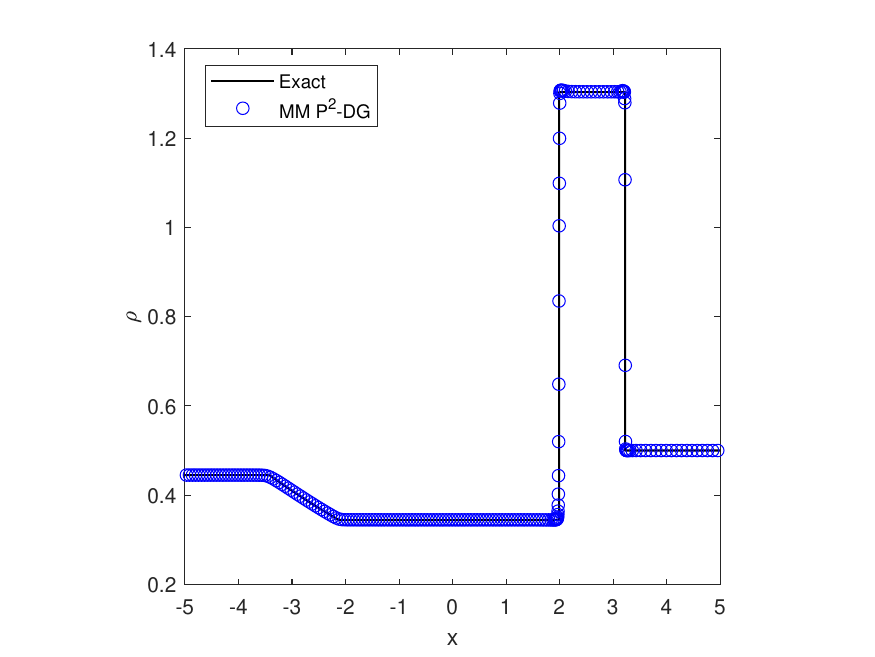}}
\subfigure[$P^3$-DG: density $\rho$]{
\includegraphics[width=0.31\textwidth,trim=25 5 40 18,clip]{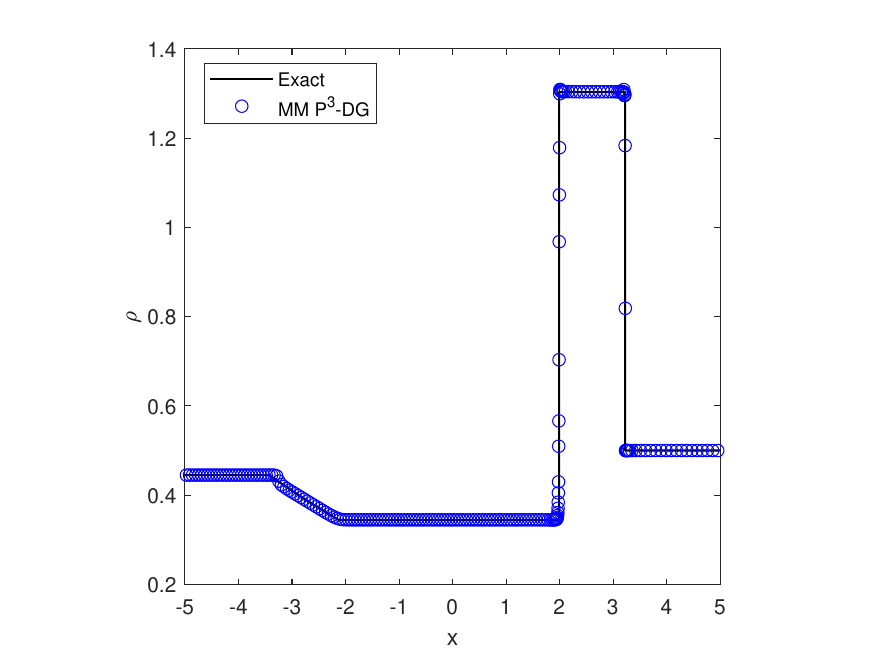}}
\subfigure[$P^1$-DG: $\Delta t$]{
\includegraphics[width=0.31\textwidth,trim=25 5 40 18,clip]{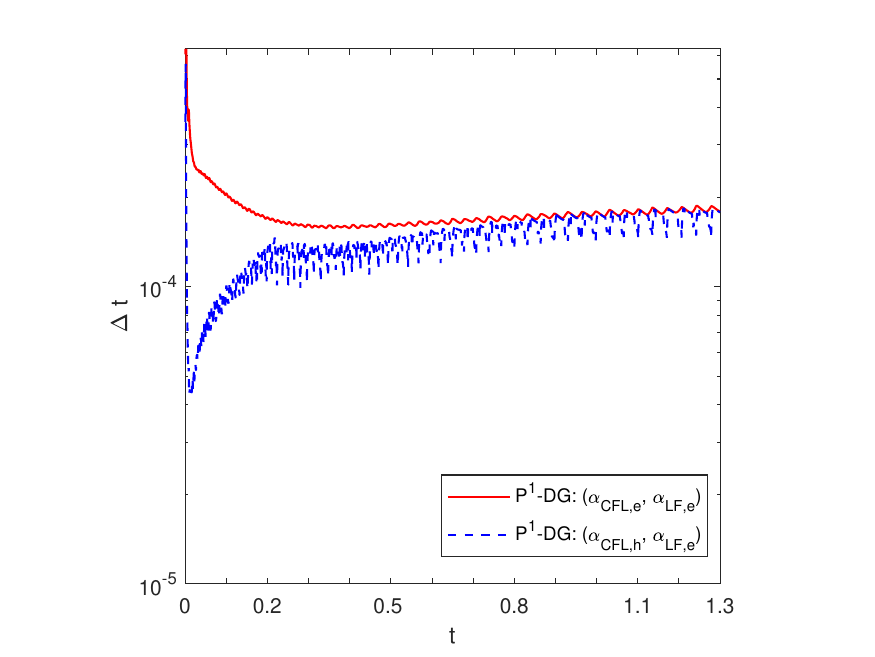}}
\subfigure[$P^2$-DG: $\Delta t$]{
\includegraphics[width=0.31\textwidth,trim=25 5 40 18,clip]{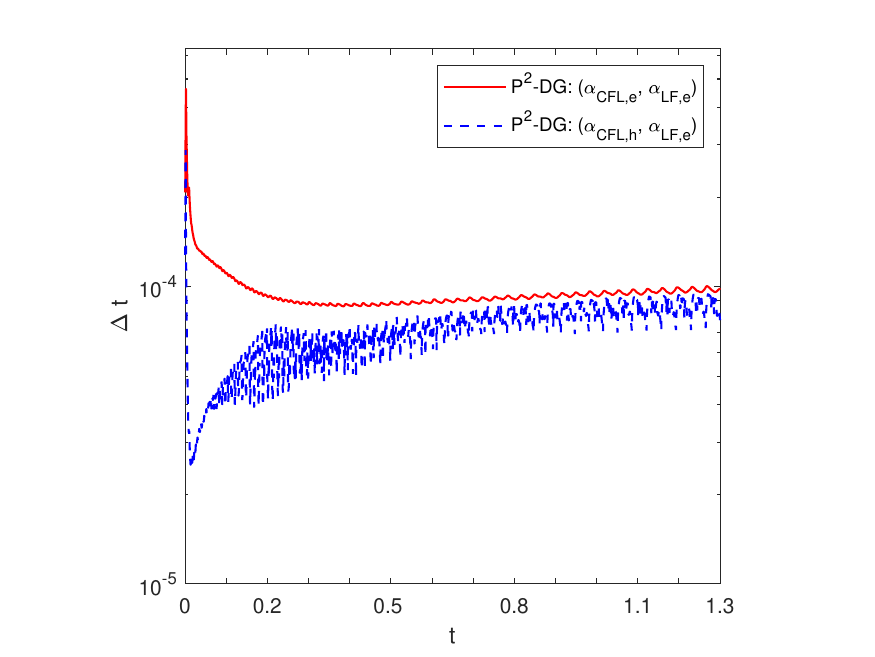}}
\subfigure[$P^3$-DG: $\Delta t$]{
\includegraphics[width=0.31\textwidth,trim=25 5 40 18,clip]{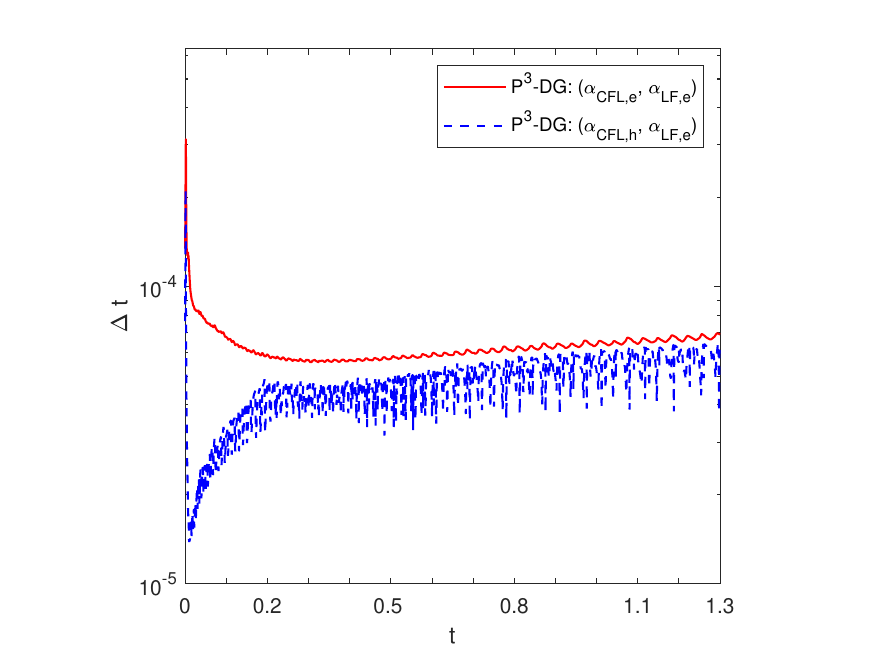}}
\caption{Example \ref{Lax}. The mesh trajectories, solution, and time step-size are obtained with the $P^k$-DG method ($k=1,2,3$) and a moving mesh of $N=200$.}
\label{Fig:Lax-Pk}
\end{figure}

\begin{example}\label{Burgers-2d}
(2D Burgers' equation)
\end{example}
We now consider Burgers' equation in two dimensions,
\begin{equation}
\label{burgers-2d}
u_t +  \Big(\frac{u^2}{2}\Big)_x +\Big(\frac{u^2}{2}\Big)_y  = 0,\quad (x,y)\in(0,2)\times(0,2)
\end{equation}
subject to the initial condition $u(x,y,0) = e^{(-c(x^2+y^2))}$, $c=- \ln(10^{-16})$
and periodic boundary conditions. The final time is $T = 2$.

The mesh, solution, and time step-size (associated with three selection strategies, ERK PI,
CFL ($\alpha_{CFL,e},\alpha_{LF,e}$), and CFL ($\alpha_{CFL,h}$, $\alpha_{LF,e}$)) obtained
with the moving mesh $P^k$-DG method ($k=1,2,3$)
and $N=30\times30\times4$ are plotted in Fig.~\ref{Fig:Burgers-2d-Pk}.
The results show that the computation is stable
and the mesh points are concentrated in regions with sharp jumps in the solution. Moreover, $\Delta t$
associated with ($\alpha_{CFL,e}$, $\alpha_{LF,e}$) is larger and has smaller oscillations than that
associated with ($\alpha_{CFL,h}$, $\alpha_{LF,e}$) while $\Delta t$
associated with ERK PI is significantly larger than that associated with ($\alpha_{CFL,e}$, $\alpha_{LF,e}$)
(and this is especially true for higher-order DG).
 The time step-size for the meshes with $N=10\times10\times4,~20\times20\times4$, and $30\times30\times4$,
is shown in Fig.~\ref{Fig:Burgers-2d-Pk-dtN}.
The number of time steps obtained with $N=10\times10\times4,~20\times20\times4$, and $30\times30\times4$, is shown in Fig~\ref{Fig:Burgers-2d-Pk-Nsteps}.
The number of time steps for CFL-based time-size selection is significantly larger than that associated with ERK PI. Moreover, it increases more significantly as $N$ increases. On the other hand, the number of time steps associated with ERK PI increases almost linearly and more mildly as $N$ increases and does not have much difference for $k = 1, 2, 3$.

\begin{figure}[H]
\centering
\subfigure[$P^1$-DG: mesh]{
\includegraphics[width=0.31\textwidth,trim=25 0 40 10,clip]{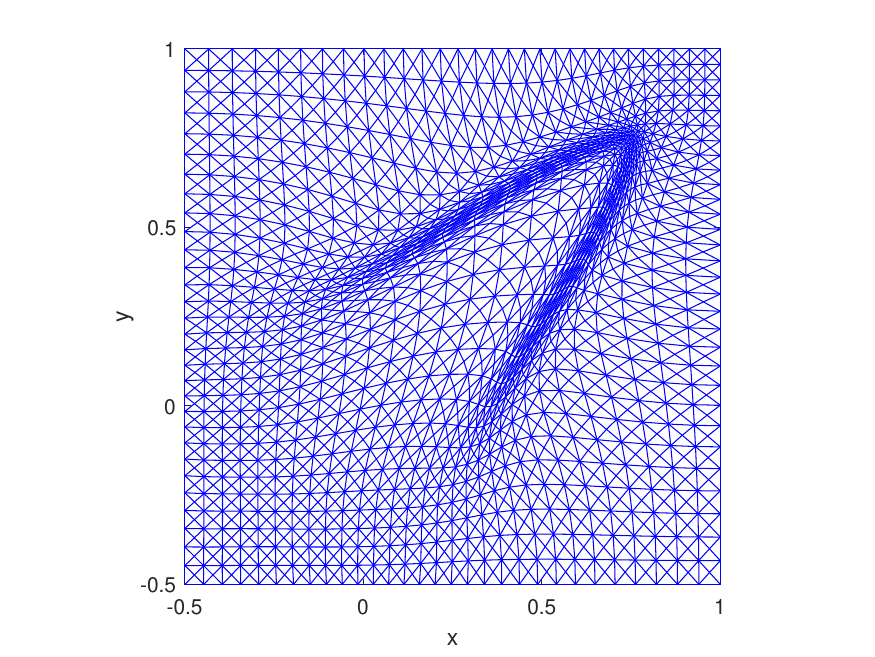}}
\subfigure[$P^2$-DG: mesh]{
\includegraphics[width=0.31\textwidth,trim=25 0 40 10,clip]{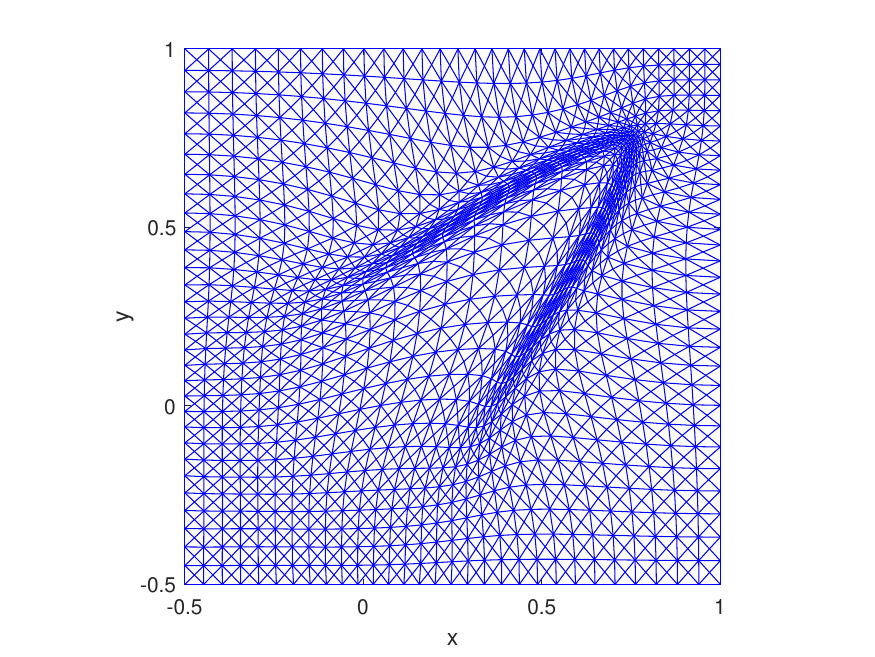}}
\subfigure[$P^3$-DG: mesh]{
\includegraphics[width=0.31\textwidth,trim=25 0 40 10,clip]{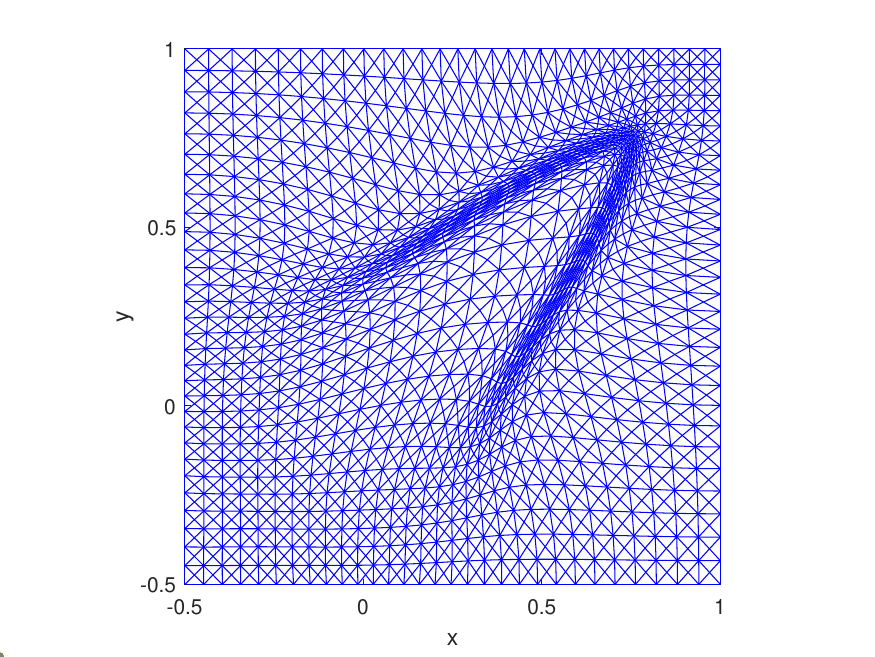}}
\subfigure[$P^1$-DG: solution $u$]{
\includegraphics[width=0.31\textwidth,trim=25 0 40 10,clip]{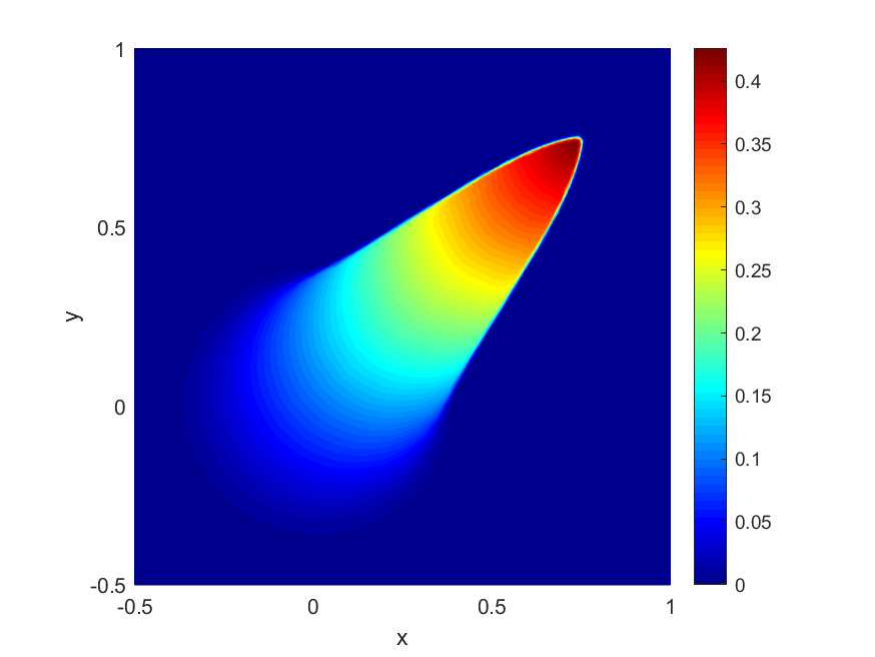}}
\subfigure[$P^2$-DG: solution $u$]{
\includegraphics[width=0.31\textwidth,trim=25 0 40 10,clip]{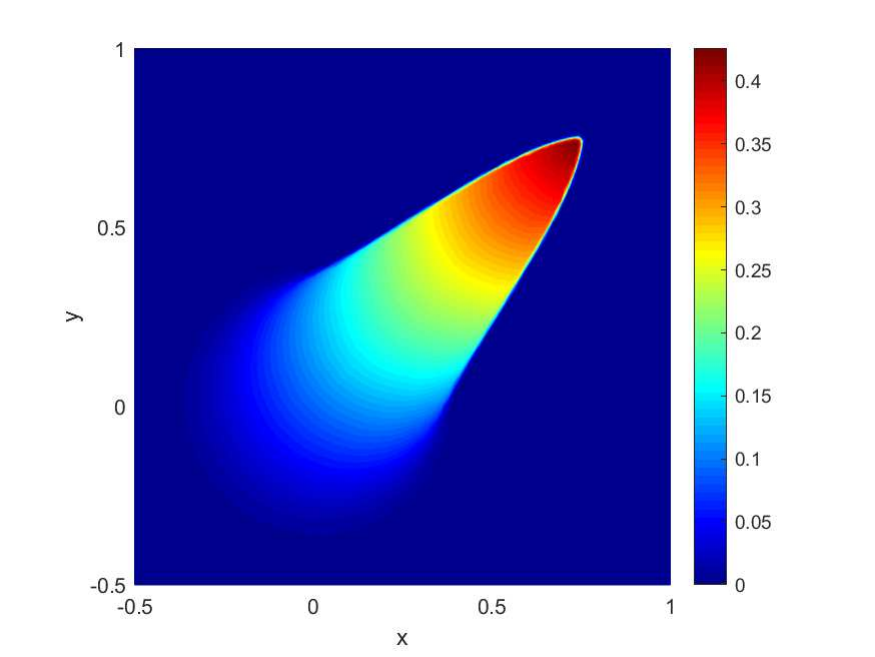}}
\subfigure[$P^3$-DG: solution $u$]{
\includegraphics[width=0.31\textwidth,trim=25 0 40 10,clip]{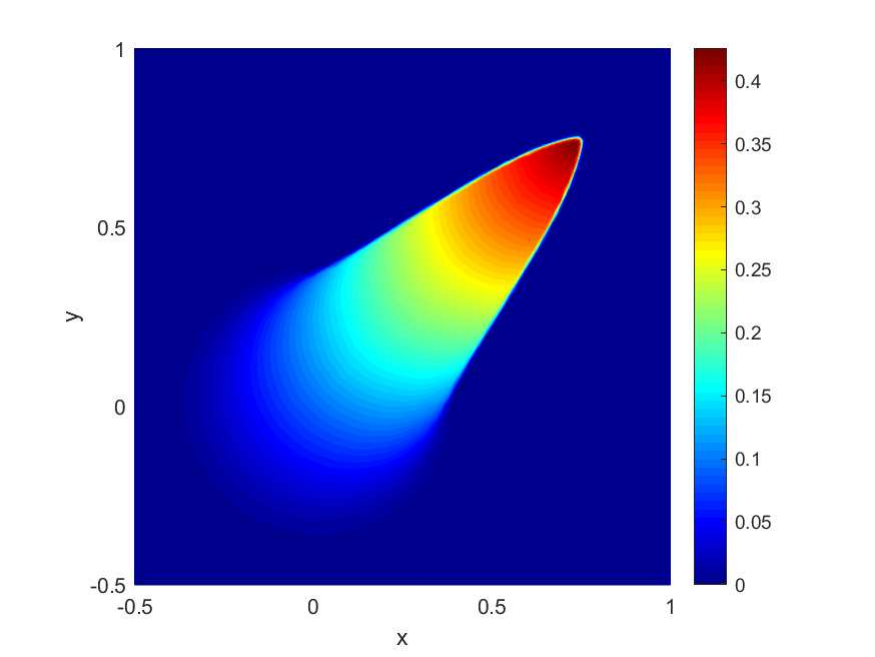}}
\subfigure[$P^1$-DG: $\Delta t$]{
\includegraphics[width=0.31\textwidth,trim=25 0 40 10,clip]{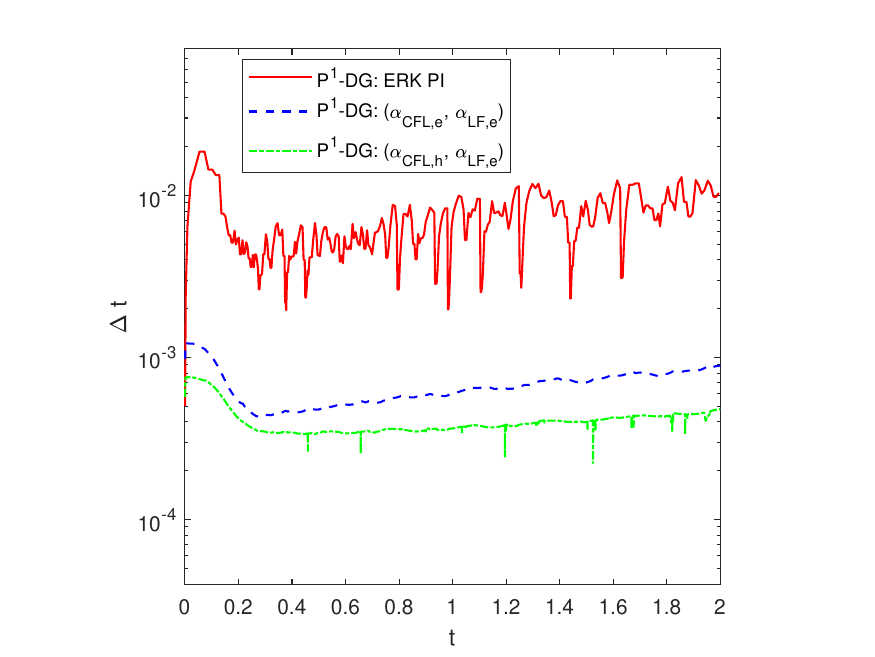}}
\subfigure[$P^2$-DG: $\Delta t$]{
\includegraphics[width=0.31\textwidth,trim=25 0 40 10,clip]{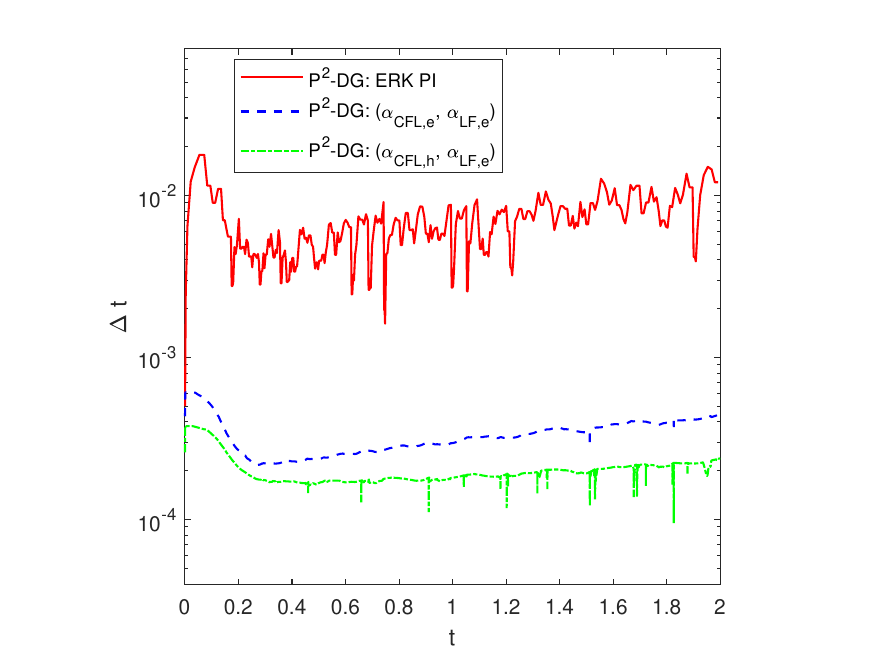}}
\subfigure[$P^3$-DG: $\Delta t$]{
\includegraphics[width=0.31\textwidth,trim=25 0 40 10,clip]{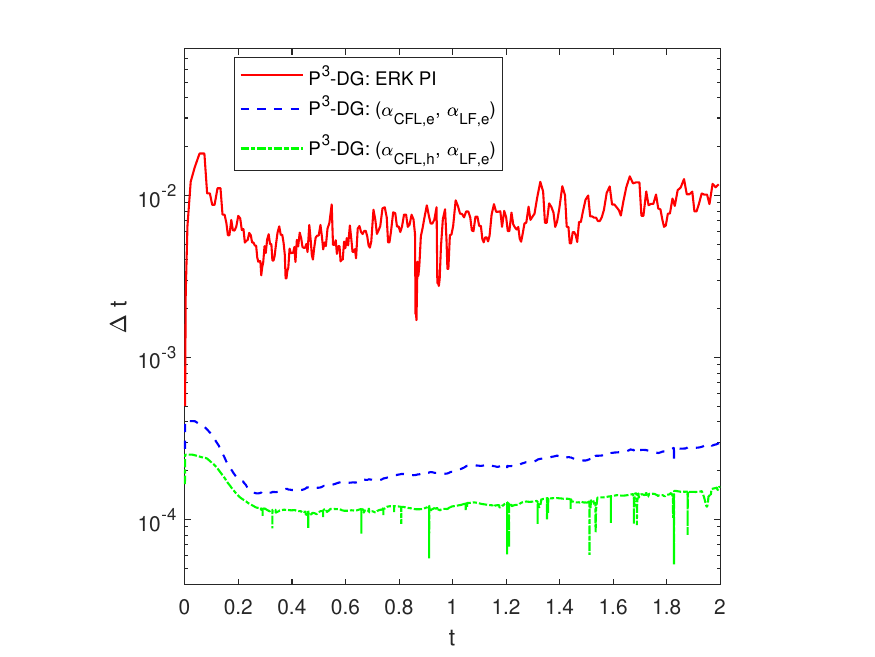}}
\caption{Example \ref{Burgers-2d}. The time step-size $\Delta t$ results from three selection strategies, ERK PI,
CFL ($\alpha_{CFL,h}$,$\alpha_{LF,e}$), and CFL ($\alpha_{CFL,e}$,$\alpha_{LF,e}$).
The $P^1$, $P^2$, and $P^3$-DG method with a moving mesh of $N=30\times30\times4$ is used.}
\label{Fig:Burgers-2d-Pk}
\end{figure}

\begin{figure}[H]
\centering
\subfigure[$P^1$-DG: $(\alpha_{CFL,e},\alpha_{LF,e})$]{
\includegraphics[width=0.31\textwidth,trim=25 0 40 10,clip]{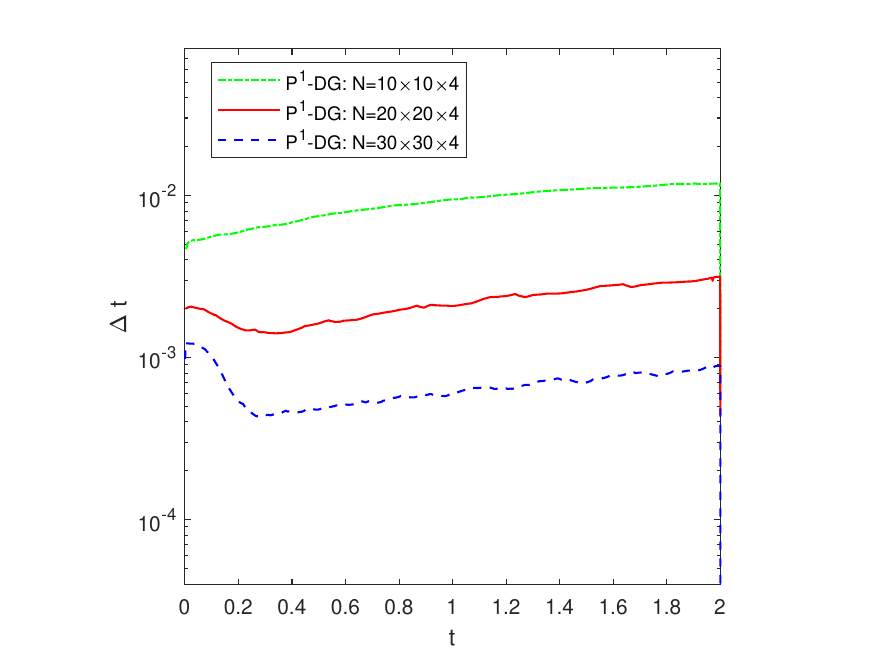}}
\subfigure[$P^2$-DG: $(\alpha_{CFL,e},\alpha_{LF,e})$]{
\includegraphics[width=0.31\textwidth,trim=25 0 40 10,clip]{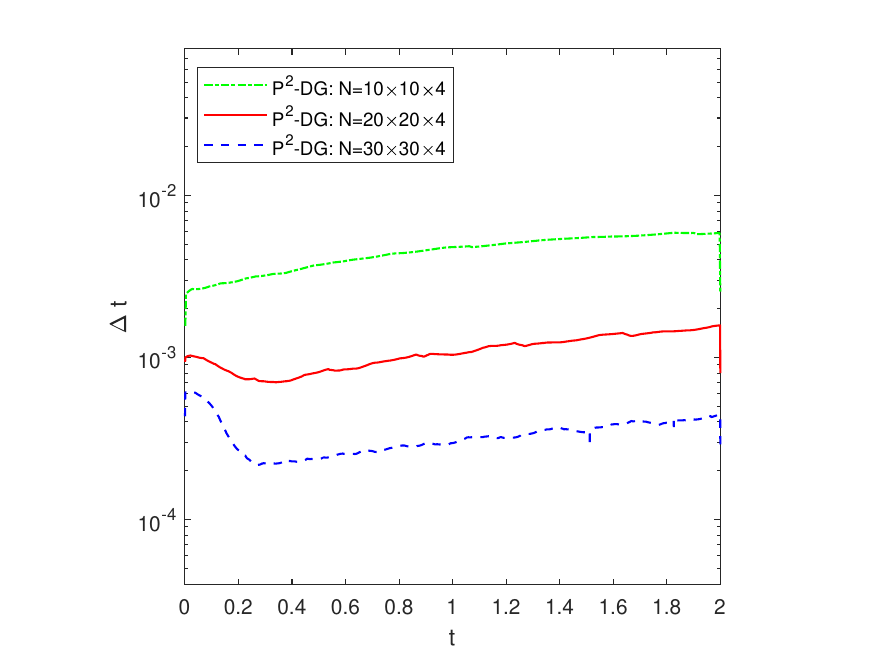}}
\subfigure[$P^3$-DG: $(\alpha_{CFL,e},\alpha_{LF,e})$]{
\includegraphics[width=0.31\textwidth,trim=25 0 40 10,clip]{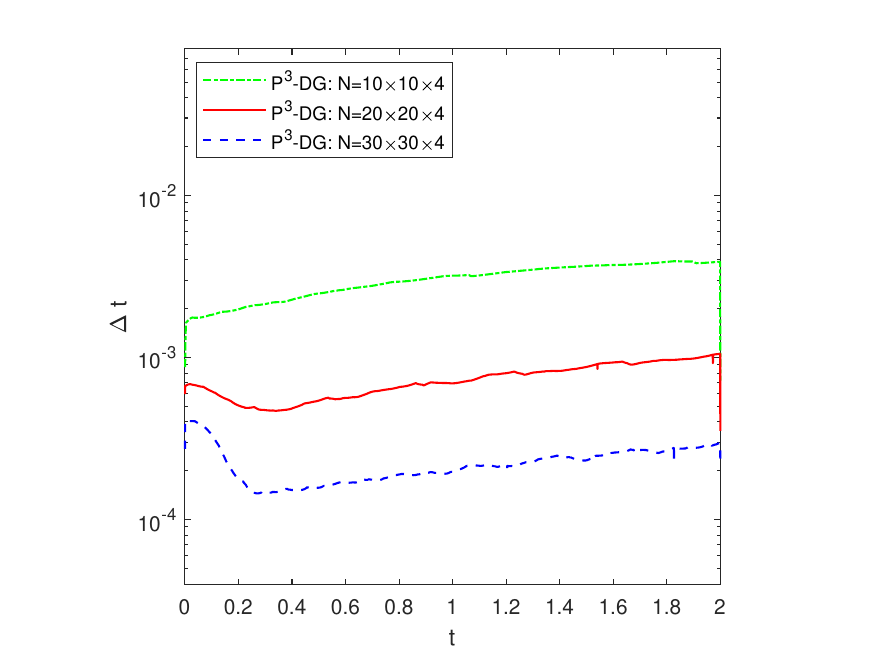}}
\subfigure[$P^1$-DG: ERK PI]{
\includegraphics[width=0.31\textwidth,trim=25 0 40 10,clip]{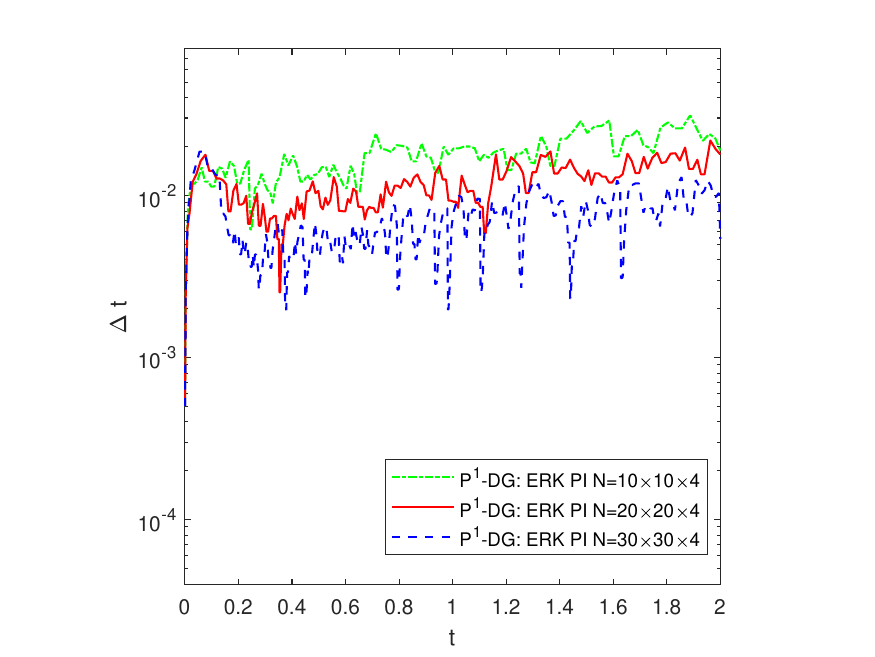}}
\subfigure[$P^2$-DG: ERK PI]{
\includegraphics[width=0.31\textwidth,trim=25 0 40 10,clip]{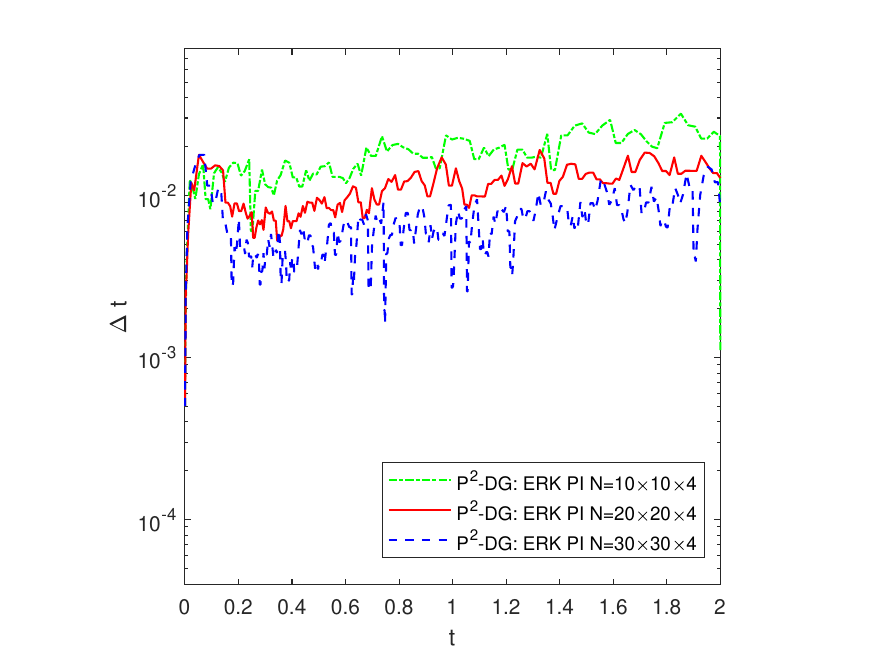}}
\subfigure[$P^3$-DG: ERK PI]{
\includegraphics[width=0.31\textwidth,trim=25 0 40 10,clip]{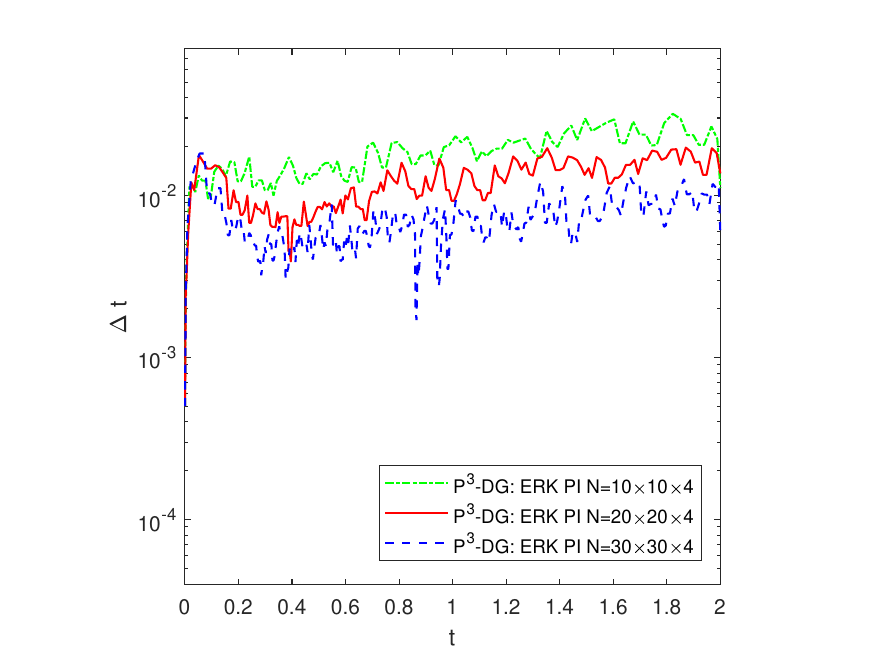}}
\caption{Example \ref{Burgers-2d}.
The time step-size results from ERK PI and CFL ($\alpha_{CFL,e}$, $\alpha_{LF,e}$) strategies for
the $P^k$-DG method ($k=1,2,3$) and the moving meshes of
$N=10\times10\times4,~20\times20\times4$, and $30\times30\times4$.}
\label{Fig:Burgers-2d-Pk-dtN}
\end{figure}

\begin{figure}[H]
\centering
\includegraphics[width=0.4\textwidth,trim=25 0 40 0,clip]{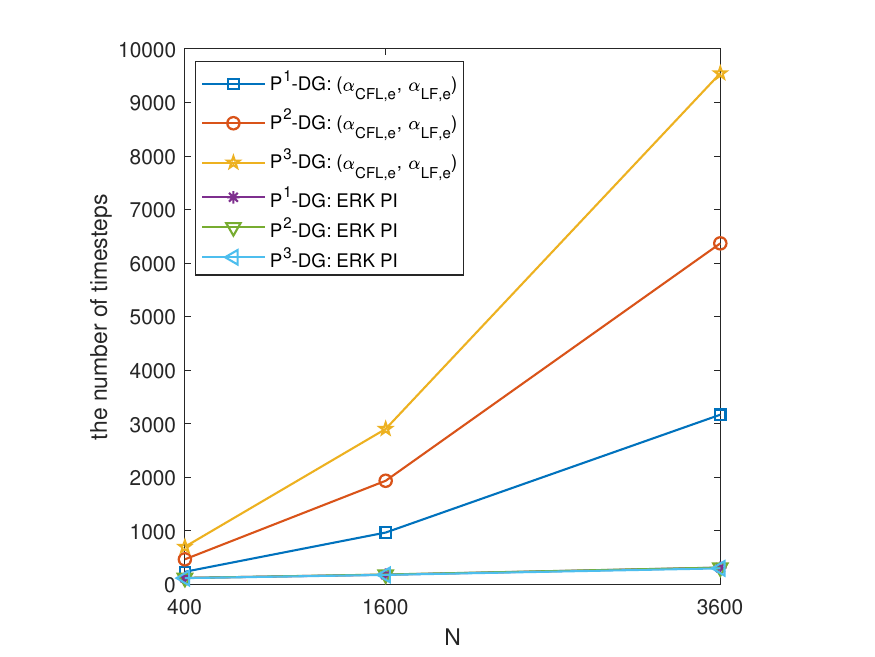}
\caption{Example \ref{Burgers-2d}.
The number of time step-sizes results from ERK PI and CFL ($\alpha_{CFL,e}$, $\alpha_{LF,e}$) strategies for
the $P^k$-DG method ($k=1,2,3$) and the moving meshes of
$N=10\times10\times4,~20\times20\times4$, and $30\times30\times4$.}
\label{Fig:Burgers-2d-Pk-Nsteps}
\end{figure}

\begin{example}\label{RP-2d}
(Riemann problem for the 2D Euler equations)
\end{example}
We consider a two-dimensional Riemann problem of the Euler equations
\begin{equation}\label{Euler-Eq-2d}
\frac{\partial}{\partial t} \begin{pmatrix} \rho \\ \rho u \\ \rho v \\ E \end{pmatrix}
+ \frac{\partial}{\partial x} \begin{pmatrix} \rho u \\ \rho u^2 + P \\ \rho u v \\ u (E + P) \end{pmatrix}
 + \frac{\partial}{\partial y} \begin{pmatrix} \rho v \\ \rho u v \\ \rho v^2 + P  \\ u (E + P) \end{pmatrix} = 0.
\end{equation}
The computational domain is taken as $[0,1]\times[0,1]$, and the initial conditions are
\begin{align}\label{RP1-initial}
(\rho,u,v,P)(x,y,0) =
\begin{cases}
(1.1,~0,~0,~1.1), \quad & x\geq0.5,~y\geq0.5\\
(0.5065,~0.8939,~0,~0.35), \quad & x<0.5,~y\geq0.5\\
(1.1,~0.8939,~0.8939,~1.1), \quad & x<0.5,~y<0.5\\
(0.5065,~0,~0.8939,~0.35), \quad & x\geq0.5,~y<0.5.\\
\end{cases}
\end{align}
The energy density $E$ and the pressure $P$ are related by the equation of the state
$E = P/(\gamma -1 ) + \rho (u^2 + v^2) /2$ with $\gamma = 1.4$.
The problem contains complicated interactions between four
initial shocks. This problem has been widely used as a benchmark test for
shock capturing methods due to the challenge in resolving
the complicated flow features that emerge from those interactions.
The final time for the computation is taken as $T = 0.25$.

The mesh and density at the final time, and time step-size obtained with the moving mesh $P^k$-DG method ($k=1$, 2, 3) and $N=50\times50\times4$ are shown in Fig.~\ref{Fig:RP1-2d-Pk}.
Like the previous examples, the computation is stable and the mesh points are concentrated correctly
around the shocks. Moreover, $\Delta t$ associated with ($\alpha_{CFL,e}$, $\alpha_{LF,e}$) is
slightly larger than that associated with ($\alpha_{CFL,h}$, $\alpha_{LF,e}$). However,
unlike the previous examples, $\Delta t$ has large oscillations before $t < 0.1$ for both cases.
\begin{figure}[H]
\centering
\subfigure[$P^1$-DG: mesh]{
\includegraphics[width=0.31\textwidth,trim=25 0 40 10,clip]
{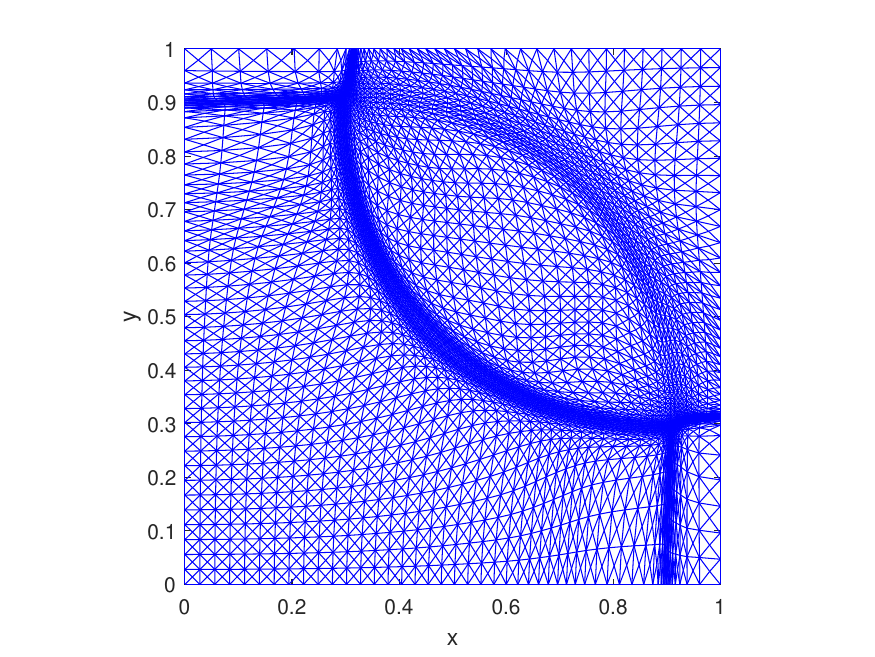}}
\subfigure[$P^2$-DG: mesh]{
\includegraphics[width=0.31\textwidth,trim=25 0 40 10,clip]
{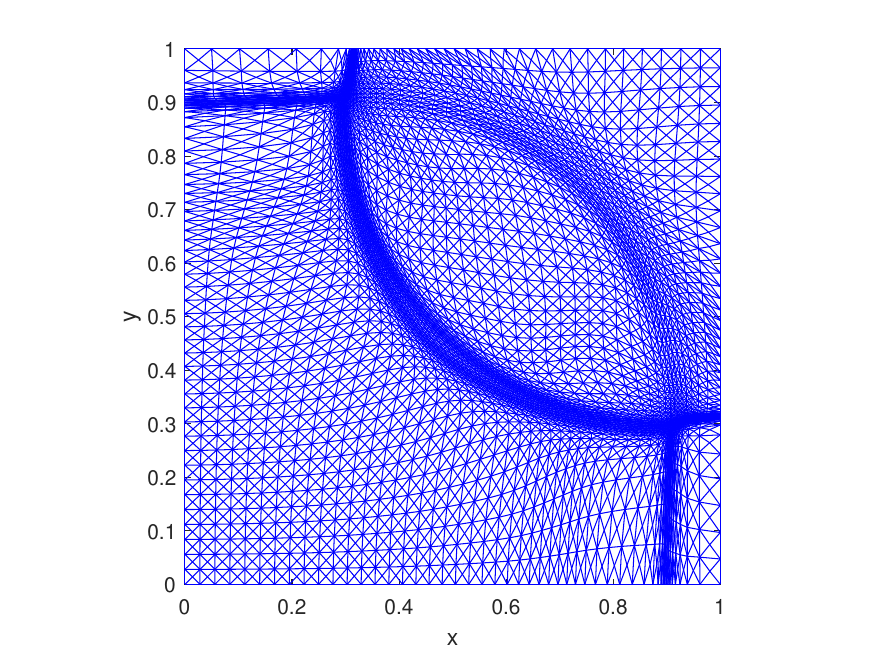}}
\subfigure[$P^3$-DG: mesh]{
\includegraphics[width=0.31\textwidth,trim=25 0 40 10,clip]
{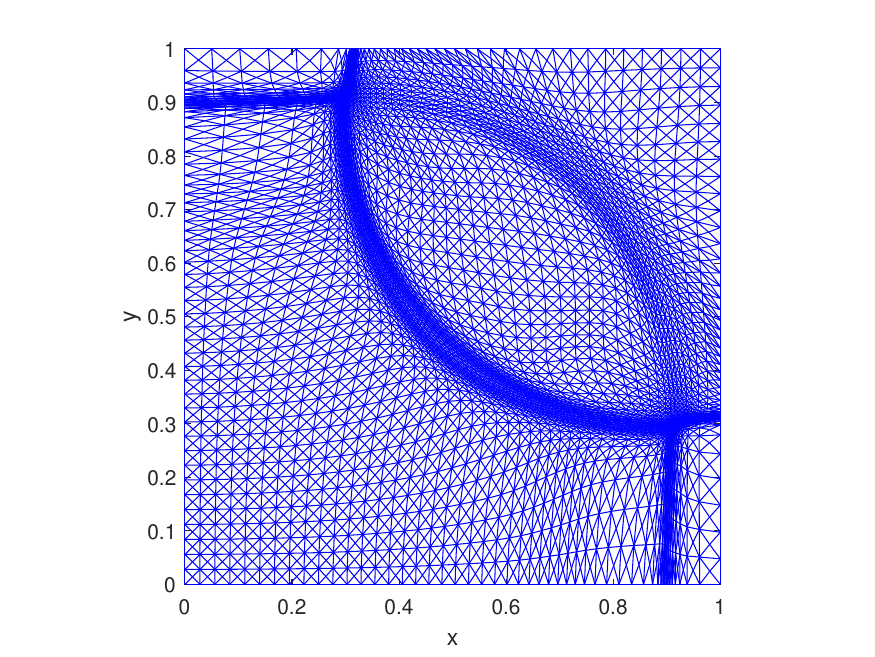}}
\subfigure[$P^1$-DG: density $\rho$]{
\includegraphics[width=0.31\textwidth,trim=25 0 40 10,clip]
{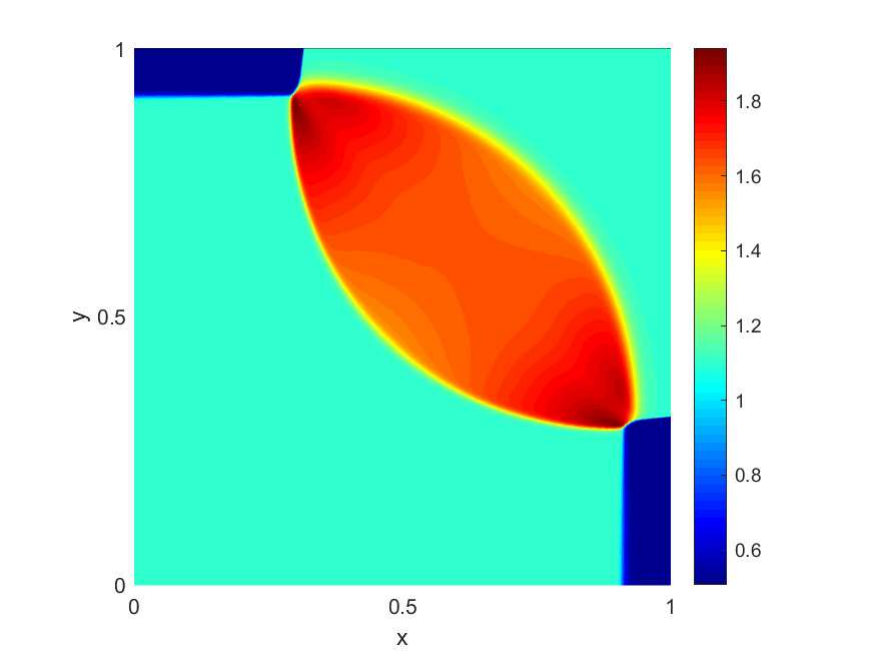}}
\subfigure[$P^2$-DG: density $\rho$]{
\includegraphics[width=0.31\textwidth,trim=25 0 40 10,clip]
{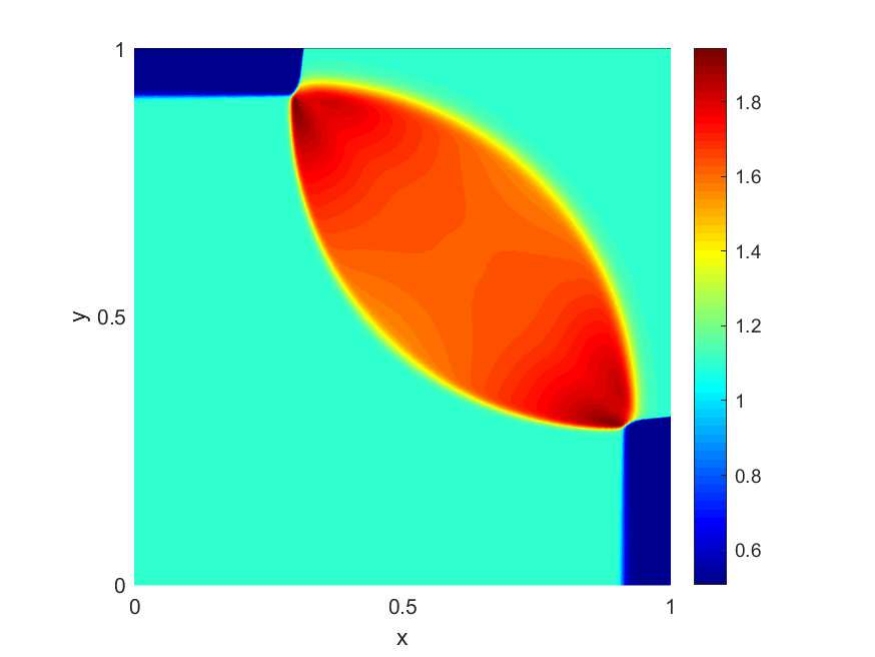}}
\subfigure[$P^3$-DG: density $\rho$]{
\includegraphics[width=0.31\textwidth,trim=25 0 40 10,clip]
{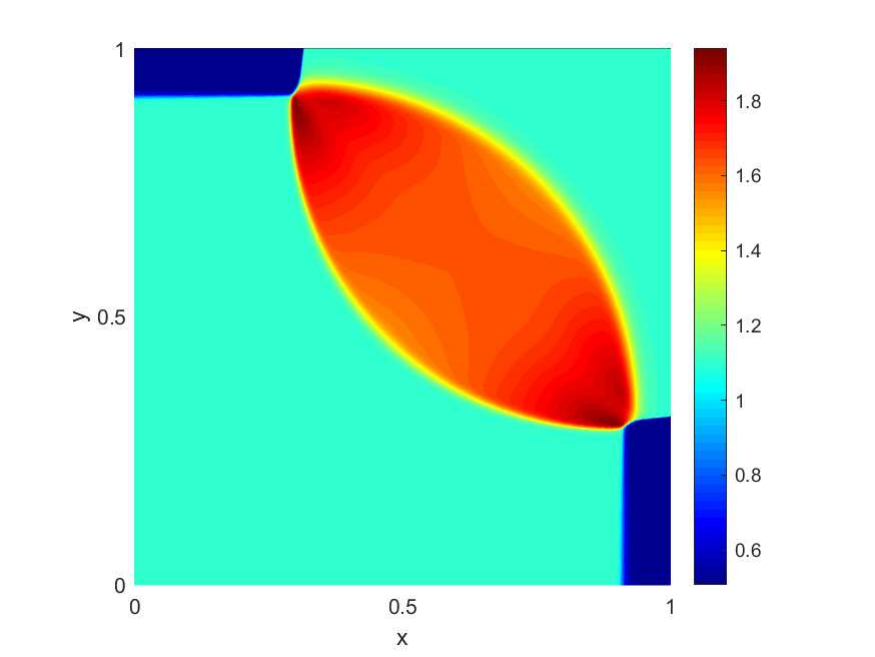}}
\subfigure[$P^1$-DG: $\Delta t$]{
\includegraphics[width=0.31\textwidth,trim=25 0 40 10,clip]
{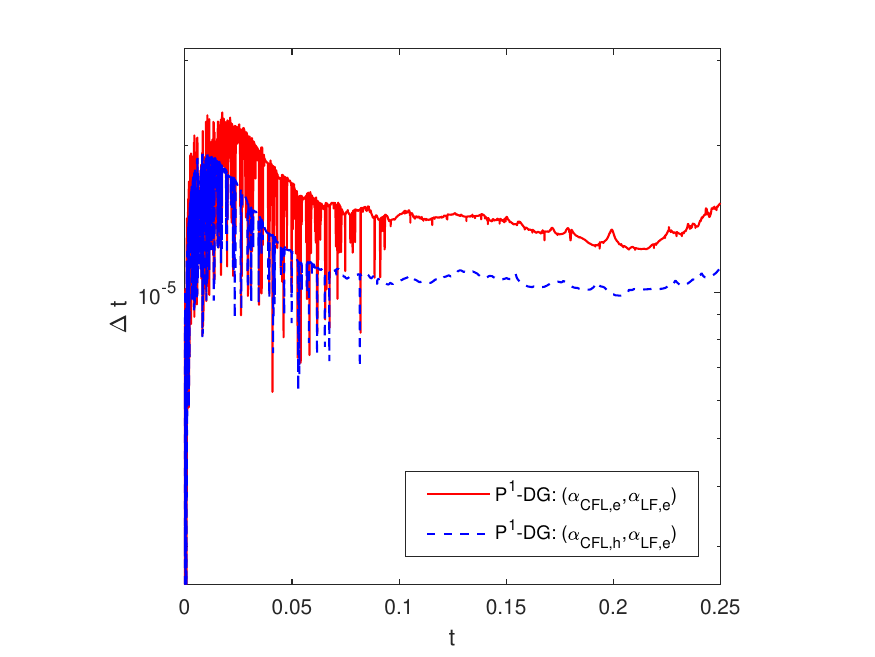}}
\subfigure[$P^2$-DG: $\Delta t$]{
\includegraphics[width=0.31\textwidth,trim=25 0 40 10,clip]
{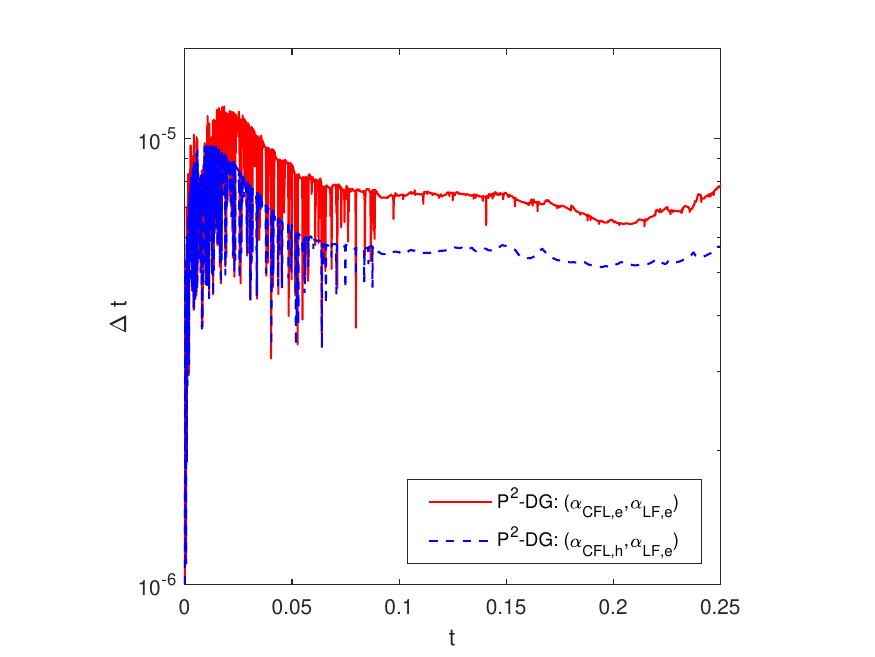}}
\subfigure[$P^3$-DG: $\Delta t$]{
\includegraphics[width=0.31\textwidth,trim=25 0 40 10,clip]
{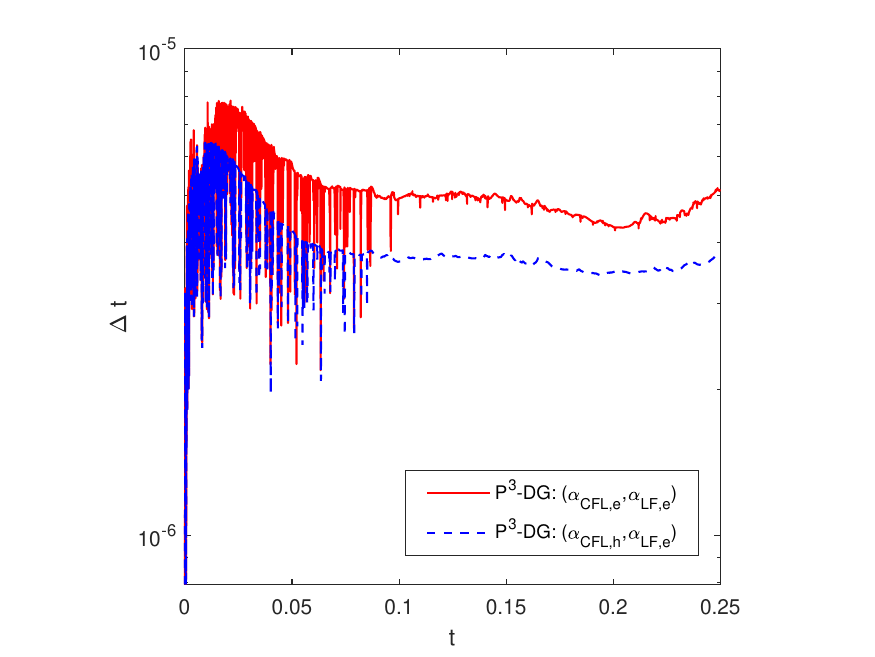}}
\caption{Example \ref{RP-2d}. The mesh, solution, and time step-size are obtained with the $P^k$-DG method ($k=1,2,3$) and a moving mesh of $N=50\times50\times4$.}
\label{Fig:RP1-2d-Pk}
\end{figure}

\begin{example}\label{Shu-2d}
(Isentropic vortex convection problem for the 2D Euler equations)
\end{example}
We consider the isentropic vortex problem of Shu \cite{Shu-1998} for the two-dimensional compressible Euler equations \eqref{Euler-Eq-2d}. The mean flow is $\rho = P = u=v=1$. We add to this mean flow an isentropic vortex perturbations
centered at $(x_0,y_0)$ in $(u,v)$ and the temperature ($T = P/\rho$) and no perturbation in the entropy $S = P\rho^{-\gamma}$, i.e.,
\begin{equation}
\begin{split}
(\delta u,\delta v) = \frac{\varepsilon}{2\pi}e^{0.5(1-r^2)}(-\bar{y},\bar{x}),\quad
\delta T = -\frac{(\gamma-1)\varepsilon^2}{8\gamma\pi^2}e^{1-r^2},
\end{split}
\end{equation}
where $(\bar{x},\bar{y})= (x-x_0, y-y_0)$,~$r^2 = x^2 + y^2$, and the vortex strength $\varepsilon =5$.
This means that the initial conditions are
\begin{align*}\label{Shu-initial}
&\rho(x,y,0) = \Big(1-\frac{(\gamma-1)\varepsilon^2}{8\gamma\pi^2}e^{1-r^2}\Big)^{\frac{1}{\gamma-1}},\quad P(x,y,0) = \rho^\gamma(x,y,0)\\
&u(x,y,0) = 1- \frac{\varepsilon}{2\pi}e^{0.5(1-r^2)}\bar{y},\quad v(x,y,0) = 1+\frac{\varepsilon}{2\pi}e^{0.5(1-r^2)}\bar{x}.
\end{align*}
The computational domain is taken as $(0,10)\times (0,10)$ and $(x_0,y_0)=(5,5)$.
The periodic boundary conditions are used for all unknown variables.
The energy density $E$ and the pressure $P$ are related by the equation of the state
$E = P/(\gamma -1 ) + \rho (u^2 + v^2) /2$ with $\gamma = 1.4$.
The final time for the computation is taken as $T = 1$.

We use the density $\rho$ only to compute the metric tensor in the MMPDE method since
entropy $\mathcal{S} = \ln(P\rho^{-\gamma})$ is constant for this example.
The mesh ($N=50\times50\times4$), the contours of density and momenta at the final time
obtained with $P^3$-DG method and CFL ($\alpha_{CFL,h}$, $\alpha_{LF,e}$) step-size selection strategy
are shown in Fig.~\ref{Fig:Shu-2d-P3}.
For comparison purpose, we plot $\Delta t$ for three time step-size strategies, ERK PI, CFL ($\alpha_{CFL,e},\alpha_{LF,e}$),
and CFL ($\alpha_{CFL,h}$, $\alpha_{LF,e}$) in Fig.~\ref{Fig:Shu-2d-dt}.
Like previous examples, the computation is stable. Moreover, $\Delta t$ associated with ($\alpha_{CFL,e}$, $\alpha_{LF,e}$)
is slightly larger than that associated with ($\alpha_{CFL,h}$, $\alpha_{LF,e}$), and $\Delta t$ associated with ERK PI is
significantly larger than that associated with ($\alpha_{CFL,e}$, $\alpha_{LF,e}$).

\begin{figure}[H]
\centering
\subfigure[Mesh]{
\includegraphics[width=0.33\textwidth,trim=25 0 30 20,clip]{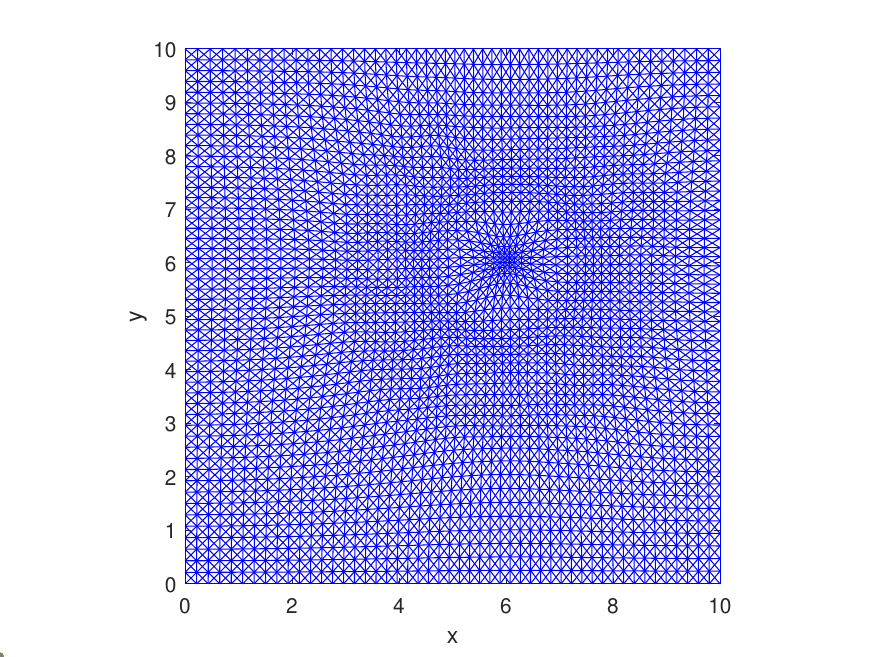}}
\subfigure[$\rho$]{
\includegraphics[width=0.33\textwidth,trim=25 0 30 20,clip]{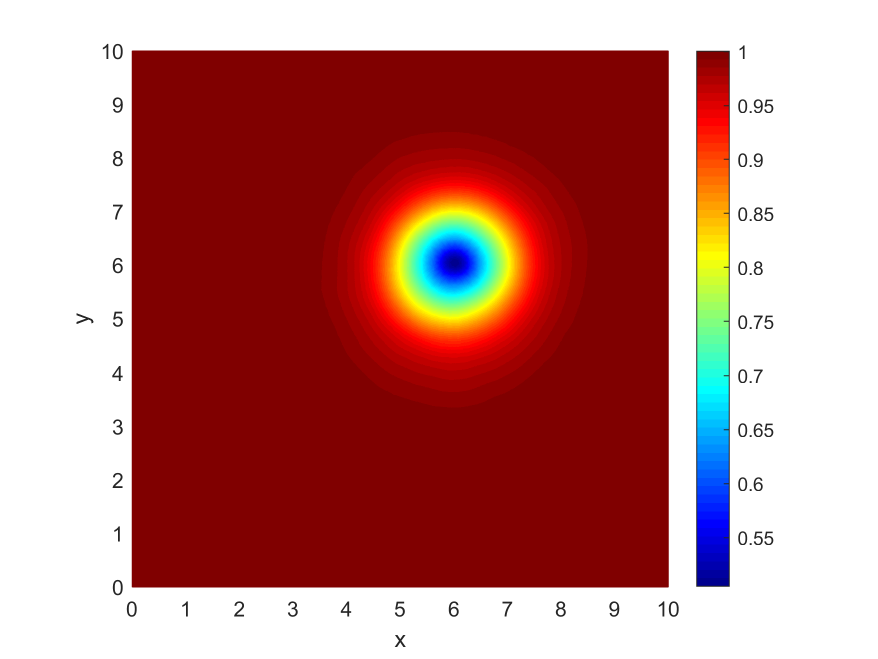}}
\subfigure[$\rho u$]{
\includegraphics[width=0.33\textwidth,trim=25 0 30 20,clip]{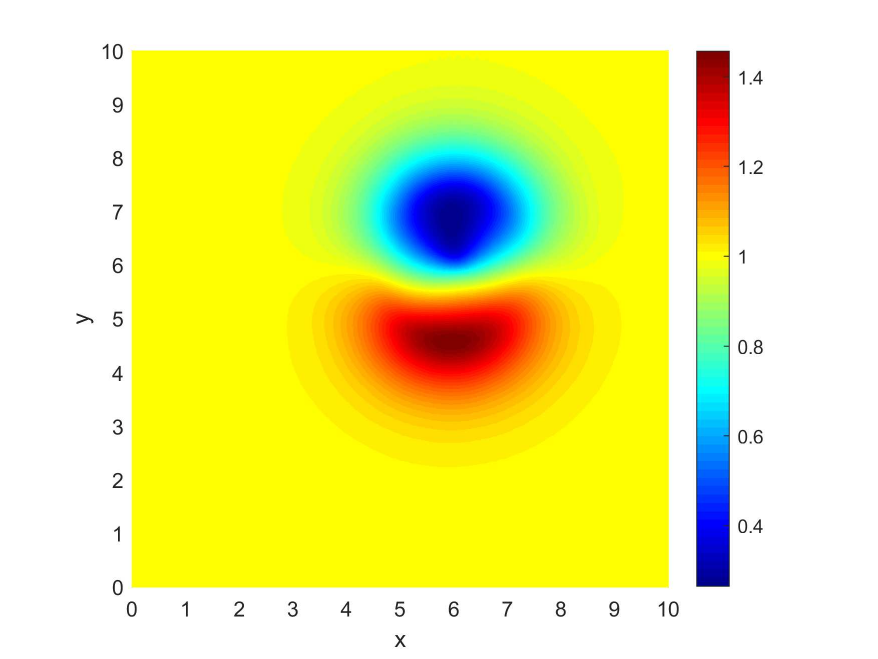}}
\subfigure[$\rho v$]{
\includegraphics[width=0.33\textwidth,trim=25 0 30 20,clip]{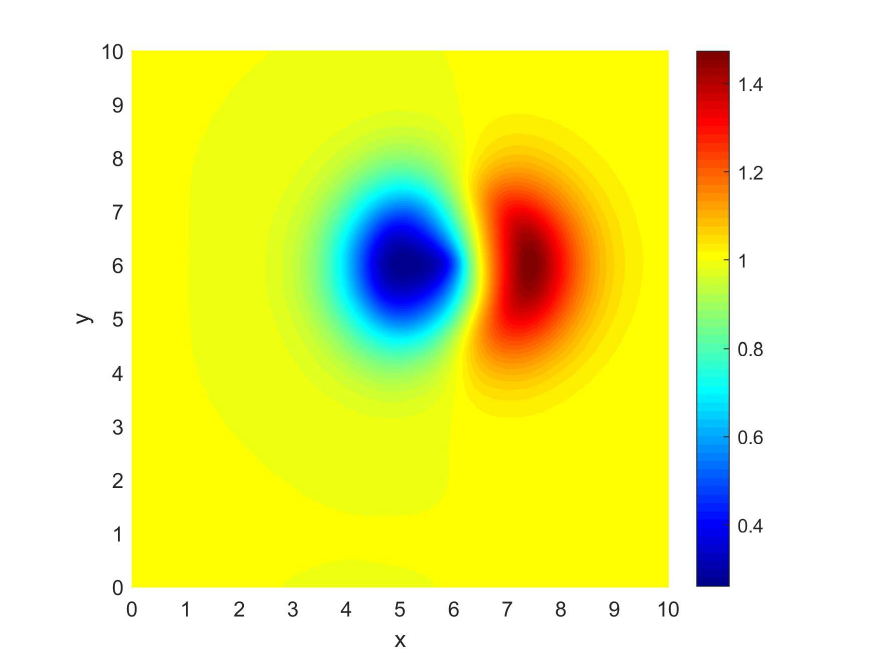}}
\caption{Example \ref{Shu-2d}. The mesh ($N=50\times50\times4$) and contours of the solutions at $t=1$
are obtained with the moving mesh $P^3$-DG method and CFL ($\alpha_{CFL,e}$, $\alpha_{LF,e}$)
time step-size selection strategy.}
\label{Fig:Shu-2d-P3}
\end{figure}

\begin{figure}[H]
\centering
\subfigure[$P^1$-DG]{
\includegraphics[width=0.31\textwidth,trim=25 0 40 10,clip]{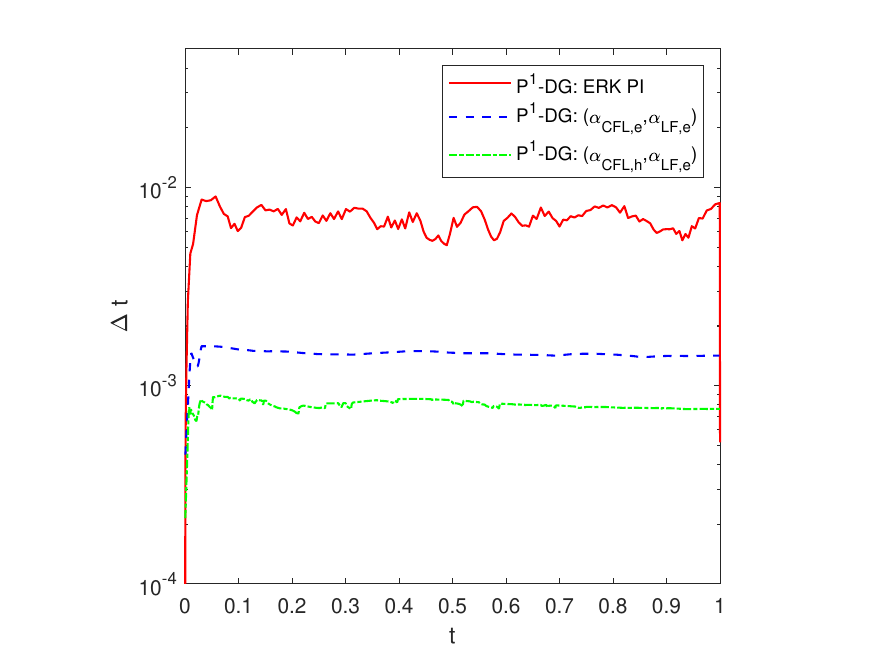}}
\subfigure[$P^2$-DG]{
\includegraphics[width=0.31\textwidth,trim=25 0 40 10,clip]{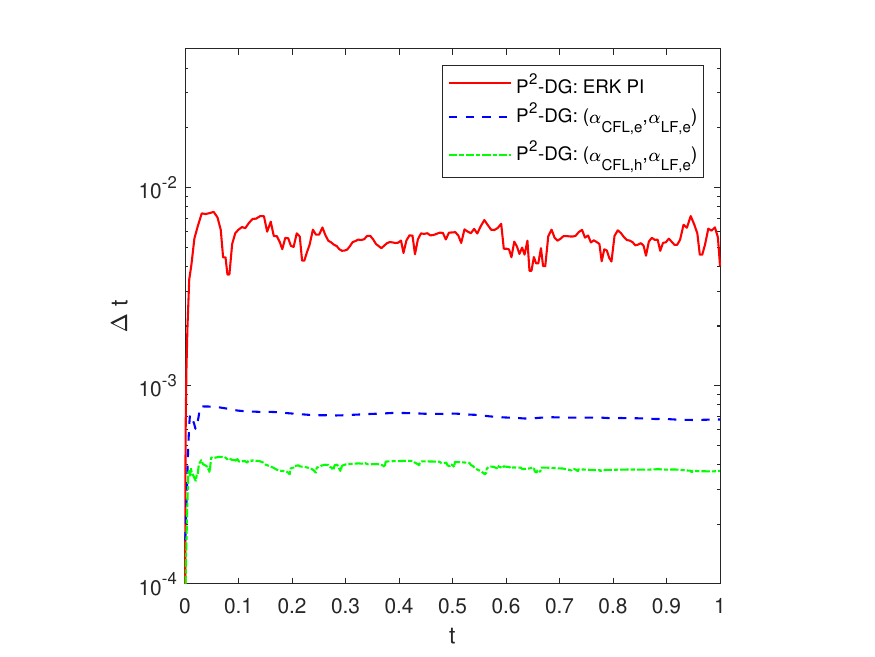}}
\subfigure[$P^3$-DG]{
\includegraphics[width=0.31\textwidth,trim=25 0 40 10,clip]{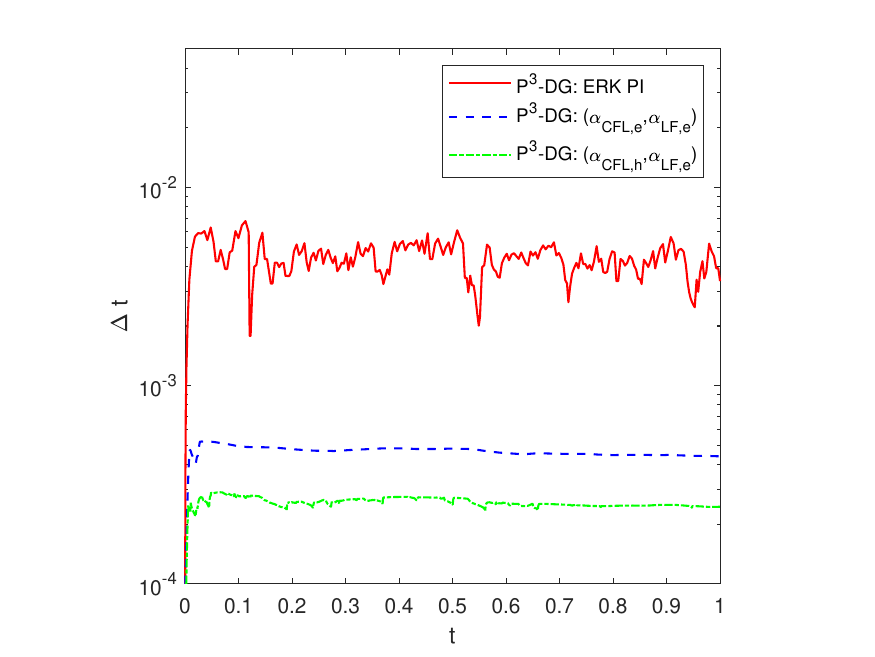}}
\caption{Example \ref{Shu-2d}. The time step-size $\Delta t$ results from with three selection strategies,
ERK PI, CFL ($\alpha_{CFL,e}$, $\alpha_{LF,e}$), and CFL ($\alpha_{CFL,h}$, $\alpha_{LF,e}$).
The $P^1$, $P^2$, and $P^3$-DG method with a moving mesh of $N=50\times50\times4$ is used.}
\label{Fig:Shu-2d-dt}
\end{figure}

\section{Conclusions}
\label{SEC:conclusions}

In the previous sections we have studied the stability of a DG solution of conservation laws on
adaptive moving meshes. Particularly, we have obtained a CFL condition (\ref{cfl-2}) for
moving mesh $P^0$-DG for linear scalar conservation laws. The condition shows that
the allowed maximum time step depends on the coupling between $\alpha$ on any edge
and the element height associated with the edge, where $\alpha$ (cf. \ref{alpha-p}) is the maximum
absolute value of the eigenvalue of the Jacobian matrix of the flux $(\bm{F} - U \dot{\bm{X}}) \cdot \bm{n}$ with respect to $U$ and the element height is the distance between the edge and the vertex opposite to the edge.
The condition justifies a known intuition that time step can be increased if the mesh velocity is chosen to minimize $(\bm{F} - U \dot{\bm{X}})$. On the other hand, if $\alpha$ does not change significantly
over the domain, it is expected that the allowed maximum time step is determined mostly
by the minimum element height of the mesh.

The stability analysis also shows that different choices of $\alpha$ can be used in the LF numerical flux and the CFL condition but a relation (\ref{alpha-0}) should be satisfied for $L^1$ stability.
Two common choices for $\alpha$ are $\alpha_e$ defined in (\ref{alpha-e}) and $\alpha_h$ defined in (\ref{alpha-h}).

Based on (\ref{cfl-2}) and CFL conditions for DG method on fixed meshes, we have proposed to choose the time step according to (\ref{cfl-5}) (with the same or a different choice of $\alpha$) for moving mesh $P^k$-DG ($k=1$, 2, 3) for general conservation laws.
This condition reduces to the one on fixed meshes when $\alpha$ is taken as $\alpha_h$.

Numerical examples have been presented with mesh adaptation by the MMPDE method for Burgers' and Euler equations in one and two dimensions.
Numerical results show that (\ref{cfl-5}) with $\alpha_{CFL} = \alpha_e$ or $\alpha_{CFL} = \alpha_h$ lead to stable computation.
Moreover, ($\alpha_{CFL,e}$, $\alpha_{LF,e}$) typically gives larger $\Delta t$ than ($\alpha_{CFL,h}$, $\alpha_{LF,e}$)
and ($\alpha_{CFL,h}$, $\alpha_{LF,h}$). All but one example also show that $\Delta t$
with ($\alpha_{CFL,e}$, $\alpha_{LF,e}$) has smaller oscillations.

In this work, we have also studied an error-based time step-size selection strategy with
the explicit SSP embedded Runge-Kutta pairs for DG computation of hyperbolic conservation laws on adaptive moving meshes.
Numerical examples show that the error-based strategy can lead to stable computation and result in larger time step-size
especially for higher-order DG than the CFL based selection strategy.

It should be pointed out that we have not considered moving mesh strategies such as Lagrangian-type methods to minimize
$\alpha$ and increase $\Delta t$ in the current work. Moreover, the error-based time step-size selection strategy seems
to result in larger time step-size than CFL condition. These are interesting topics that may deserve more studies in near future.

\vspace{20pt}

\noindent
{\bf Acknowledgment.} M. Zhang was partially supported by the Postdoctoral Science Foundation of China under grant 2022M710229 and
J. Qiu was partially supported by National Natural Science Foundation of China under grant 12071392.
The authors thank the anonymous referees for their valuable comments and suggestions that helped improve the quality
of the paper.



\begin{thebibliography}{00}

\bibitem{Bai94a}
M.~J. Baines,
Moving Finite Elements,
Oxford University Press, Oxford, 1994.

\bibitem{Baines-2011}
M.~J. Baines, M.~E. Hubbard, and P.~K. Jimack,
Velocity-based moving mesh methods for nonlinear partial differential equations,
{\em Comm. Comput. Phys.}, 10 (2011), 509-576.

\bibitem{Bey-Oden-1996}
K. Bey and J. Oden,
$hp$-version discontinuous Galerkin methods for hyperbolic conservation laws,
{\em Comput. Meth. Appl. Mech. Engrg.} 133 (1996), 259-286.

\bibitem{ode23-1989}
P. Bogacki and L. F. Shampine,
A 3(2) pair of Runge-Kutta formulas,
 {\em Appl. Math. Lett.}, 2 (1989), 321-325.

\bibitem{BHR09}
C.~J. Budd, W.~Huang, and R.~D. Russell,
Adaptivity with moving grids,
{\em Acta Numer.}, 18 (2009), 111-241.

\bibitem{DG-series0}
G. Chaventg and B. Cockburn,
The local projection $P^0$-$P^1$-discontinuous-Galerkin finite element method for scalar conservation laws,
{\em $M^2$AN}, 23 (1989), 565-592.

\bibitem{DG-series1}
B. Cockburn and C.-W. Shu,
The Runge-Kutta local projection $P^1$-discontinuous-Galerkin finite element method for scalar conservation laws,
{\em $M^2$AN}, 25 (1991), 337-361.

\bibitem{DG-series2}
B. Cockburn and C.-W. Shu,
TVB Runge-Kutta local projection discontinuous Galerkin finite element method for conservation laws II: General framework,
{\em Math. Comp.}, 52 (1989), 411-435.

\bibitem{DG-series3}
B. Cockburn, S.-Y. Lin, and C.-W. Shu,
TVB Runge-Kutta local projection discontinuous Galerkin finite element method for
conservation laws III: one dimensional systems,
{\em J. Comput. Phys.}, 84 (1989), 90-113.

\bibitem{DG-series4}
B. Cockburn, S. Hou, and C.-W. Shu,
 The Runge-Kutta local projection discontinuous Galerkin method for conservation laws IV: multidimensional systems,
 {\em Math. Comp.}, 54 (1990), 545-581.

\bibitem{DG-series5}
B. Cockburn and C.-W. Shu,
 The Runge-Kutta discontinuous Galerkin method for conservation laws V: multidimensional systems,
{\em J. Comput. Phys.}, 141 (1998), 199-224.

\bibitem{DG-review}
B. Cockburn and C.-W. Shu,
Runge-Kutta discontinuous Galerkin methods for convection dominated problems,
{\em J. Sci. Comput.}, 16 (2001), 173-261.

\bibitem{SSP-time}
S. Gottlieb, C.-W. Shu, and E. Tadmor,
Strong stability-preserving high-order time discretization methods,
{\em SIAM Review}, 43 (2001), 89-112.

\bibitem{SSPERK-2018}
S. Conde, I. Fekete, and J. N. Shadid,
Embedded error estimation and adaptive step-size control for optimal explicit strong stability preserving Runge-Kutta methods,
arxiv: 1806.08693.

\bibitem{Dumbser-Boscheri-Lagrangian-2013}
M. Dumbser and W. Boscheri, High-order unstructured Lagrangian one-step WENO finite volume schemes for non-conservative hyperbolic systems: applications to compressible multi-phase flows,
{\em Comput. \& Fluids}, 86 (2013), 405-432.

\bibitem{HW1991}
E. Hairer and G. Wanner,
Solving Ordinary Differential Equations II: Stiff and Differential-Algebraic Problems,
Springer-Verlag, Berlin, 1991.

\bibitem{Hirt-ALE1-1971}
C. W. Hirt,
An arbitrary Lagrangian-Eulerian computing technique,
Proceedings of the Second International Conference on Numerical Methods in Fluid Dynamics, Volume 8, 1971.

\bibitem{Hirt-ALE2-1974}
C. W. Hirt, A. A. Amsden, and J. Cook,
An arbitrary Lagrangian-Eulerian computing method for all flow speeds,
{\em J. Comput. Phys.}, 14 (1974), 227-253.

\bibitem{Huang-Ren-Russell-1994a}
W.~Huang, Y.~Ren, and R.~Russell,
Moving mesh partial differential equations (MMPDEs) based upon the equidistribution principle,
{\em SIAM J. Numer. Anal.}, 31 (1994), 709-730.

\bibitem{Huang-Sun-2003JCP}
W.~Huang and W.~Sun,
Variational mesh adaptation II: error estimates and monitor functions,
{\em J. Comput. Phys.}, 184 (2003), 619-648.

\bibitem{Huang-Russell-2011}
W.~Huang and R.~Russell,
Adaptive Moving Mesh Methods,
Springer, New York, Applied Mathematical Sciences Series, Vol. 174 (2011).

\bibitem{Huang-Kamenski-2015JCP}
W.~Huang and L.~Kamenski,
A geometric discretization and a simple implementation for variational mesh generation and adaptation,
{\em J. Comput. Phys.}, 301 (2015), 322-337.

\bibitem{Huang-Kamenski-2018MC}
W. Huang and L. Kamenski,
On the mesh nonsingularity of the moving mesh PDE method,
{\em Math. Comp.}, 87 (2018), 1887-1911.

\bibitem{Kucharik-Shashkov-ALE-2014}
M. Kucharik and M. Shashkov, Conservative multi-material remap for staggered
multi-material arbitrary Lagrangian-Eulerian methods,
{\em J. Comput. Phys.}, 258 (2014), 268-304.

\bibitem{LeSaint-Raviart-1974}
P. LeSaint and P.-A. Raviart,
On a finite element method for solving the neutron transport equation,
in: C. de Boor (Ed.), Mathematical Aspects of Finite Elements in Partial Differential Equations,
Academic Press, New York, 1974, pages 89-123.

\bibitem{Li-Tang-2006}
R. Li and T. Tang,
Moving mesh discontinuous Galerkin method for hyperbolic conservation laws,
{\em J. Sci. Comput.} 27 (2006), 347-363.

\bibitem{Luo-Huang-Qiu-2019JCP}
D.~Luo, W.~Huang, and J.~Qiu,
A quasi-Lagrange moving mesh discontinuous Galerkin method for hyperbolic conservation laws,
{\em J. Comput. Phys.}, 396 (2019), 544-578.

\bibitem{Luo-LHQC-2021}
D. Luo, S. Li, W. Huang, J. Qiu, and Y. Chen,
A quasi-conservative DG-ALE method for multi-component flows using the non-oscillatory kinetic flux,
(arXiv:2101.04897)

\bibitem{Mackenzie-Nicola-2007}
J. Mackenzie and A. Nicola,
A discontinuous Galerkin moving mesh method for Hamilton-Jacobi equations,
{\em SIAM J. Sci. Comput.} 29 (2007), 2258-2282.

\bibitem{Morgan-etal-Lagrangian-2014}
N. R. Morgan, K. N. Lipnikov, D.E. Burton, and M. A. Kenamond,
A Lagrangian staggered grid Godunov-like approach for hydrodynamics,
{\em J. Comput. Phys.}, 259 (2014), 568-597.

\bibitem{Reed-Hill-1973}
W. H. Reed and T. R. Hill,
Triangular mesh methods for neutron transport equation, Los Alamos Scientific Laboratory Report LA-UR-73-479 (1973).

\bibitem{Shu-1998}
C.-W. Shu, Essentially non-oscillatory and weighted essentially non-oscillatory schemes for hyperbolic conservation laws, in Advanced Numerical Approximation of Nonlinear Hyperbolic Equations, Lecture Notes in Mathematics, volume 1697, Springer, 1998, 325-432.

\bibitem{Shu-1998-TVDtime}
C.-W. Shu,
Total-variation-diminishing time discretizations,
{\em SIAM J. Sci. Stat. Comput.}, 9 (1988), 1073-1084.

\bibitem{Tan05}
T.~Tang,
Moving mesh methods for computational fluid dynamics flow and transport,
{\em Recent Advances in Adaptive Computation (Hangzhou, 2004)},
Volume 383 of {\em AMS Contemporary Mathematics}, pages 141-173. Amer. Math.
Soc., Providence, RI, 2005.

\bibitem{Uzunca-2017}
M. Uzunca, B. Karas\"{o}zen, and T. K\"{u}\c{c}\"{u}kseyhan,
Moving mesh discontinuous Galerkin methods for PDEs with traveling wave,
{\em Appl. Math. Comput.} 292 (2017), 9-18.


\bibitem{Vilar-Maire-Abgrall-Lagrangian-2014}
F. Vilar, P.-H. Maire, and R. Abgrall,
A discontinuous Galerkin discretization for solving
the two-dimensional gas dynamics equations written under
total Lagrangian formulation on general unstructured grids,
{\em J. Comput. Phys.}, 276 (2014), 188-234


\bibitem{Wang-Luo-Shashkov-Lagrangian-2020}
C. Wang, H. Luo, and M. Shashkov,
A reconstructed discontinuous Galerkin method for compressible flows in Lagrangian formulation,
{\em Comput. \& Fluids},  202 (2020), 104522.

\bibitem{Zhang-Cheng-Huang-Qiu-2020CiCP}
M.~Zhang, J.~Cheng, W.~Huang, and J.~Qiu,
An adaptive moving mesh discontinuous Galerkin method for the radiative transfer equation,
{\em Commun. Comput. Phys.}, 27 (2020), 1140-1173.

\bibitem{Zhang-Huang-Qiu-2020SISC}
M.~Zhang, W.~Huang, and J.~Qiu,
High-order conservative positivity-preserving DG-interpolation for deforming meshes
and application to moving mesh DG simulation of radiative transfer,
{\em SIAM J. Sci. Comput.}, 42 (2020), A3109-A3135.

\bibitem{Zhang-Huang-Qiu-2021JSC}
M. Zhang, W. Huang, and J. Qiu,
A high-order well-balanced positivity-preserving moving mesh DG method for the shallow water equations with non-flat bottom topography,
{\em J. Sci. Comput.}, 87 (2021), No. 88.

\bibitem{Zhang-Huang-Qiu-2022CICP}
M. Zhang, W. Huang, and J. Qiu,
A well-balanced positivity-preserving quasi-Lagrange moving mesh DG method for the shallow water equations, {\em Commun. Comput. Phys.}, 31 (2022), 94-130.

\bibitem{Zhang-Xia-Xu-2021JSC}
W. Zhang, Y. Xia and Y. Xu,
Positivity-preserving well-balanced arbitrary Lagrangian-Eulerian discontinuous Galerkin methods for the shallow water equations,
{\em J. Sci. Comput.}, 88 (2021) No. 57.


\bibitem{Zhang-Xia-Shu-2012JSC}
X. Zhang, Y. Xia, and C.-W. Shu,
Maximum-principle-satisfying and positivity-preserving high order discontinuous Galerkin schemes
for conservation laws on triangular meshes,
{\em J. Sci. Comput.}, 50 (2012), 29-62.

\end{thebibliography}
\end{document}